\documentclass[11pt]{article}
\usepackage{amsmath,amssymb,fullpage}

\usepackage[T1]{fontenc}
\usepackage{amsmath}
\usepackage[all,dvips]{xy}

\usepackage{tikz}
\usepackage{tikz,fullpage}
\usetikzlibrary{matrix}

\usepackage{enumerate}
\usepackage{mathtools}

\usepackage{graphicx}
\usepackage{color}
\definecolor{vert}{rgb}{0.02,0.4,0.10}
\usepackage{hyperref}
\hypersetup{
    colorlinks=true,                % colorise les liens
    breaklinks=true,                % permet les retours à la ligne pour les liens trop longs
    linkcolor= vert,                % couleur des liens internes aux documents (index, figures, tableaux, equations,...)
    citecolor= blue                % couleur des liens vers les references bibliographiques
    }

\usepackage{graphicx}
\graphicspath{ {images/} }
\usepackage{pgf,tikz,pgfplots}
\pgfplotsset{compat=1.14}
\usepackage{mathrsfs}
\usetikzlibrary{arrows}

\usepackage{makeidx}
\usepackage{pstricks}
\providecommand{\otherindexspace}[1]{}

\usepackage{amsthm}

\newtheorem{theorem}{Theorem}[section]

\newtheorem{lemma}[theorem]{Lemma}
\newtheorem{proposition}[theorem]{Proposition}
\newtheorem{remark}[theorem]{Remark}

\newtheorem{corollary}[theorem]{Corollary}

\numberwithin{equation}{section}

\def\D{\Delta}

\def\vp{\varepsilon}

\def \R{\mathbb {R}}
\def \N{\mathbb{N}}
\def \Z{\mathbb{Z}}

\def \D{\mathbb{D}}

\def\E{\mathbb{E}}
\def\P{\mathbb{P}}
\def\lb{[\![}
\def\rb{]\!]}

\usepackage{fancyhdr}
\usepackage{graphics}
\usepackage{graphicx}

\DeclareGraphicsExtensions{.eps,.bmp,.jpg,.pdf,.mps,.png,.gif}

\makeatletter
\def\titre{\@title}
\makeatother

\title{On large $3/2$-stable maps}

\author{Emmanuel Kammerer}
\date{\today}

\begin{document}

\maketitle

\begin{abstract}
We discuss asymptotics of large Boltzmann random planar maps such that every vertex of degree $k$ has weight of order $k^{-2}$. Infinite maps of that kind were studied by Budd, Curien and Marzouk. These maps correspond to the dual of the discrete $\alpha$-stable maps studied by Le Gall and Miermont for $\alpha=3/2$ and to the gaskets of critical $O(2)$-decorated random planar maps. We compute the asymptotics of the graph and first passage percolation distances between two uniform vertices, which are respectively equivalent in probability to $(\log \ell)^2/\pi^2$ and $2(\log \ell)/(\pi^2 p_{\bf q})$ when the perimeter of the map $\ell$ goes to $\infty$, where $p_{\bf q}$ is a constant depending on the model. We also show that the diameter is of the same order as those distances for both metrics and establish the absence of scaling limits in the sense of Gromov-Prokhorov or Gromov-Hausdorff for lack of tightness. To study the peeling exploration of these maps, we prove local limit and scaling limit theorems for a class of random walks with heavy tails conditioned to remain positive until they die at $-\ell$ towards processes that we call stable Lévy processes conditioned to stay positive until they jump and die at $-1$.
\end{abstract}

\begin{figure}[!ht]
\centering
\includegraphics[trim=2.1cm 2cm 2.1cm 2.3cm,width=0.75\linewidth]{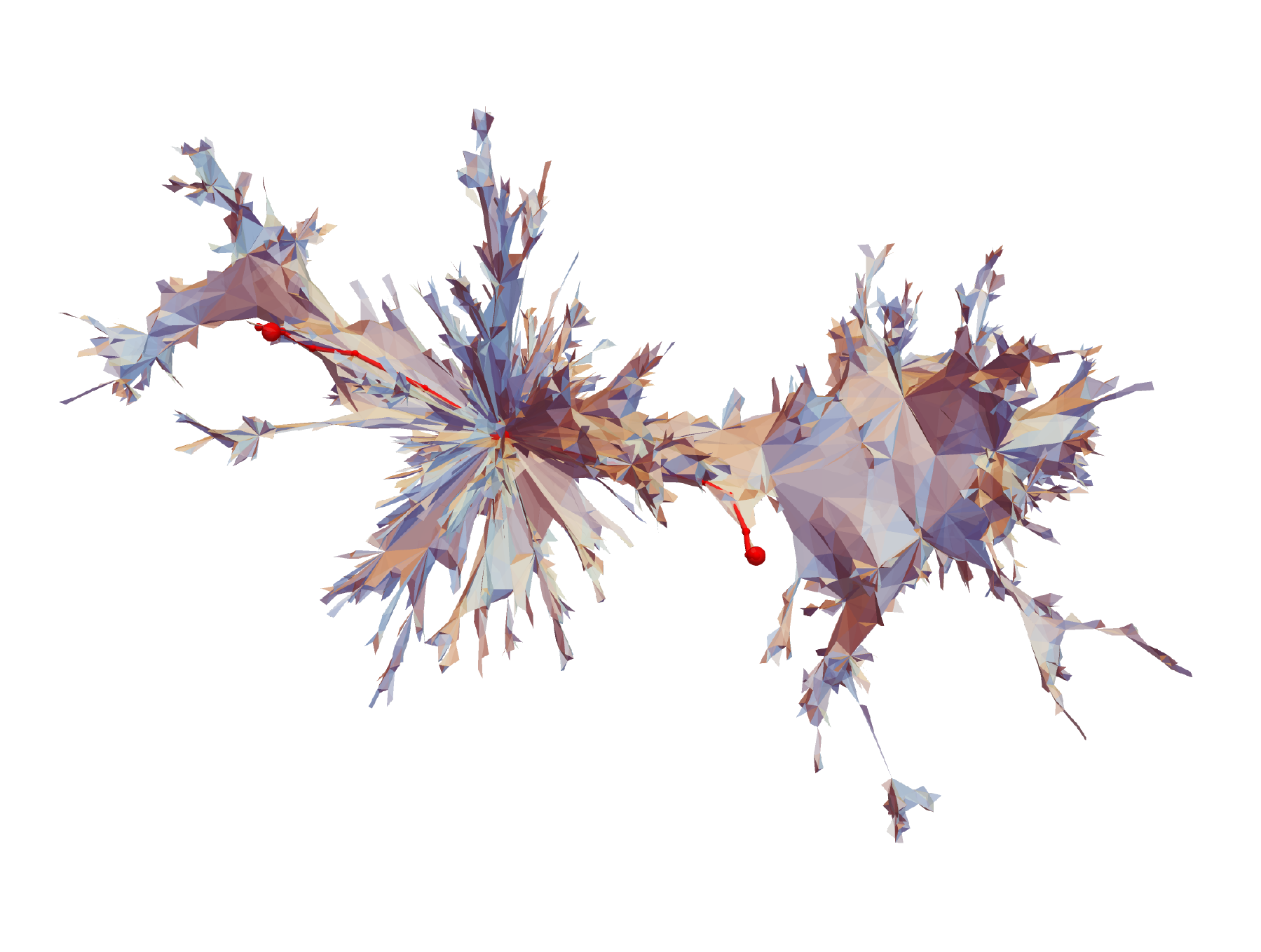}%
\caption{Simulation of a $3/2$-stable map $\mathfrak{M}^{(500,\dagger)}$ of perimeter $1000$. The two large red spheres are two uniform random vertices of $\mathfrak{M}^{(500,\dagger)}$ and the red path is the shortest path between them.}
\label{fig:ssimus_intro}
\end{figure}

\tableofcontents

\section{Introduction}

\subsection{Boltzmann $3/2$-stable planar maps}

Random planar maps have been studied very actively, motivated partly by their connections with two dimensional Liouville quantum gravity. In this work, we study the geometry of large random $3/2$-stable Boltzmann planar maps. A planar map $\mathfrak{m}$ is a connected planar graph embedded in the sphere seen up to orientation-preserving homeomorphism and equipped with a distinguished oriented edge $\vec{e}_r$ called the root edge. The face $f_r$ adjacent to the right of the root edge is called the root face. The degree of $f_r$ is also called the perimeter of $\mathfrak{m}$. We also require the map $\mathfrak{m}$ to be bipartite, i.e. that all its faces have even degree, inasmuch as the enumeration of bipartite maps is particularly simple. We write $\mathcal{M}$ for the set of finite bipartite rooted planar maps.

Random planar maps can be picked uniformly at random among maps of a given size. A more general way to pick a random map $\mathfrak{m} \in \mathcal{M}$ is to assign Boltzmann weights to them, following \cite{MM07}. Let ${\bf q} = (q_k)_{k\ge 1}$ be a non-zero sequence of non-negative numbers called the sequence of weights. The weight of a map $\mathfrak{m} \in \mathcal{M}$ is defined as
$$
w_{\bf q} (\mathfrak{m}) = \prod_{f \in \mathrm{Faces}(\mathfrak{m})\setminus \left\{ f_r\right\}} q_{\mathrm{deg}(f)/2}.
$$
For all $\ell \ge 1$, let $\mathcal{M}^{(\ell)}\subset \mathcal{M}$ be the set of all (finite) maps with perimeter $2\ell$. We introduce the partition function
$$
W^{(\ell)} = \sum_{\mathfrak{m} \in\mathcal{M}^{(\ell)}} w_{\bf q}(\mathfrak{m}).
$$
When $W^{(1)}< \infty$, we say that $\bf q$ is admissible. 
In what follows, we also require the weight sequence to be critical of type $a\in (3/2,5/2]$, which is equivalent by Proposition 5.10 of \cite{StFlour} to having
\begin{equation}\label{criticité de type 2}
	W^{(\ell)} \mathop{\sim}\limits_{\ell \to \infty} \frac{p_{\bf q}}{2} c_{\bf q}^{\ell+1} \ell^{-a},
\end{equation}
for some constant $p_{\bf q}>0$. The cases $a\in (3/2,5/2)$ are called non-generic. Such sequences $\bf q$ do exist, see Lemma 6.1 of \cite{BC} for explicit examples.
If $\bf q$ is admissible, we write $\P^{(\ell)}$ the probability measure on $\mathcal{M}^{(\ell)}$ associated with the weight sequence ${\bf q}$, called the Boltzmann measure, defined by
$$
\P^{(\ell)}(\mathfrak{m}) = \frac{w_{\bf q}(\mathfrak{m})}{W^{(\ell)}}.
$$
We also write $\E^{(\ell)}$ for the associated mathematical expectation. A random map $\mathfrak{M}^{(\ell)}$ following the law $\P^{(\ell)}$ is called a $\bf q$-Boltzmann random map of perimeter $2\ell$. The $\bf q$-Boltzmann random maps for $\bf q$ of type $a$ are also called $(a-1/2)$-stable maps. The object of study in this work is not the map $\mathfrak{M}^{(\ell)}$ in itself but rather its dual map $\mathfrak{M}^{(\ell,\dagger)}$ obtained by exchanging the roles of the faces and of the vertices. 

For $a=2$, the local limit, the infinite $3/2$-stable map, was studied in \cite{BCM} by Budd, Curien and Marzouk. In particular, they established the scaling limit of the perimeter associated with the exploration of the infinite map, and studied its basic metric properties and their results will be especially helpful in Section \ref{section distance racine face}.

The peeling exploration of the map makes appear a random walk $S$ on $\Z$ with step distribution $\nu$ which will play a crucial part in this work. It is related to $\bf q$ by
\begin{equation}\label{def marche}
	\forall k \ge 0, \enskip \nu(k)= q_{k+1} c_{\bf q}^k \qquad \text{and} \qquad \nu(-k-1)= 2W^{(k)} c_{\bf q}^{-k-1}.
\end{equation}
When $\bf q$ is critical of type $a=2$, we know in particular by Proposition 5.10 of \cite{StFlour} that 
\begin{equation}\label{comportement asymptotique nu}
	\nu(-k)\sim p_{\bf q} k^{-2} \qquad \text{and}  \qquad \nu([k,\infty)) \sim p_{\bf q}  k^{-1} \enskip \text{ as } k \longrightarrow \infty,
\end{equation}
so that $\nu$ is in the domain of attraction of the symmetric Cauchy process (with no drift).

\subsection{Main results}

Throughout this work, unless explicitly stated otherwise, $\bf q$ will be a critical non-generic weight sequence of type $a=2$ and, for all $\ell \ge 1$, $\mathfrak{M}^{(\ell)}$ will be a $\bf q$-Boltzmann random map of perimeter $2\ell$.

Before stating the main results, let us introduce the distance functions on a map $\mathfrak{m}$. For every faces $f_1$, $f_2$ of $\mathfrak{m}$, we denote by $d_{\mathrm{gr}}^\dagger (f_1,f_2)$ the graph distance between $f_1$ and $f_2$ in the dual map $\mathfrak{m}^\dagger$. Another distance, easier to study in our context, is the first-passage percolation distance $d_\mathrm{fpp}^\dagger$ defined as follows: to each edge of $\mathfrak{m}^\dagger$ we assign an independent random variable following the exponential law of parameter $1$ (that we will write $\mathcal{E}(1)$), interpreted as its random length. For every faces $f_1$, $f_2$ of $\mathfrak{m}$, we then denote by $d_{\mathrm{fpp}}^\dagger (f_1,f_2)$ the distance between $f_1$ and $f_2$ in the dual map $\mathfrak{m}^\dagger$ whose edges are equipped with those random lengths.

The first theorem describes the distance from the root face to a uniform random face (we also prove that the same results hold for a uniform random edge or a uniform random vertex).

\begin{theorem}\label{distance de la racine à une face uniforme}
	For all $\ell \ge 1$, let $F^{(\ell)}$ be a random uniform face in $\mathfrak{M}^{(\ell)}$ (conditionally on $\mathfrak{M}^{(\ell)}$) and recall that $f_r$ is the root face. Then
	$$\frac{d_\mathrm{fpp}^\dagger (f_r, F^{(\ell)})}{\log \ell } \mathop{\longrightarrow}\limits_{\ell \to \infty}^{(\P)} \frac{1 }{\pi^2 p_{\bf q}}
	\qquad \text{and} \qquad
	\frac{d_\mathrm{gr}^\dagger (f_r, F^{(\ell)})  }{(\log \ell)^2 } \mathop{\longrightarrow}\limits_{\ell \to \infty}^{(\P)} \frac{1 }{2\pi^2 }.
	$$
\end{theorem}
The scales for the distances in $\log \ell$ and $(\log \ell )^2$ are specific to the $3/2$-stable case on account of its criticality and of the singular behavior of Cauchy processes. Deeply related behaviors have already been observed for the volume growth of balls in the infinite map in Theorems A and B of \cite{BCM}.

We then prove that the geodesics from the root to two uniform random faces ``split very quickly'', so that we get the asymptotic of the distance between two uniform random faces, see Corollary \ref{distance entre deux faces uniformes}. 
This corollary implies that if $(F_i^{(\ell)})_{i\ge 1}$ are i.i.d. uniform random faces of $\mathfrak{M}^{(\ell)}$ for all $\ell \ge 1$, then, for the product topology,
\begin{equation}\label{eq étoile}\left(\frac{d_\mathrm{fpp}^\dagger(F_i,F_j)}{\log \ell} \right)_{i,j\ge 1}
	\mathop{\longrightarrow}\limits_{\ell \to \infty}^{(\P)}
	\left(\frac{2}{ \pi^2 p_{\bf q}}\right)_{i,j\ge 1}
	\qquad
	\text{and}
	\qquad
	\left(\frac{d_\mathrm{gr}^\dagger(F_i,F_j)}{(\log \ell)^2} \right)_{i,j\ge 1}
	\mathop{\longrightarrow}\limits_{\ell \to \infty}^{(\P)}
	\left(\frac{1}{ \pi^2}\right)_{i,j\ge 1}.
\end{equation}

So as to study the diameter in Section \ref{section diamètre}, we control the distances uniformly on all the faces. Aside from the first and second moment methods, a convenient tool is a martingale associated with the perimeter process, hence enabling us to show that the diameter is of the same order as the distance between two faces picked uniformly at random.

\begin{theorem}\label{majoration du diamètre graphe et fpp}
	Let $\mathrm{Diam}_\mathrm{gr}^\dagger(\mathfrak{M}^{(\ell)})$ (resp. $\mathrm{Diam}_\mathrm{fpp}^\dagger(\mathfrak{M}^{(\ell)})$) be the diameter of $\mathfrak{M}^{(\ell,\dagger)}$ for the graph distance (resp. fpp distance). Let $\delta>0$. There exists a constant $K({\bf q})$ depending on $\bf q$ such that, with probability $1-o(1)$ when $\ell \to \infty$,
	$$
	\left(\frac{3}{4(1-q_1^2)}\vee \frac{2}{\pi^2 p_{\bf q}} \right) -\delta \le
	\frac{\mathrm{Diam}_\mathrm{fpp}^\dagger(\mathfrak{M}^{(\ell)})}{\log \ell}
	\le K({\bf q}) \qquad
	\text{and} \qquad
	\frac{1 }{\pi^2 }-\delta \le
	\frac{\mathrm{Diam}_\mathrm{gr}^\dagger(\mathfrak{M}^{(\ell)})}{(\log \ell)^2}
	\le 18+\delta.
	$$
\end{theorem}

These bounds on the diameter indicate in particular the scale at which one should look for studying Gromov-Hausdorff convergence. Actually, (\ref{eq étoile}) shows that at the macroscopic scale, $\mathfrak{M}^{(\ell),\dagger}$ looks like a star with many branches when $\ell \to \infty$, see the simulation in Figure \ref{fig:ssimus_intro}. As a consequence, the sequence $\left(\mathfrak{M}^{(\ell),\dagger}\right)_{\ell \ge 1}$ equipped with the rescaled distances is not tight, neither for the Gromov-Hausdorff topology, nor for the Gromov-Prokhorov topology. Indeed, on the one hand, the limit of any subsequence can not be compact due to (\ref{eq étoile}), while on the other hand any limiting random separable metric measured space $(X,d,\mu)$ of some subsequence for Gromov-Prokhorov would have an empty support since two uniform points of law $\mu$ would be at constant distance by (\ref{eq étoile}), which contradicts the separability (see e.g. Remark 4.5 of \cite{J20} for details on this reasoning).

More precisely, for all $\ell \ge 1$, let $\mu_\ell$ be the uniform measure on the faces of $\mathfrak{M}^{(\ell)}$. We have thus shown the following corollary.
\begin{corollary}\label{noGPnoGH}
	The sequences $( \mathfrak{M}^{(\ell,\dagger)}, \frac{1}{(\log \ell)^2}d_\mathrm{gr}^\dagger, \mu_\ell )_{\ell\ge 1}$ and $( \mathfrak{M}^{(\ell,\dagger)}, \frac{1}{ \log \ell}d_\mathrm{fpp}^\dagger, \mu_\ell )_{\ell\ge 1}$ are not tight for the Gromov-Prokhorov topology. Besides, the sequences 
	$
	(\mathfrak{M}^{(\ell,\dagger)}, \frac{1}{(\log\ell)^2} d_\mathrm{gr}^\dagger)_{\ell \ge 1}
	$ and 
	$
	(\mathfrak{M}^{(\ell,\dagger)}, \frac{1}{ \log\ell} d_\mathrm{fpp}^\dagger)_{\ell \ge 1}
	$
	are not tight for the Gromov-Hausdorff topology. In particular, there is no Gromov-Hausdorff scaling limit of any subsequence of $(\mathfrak{M}^{(\ell,\dagger)}, d_\mathrm{gr}^\dagger )_{\ell \ge 1}$ or $(\mathfrak{M}^{(\ell,\dagger)}, d_\mathrm{fpp}^\dagger )_{\ell \ge 1}$
\end{corollary}

Still, following \cite{ET22} and \cite{Ja21}, the convergence $(\ref{eq étoile})$ can also be interpreted  as the convergence in probability of $( \mathfrak{M}^{(\ell,\dagger)}, \frac{1}{ (\log \ell)^2}d_\mathrm{gr}^\dagger, \mu_\ell )_{\ell\ge 1}$ and $( \mathfrak{M}^{(\ell,\dagger)}, \frac{1}{ \log \ell}d_\mathrm{fpp}^\dagger, \mu_\ell )_{\ell\ge 1}$ towards some long dendrons defined in Example 7.2 of \cite{Ja21}.

\subsection{Conditioned random walks}

The main technique to study the metric properties of $\mathfrak{M}^{(\ell,\dagger)}$ is the filled-in peeling exploration. It was first introduced in \cite{Wa95} and \cite{An03} for triangulations. A general version is defined in \cite{B16} and is presented in Section \ref{section exploration}. See \cite{BuddPeeling} and \cite{StFlour} for a more detailed introduction. The peeling process consists in a step-by-step Markovian exploration of the map starting from the root face according some peeling algorithm. At each step, the peeling process reveals the face behind the peeled edge chosen by the peeling algorithm. Various choices of peeling algorithms enable to understand different properties of the map. Our use of the peeling algorithm is similar to that of \cite{BC} for the study of infinite $(a-1/2)$-stable Boltzmann maps for $\bf q$ critical non-generic of type $a \in (3/2,5/2)\setminus\left\{2\right\}$ and \cite{BCM} for the study of infinite $3/2$-stable Boltzmann maps. A difference is that most of the time we will work with the peeling algorithm on the map which has been rerooted (see Subsection \ref{échange racine cible}).

In our case, the evolution of the perimeter of the explored part has the same law as the $\nu$-random walk conditioned to stay positive until it jumps and dies at $-\ell$. This is, to the best of our knowledge, a new type of conditioning which we study in details in Section \ref{section périmètre}. 
In that section we introduce a key coupling (Lemma \ref{couplage marche}) between this conditioned random walk and the $\nu$-random walk conditioned to stay positive forever
, using that such conditionings can be performed via Doob $h$-transforms. This coupling enables us in particular to use some results of \cite{BCM}. 
This lemma states that for all $n\ge 1$, there exists a coupling between the two conditioned random walks such that if the second one has not been killed yet at time $n$, then the two conditioned random walks coincide.

The rest of Section \ref{section périmètre} can be seen as an asymptotic study of this conditioned $\nu$-random walk. 
We first prove a local limit theorem for the lifetime and the last positive value of the walk. Relying on this local limit, we next establish in Theorem \ref{cvpérimètre} the scaling limit of that conditioned random walk dying at $-\ell$ towards a process that we call an $(a-1)$-stable Lévy process conditioned to die at $-1$. We hope that those results could be fruitful in other contexts.

\subsection{Discussion}

The gasket of a critical $O(n)$ loop-decorated planar map, which is obtained after deleting the edges and vertices in the outermost loops, is a particular example of an $(a-1/2)$-stable map, where the case $n=2$ corresponds to $a=2$. This relation is depicted for instance in \cite{LGM11,BBG12}.
The $O(n)$ loop-decorated maps undergo a phase transition described in \cite{BBG12,B18}. The critical case $n=2$ is often excluded in the literature. 
For the infinite Boltzmann $(a-1/2)$-stable maps, in \cite{BC} a phase transition at $a=2$ is equally exhibited. Let us also mention the conjecture that for $n=2$ the gasket should converge in some sense to the $\mathrm{CLE}_4$ ensemble (introduced in \cite{S09,SW12}).

We are confident that some of our results such as Theorem \ref{distance de la racine à une face uniforme} and Corollary \ref{distance entre deux faces uniformes} can be extended to the case $a \in (3/2,2)$ with a $\log \ell$ scaling for the dual graph distance. The martingale (\ref{martingale}) seems also adapted in this setting, so that Theorem \ref{majoration du diamètre graphe et fpp} should have an analog in this case.

\section{The random walk conditioned to die at $-\ell$}\label{section périmètre}

This section is devoted to the study the $\nu$-random walk $S$ conditioned to stay positive until it jumps and dies at $-\ell$, which will be denoted by $(P_\ell(n))_{n\ge 0}$ under $\P_\ell^{(1)}$. 
We establish in Lemma \ref{couplage marche} a coupling between this conditioned random walk and the random walk conditioned to stay always positive. 
We then prove a local limit result for the hitting time $\tau_{-\ell}$ of $-\ell$ (Proposition \ref{limite locale bivariée}) which we use to establish a scaling limit (Theorem \ref{cvpérimètre}).

\subsection{The $\nu$ random walk, harmonic functions and conditionings}
The criticality condition (\ref{criticité de type 2}) imposes strong properties on the step distribution $\nu$ which are better stated in terms of harmonic functions. 
For all $\ell\ge 0$ and $p\ge 1$, let
\begin{equation}\label{h flèche}
	h^\downarrow_p(\ell)= h^\downarrow (\ell) h^\downarrow(p)\frac{\ell}{\ell+p} \qquad \text{and} \qquad h_0^\downarrow (\ell)=h^\downarrow(\ell) = 2^{-2\ell} \binom{2\ell}{\ell},
\end{equation}
with by convention $h^\downarrow(\ell)=h_p^\downarrow(\ell)=0$ for $\ell <0$ except for $\ell=-p$ where we set $h_p(-p)=1$. We also set $h^\uparrow(\ell)=2\ell h^\downarrow(\ell)$ for all $\ell \in \mathbb{Z}$.
Let $a \in (3/2,5/2]$. In the rest of this section we only assume that $\nu$ is a probability measure on $\mathbb{Z}$ such that
\begin{enumerate}[(I)]
	\item \label{condition harmonicité}The functions $h^\downarrow$ and $h^\uparrow$ are harmonic for $\nu$.
	\item \label{condition queue}We have $\nu(-k)\sim p_{\bf q}/k^a$ as $k \to \infty$.
\end{enumerate}

These two conditions are equivalent to the criticality of type $a$ of the weight sequence $\bf q$ associated with $\nu$ by (\ref{def marche}) (see the paragraph 3.2 in \cite{B16}, Theorem 1 and Proposition 5 of \cite{BuddPeeling} or Theorem 5.4 and its proof in \cite{StFlour} for (\ref{condition harmonicité}), and Propositions 5.9 and 5.10 of \cite{StFlour} for (\ref{condition queue})). The admissibility of $\bf q$ then entails the harmonicity of $h^\downarrow_p$ for $S$ for all $p\ge 1$ (see the proof of Lemma 5.2 in \cite{StFlour}).  

It is also known (see e.g. Proposition 10.1 in \cite{StFlour}) that $(S_{\lfloor nt \rfloor}/n^{1/(a-1)})_{t\ge 0}$ converges in distribution towards $p_{\bf q}^{1/(a-1)} \Upsilon_a$ for the Skorokhod topology, where $\Upsilon_a$ is the $(a-1)$-stable Lévy process starting at zero with positivity parameter $\rho = \P(S_t\ge 0)$ satisfying $(a-1)(1-\rho)= 1/2$, normalized so that its Lévy measure is $\cos(a \pi) (dx/x^a){\bf 1}_{x>0} + (dx/|x|^a){\bf 1}_{x<0}$.

Finally, from the harmonic functions, one can define the corresponding Doob $h$-transforms of the $\nu$ random walk $S$. For all $\ell\ge 1$, we denote by $(P_\infty(n))_{n\ge 0}$ under $\P^{(\ell)}_\infty$ the Doob $h^\uparrow$-transform of $S$ starting at $\ell$ and for all $\ell \ge 1, p\ge 0$, we write $(P_p(n))_{n\ge 0}$ under $\P^{(\ell)}_p$ the Doob $h^\downarrow_p$-transform of $S$ killed when it reaches $-p$. The notation is chosen to be consistent with the rest of the paper since these conditioned random walks will be the perimeter processes in Section \ref{section exploration}. The Doob $h^\uparrow$-transform of $S$ can be seen as the random walk $S$ conditioned to stay always positive while the Doob $h^\downarrow_p$-transform of $S$ killed at $-p$ can be understood analogously as the random walk conditioned to stay positive until it is killed at $-p$ (see \cite{StFlour} Proposition 5.3).

\subsection{Preliminaries on the random walk conditioned to stay positive} 

Before delving into the study of the walk $P_\ell$ under $\P^{(1)}_\ell$, we record for future use some known results on the walks $P_\infty$, $P_p$ under $\P^{(1)}_\infty$, $\P^{(\ell)}_p$ for fixed $p$ as $\ell \to + \infty$
. 
One of the most useful results will be the scaling limit of the random walk conditioned to stay always positive $P_\infty$ under $\P^{(1)}_\infty$. This scaling limit was introduced in \cite{BC} and \cite{BCM} using Theorem 1.1 from \cite{CC}. The limit is described informally as the Lévy process $\Upsilon_a$, introduced in the above subsection, conditioned to stay positive and starting from zero. To do this, one can first define the process $(\mathcal{S}^\uparrow_x(t))_{t\ge 0}$ starting at some $x>0$ as the Doob $h$-transform of $(x+\Upsilon_a(t))_{t\ge 0}$ starting at $x$, using the harmonic function $h:y\mapsto {\bf 1}_{y\ge 0}\sqrt{y}$. The process $(\mathcal{S}^\uparrow_x(t))_{t\ge 0}$ then converges in distribution when $x\to 0$ towards a process written $\Upsilon_a^\uparrow$ which starts at zero. Recall the constant $p_{\bf q}$ which first appears in (\ref{criticité de type 2}).

\begin{theorem}\label{upsilon flèche}(Proposition 10.3 of \cite{StFlour})
	The following convergence holds for the $J_1$-Skorokhod topology: under $\P^{(1)}_\infty$,
	$$
	\left(\frac{P_\infty(\lfloor n t\rfloor)}{n^{1/(a-1)}}\right)_{t\ge 0} 
	\enskip
	\mathop{\longrightarrow}\limits_{n \to \infty}^{(\mathrm{d})} 
	\enskip
	\left(\Upsilon_a^\uparrow(p_{\bf q} t)\right)_{t\ge 0}\mathop{=}\limits^{(\mathrm{d})} \left(p_{\bf q}^{1/(a-1)}\Upsilon_a^\uparrow( t)\right)_{t\ge 0}.
	$$
\end{theorem}

The proof of Theorem 1.1 in \cite{CC} relies mostly on the scaling limit of the $\nu$-random walk $S$ conditioned to stay positive until some time towards a stable meander: let $\widetilde{\Upsilon}_a$ be the meander of length $1$ associated with $\Upsilon_a$, which is informally the process $\Upsilon_a$ conditioned to remain positive up to time $1$. Then the following theorem is a consequence of the main theorem of \cite{D85}.

\begin{theorem}\label{cvméandre}(\cite{D85})
	For $S_0=1$, conditionally on $S_1, \ldots,S_n>0$, the process $(S(\lfloor nt \rfloor)/n^{1/(a-1)})_{t\ge 0}$ converges in distribution for the $J_1$-Skorokhod topology towards  $(p_{\bf q}^{1/(a-1)} \widetilde{\Upsilon}_a(t))_{t\ge 0}$.
\end{theorem}

The density of $\Upsilon_a^\uparrow(1)$ (resp. $\widetilde{\Upsilon}_a(1)$)  will be written $f^\uparrow_a$ (resp. $\tilde{f}_a$). We know from \cite{DS} that  $\tilde{f}_a$ is continuous and bounded (and also bounded away from zero on any compact of $(0,\infty)$) and moreover that there exists $C_1>0$ such that $f_a^\uparrow (x) = C_1 x^{1/2} \tilde{f}_a(x)$.

Besides, so as to establish our local limit results, we state the following lemma which is an extension and a reformulation of Lemma 3 (i) of \cite{BCM}. 
Its proof is exactly the same as in \cite{BCM} and builds upon \cite{VW}.

\begin{lemma}\label{lemme3BCM}(Lemma 3 (i) from \cite{BCM})
	Uniformly on $x\ge 1$ as $n\to \infty$, 
	$$
	\P_1\left(S_1, \ldots, S_n \ge 1 \text{ and } S_n=x\right) 
	\enskip
	=
	\enskip
	\frac{\sqrt{\pi}}{2(np_{\bf q})^{1/(a-1)}\sqrt{x}} f_a^\uparrow\left(\frac{x}{(p_{\bf q}n)^{1/(a-1)}}\right)
	+o(n^{-3/(2(a-1))}).
	$$
\end{lemma}
Finally, the next proposition, which describes the scaling limit of $P_p$ under $\P^{(\ell)}_p$ for fixed $p$, will be of some use in Section \ref{section distance entre sommets uniformes}. The scaling limit is described by the Lévy process $\Upsilon_a^\downarrow$ conditioned to stay positive until it dies continuously at zero. The process $\Upsilon_a^\downarrow$ can be defined as a Doob transform of $\Upsilon_a+1$ using the harmonic function $x \mapsto (1/\sqrt{x}){\bf 1}_{x>0}$.
This proposition slightly extends a particular case of Proposition 6.6 from \cite{BBCK}. Not only does its proof rely on the convergence in distribution of $(S_{\lfloor nt \rfloor}/n^{1/(a-1)})_{t\ge 0}$ towards $p_{\bf q}^{1/(a-1)} \Upsilon_a$ but it also relies on Theorem 1.3 of \cite{CC} in the case $p=0$. In that case $P_0$ is a Doob $h^\downarrow$ transform of the $\nu$-random walk, i.e. a random walk conditioned to die at zero. The proof of Theorem 1.3 in \cite{CC} can be rewritten by replacing $h^\downarrow$ (which is denoted by $W$ in \cite{CC}) by $h^\downarrow_p$, i.e. by replacing the walk conditioned to die at zero by a walk conditioned to stay positive until it dies at $-p$. 
\begin{proposition}\label{périmètre des cartes à grand bord}(Extension of Proposition 6.6 from \cite{BBCK})
	For all $p\ge 0$, the following convergence holds for the $J_1$-Skorokhod topology: under $\P^{(\ell)}_p$,
	$$
	\left(\frac{P_p(\lfloor\ell^{a-1} t\rfloor)}{\ell}\right)_{t\ge 0} 
	\enskip
	\mathop{\longrightarrow}\limits_{\ell \to \infty}^{(\mathrm{d})}
	\enskip
	\left(\Upsilon_a^\downarrow(p_{\bf q} t)\right)_{t\ge 0}.
	$$
	Moreover, jointly with the above convergence, we have the convergence of the lifetimes:
	$$
	\frac{\tau_{-p}}{\ell^{a-1} }
	\enskip
	\mathop{\longrightarrow}\limits_{\ell \to \infty}^{(\mathrm{d})}
	\enskip
	\frac{\zeta }{p_{\bf q}},
	$$
	where $\tau_{-p}=\inf\left\{ n\ge 0; \ P_p(n)=-p\right\}$ and $\zeta= \inf \left\{t\ge 0; \  \Upsilon_a^\downarrow(t)=0 \right\}$.
\end{proposition}
One may alternatively prove the first convergence of the above proposition by applying Theorem 2 of \cite{BK} while the second assertion can be seen as a consequence of Theorem 3 (i) of \cite{BK}.

\subsection{Coupling random walks conditioned to stay positive with random walks conditioned to die at $-\ell$}

Let us denote by $\tau_{-\ell}$ the first hitting time of $-\ell$ by  $(P_\ell(k))_{k\ge 0}$ the random walk conditioned to remain positive until it dies at $-\ell$ under $\P^{(1)}_\ell$. Let us first state a coupling lemma which will be useful in this section and in the next ones. In particular, it will enable us to take advantage of the results from \cite{BCM} in Section \ref{section distance racine face}. 
\begin{lemma}\label{couplage marche}
	There exists a coupling of the random walks $P_\ell$ under $\P^{(1)}_\ell$ and $P_\infty$ under $\P^{(1)}_\infty$ up to time $n$ written $\P^{(1),n}_{\ell,\infty}$ such that:
	under $\P^{(1),n}_{\ell,\infty}$, for all $s_1, \ldots, s_n \in \N^*$, conditionally on $(P_\infty(k))_{1\le k \le n} = (s_k)_{1\le k \le n}$, the equality $(P_\ell(k))_{1\le k \le n} = (s_k)_{1\le k \le n}$ holds with probability $\frac{1+\ell}{ s_n+\ell}$ and $\tau_{-\ell} \le n$ with probability $1-\frac{1+\ell}{ s_n+\ell}$ (see Figure \ref{image couplage}).
\end{lemma}

\begin{figure}[h]
	\centering
	\includegraphics[scale=0.75]{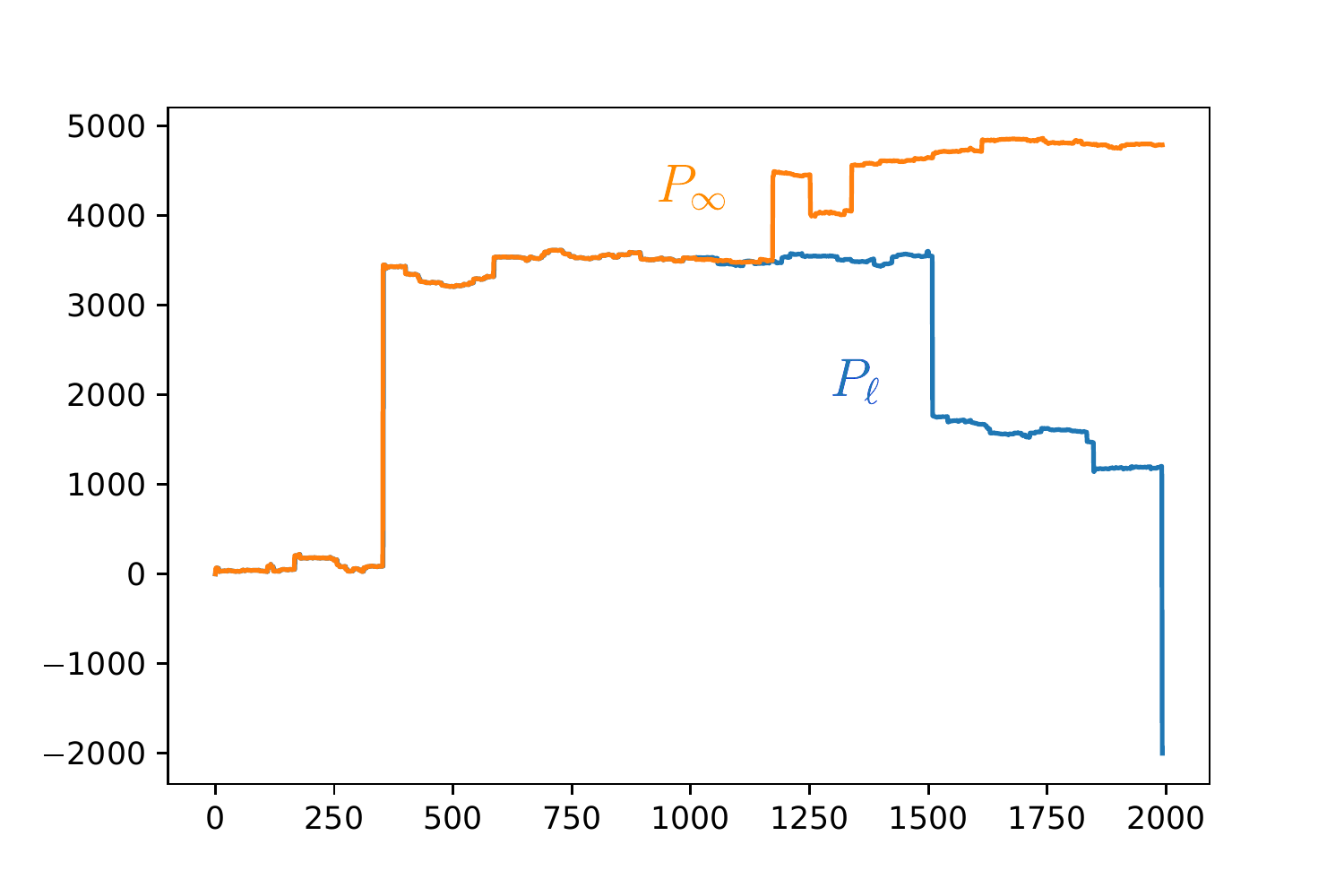}
	\caption{Simulation of the conditioned random walks $P_\infty$ and $P_{\ell}$ under the coupling $\P^{(1),n}_{\ell,\infty}$ in the case $\tau_{-\ell}>n$ for $n=1000$ and $\ell=2000$ (after time $n$, we let both processes evolve independently).}
	\label{image couplage}
\end{figure}

\begin{proof}
	The key ingredient of the proof is the relation $h^\downarrow_\ell(k)=h^\downarrow(\ell)h^\uparrow(k)/(2(\ell+k))$ for $k,\ell \ge 1$ coming from the exact expression of the harmonic functions. From the definition of $P_\ell$ under $\P^{(1)}_\ell$ as a $h^\downarrow_\ell$-Doob transform of a $\nu$-random walk $S$, we note that if $s_1,\ldots, s_n \in \N^*$, then
	\begin{align*}
		\P^{(1)}_\ell ((P_\ell(k))_{1\le k \le n} = (s_k)_{1\le k \le n}) &=  \frac{h^\downarrow_\ell (s_n) }{ h^\downarrow_\ell (1)} \P_1( (S_k)_{1\le k \le n} =(s_k)_{1\le k \le n} ) \\
		&= \frac{h^\uparrow (s_n) }{ h^\uparrow (1)} \frac{1+\ell }{ s_n+\ell} \P_1( (S_k)_{1\le k \le n} =(s_k)_{1\le k \le n} ) \\
		&= \frac{1+\ell }{ s_n+\ell} \P^{(1)}_\infty ((P_\infty(k))_{1\le k \le n} = (s_k)_{1\le k \le n}) \\
		&\le \P^{(1)}_\infty ((P_\infty(k))_{1\le k \le n} = (s_k)_{1\le k \le n}).
	\end{align*}
	One can thus define such a coupling.
\end{proof}

One can see from the proof that the above lemma only relies on condition (\ref{condition harmonicité}). A simple consequence of this coupling is a limit theorem for $\tau_{-\ell}/\ell^{(a-1)}$ when $\bf q$ has type $a\in (3/2,5/2]$. Recall that $\Upsilon_a^\uparrow$ is the stable Lévy process $\Upsilon_a$ starting at zero conditioned to stay positive.

\begin{proposition}\label{limite du temps d'arret}
	We have the convergence in law ${\tau_{-\ell}}/{\ell^{a-1}} \mathop{\longrightarrow}\limits_{\ell \to \infty}^\text{d} \tilde{\tau}/p_{\bf q}$
	where $\tilde{\tau}$ is a random variable such that for all $u\ge 0$,
	$
	\P(\tilde{\tau}\ge u)  = \E( {1/ (1+  \Upsilon_a^\uparrow(u))}).
	$
\end{proposition}

\begin{proof}
	By Lemma \ref{couplage marche} and Theorem \ref{upsilon flèche}, for every $u \ge 0$,
	\begin{align*}
		\P^{(1)}_\ell(\tau_{-\ell} > u\ell^{a-1} )  
		&=\E^{(1)}_\infty \frac{1 + \ell }{ P_\infty({\lfloor u \ell^{a-1} \rfloor}) +\ell} \mathop{\longrightarrow}\limits_{\ell \to \infty}
		\E \frac{1 }{ 1+ \Upsilon_a^\uparrow(p_{\bf q}u)},
	\end{align*}
	and the claim follows.
\end{proof}

\subsection{Local limit results for the random walk conditioned to die at $-\ell$}

The result of Proposition \ref{limite du temps d'arret} can be refined to obtain a joint local limit theorem for the lifetime $\tau_{-\ell}$ together with the last positive value of the conditioned random walk $P_\ell(\tau_{-\ell}-1)$:
\begin{proposition}\label{limite locale bivariée} For all $\vp >0$ the convergence below holds uniformly in $u,s\in [\vp,\vp^{-1}]$:
	$$ \ell^a\P^{(1)}_\ell (\tau_{-\ell}  = \lfloor \ell^{a-1} s \rfloor \text{ and } P_\ell(\tau_{-\ell}-1)  = \lfloor \ell u  \rfloor) \mathop{\longrightarrow}\limits_{\ell \to \infty} \frac{\pi p_{\bf q}^{1-1/(a-1)}}{ s^{1/(a-1)}}  \frac{1}{\sqrt{u} (1+u)^a} f_a^\uparrow\left(\frac{u}{ (p_{\bf q}s)^{1/(a-1)}}\right) ,$$
	and the right-hand side is a probability density.
\end{proposition}
\begin{proof}
	Let $\vp >0$. Let us write $n=\lfloor \ell^{a-1} s \rfloor$ and $x=\lfloor \ell u  \rfloor$. Then,
	$$
	\frac{1}{ 2} \frac{h^\downarrow (\ell) }{ \ell +1}\P^{(1)}_\ell (\tau_{-\ell}  = \lfloor \ell^{a-1} s \rfloor \text{ and } P_\ell(\tau_{-\ell}-1)  = \lfloor \ell u  \rfloor)=\P_1(S_1, \ldots, S_{n-1} \ge 1 \text{ and } S_{n-1} = x) \nu(-\ell-x).
	$$
	But by Lemma \ref{lemme3BCM} and since $\nu(-k)\sim p_{\bf q} k^{-a}$ as $k \to \infty$ by (\ref{condition queue}), uniformly in $u,s \in[\vp,\vp^{-1}]$,
	\begin{align*}
		\P_1(S_1, \ldots, S_{n-1} \ge 1 \text{ and } S_{n-1} = x) \nu(-\ell-x) 
		&\mathop{\sim}\limits_{\ell\to \infty} \frac{\sqrt{\pi} p_{\bf q}^{1-1/(a-1)}}{2 s^{1/(a-1)} \ell^{a+3/2}}  f_a^{\uparrow} \left(\frac{ u}{ (p_{\bf q} s)^{1/(a-1)} }\right) \frac{1}{ \sqrt{u} (1+u)^a} .
	\end{align*}
	This proves the desired convergence using that $h^\downarrow(\ell)\sim 1/\sqrt{\pi \ell}$ as $\ell \to \infty$. Let us check that the right term indeed corresponds to a probability density, in other words that
	$$
	\int_0^{\infty} \frac{\pi p_{\bf q}^{1-1/(a-1)}}{s^{1/(a-1)} } \int_0^{\infty} f_a^{\uparrow} \left(\frac{ u}{ (p_{\bf q} s)^{1/(a-1)} }\right) \frac{1}{ \sqrt{u} (1+u)^a} du ds = 1
	$$
	By the change of variable $w = \frac{u}{ (p_{\bf q} s)^{1/(a-1)} }$ at fixed $u$,
	\begin{align*}
		\int_0^{\infty} \frac{\pi p_{\bf q}^{1-1/(a-1)}}{s^{1/(a-1)} } \int_0^{\infty} f_a^{\uparrow} \left(\frac{ u}{ (p_{\bf q} s)^{1/(a-1)} }\right) \frac{1}{\sqrt{u} (1+u)^a} du ds
		&=
		\int_0^\infty\int_0^\infty \frac{\pi u^{a-2} (a-1)}{\sqrt{u}(1+u)^a} w^{1-a} f_a^\uparrow(v) dwdu \\
		&=\int_0^\infty \frac{\pi u^{a-2} (a-1)}{\sqrt{u}(1+u)^a} du \E \frac{1}{\Upsilon^\uparrow_a(1)^{a-1}}.
	\end{align*}
	
	Yet, by the change of variable $t=1/(1+u)$, by using the link between the Beta function and the Gamma function and using that $\Gamma(3/2)= \sqrt{\pi}/2$, one can see that 
	$$
	\int_0^\infty \frac{\pi u^{a-2} (a-1)}{\sqrt{u}(1+u)^a} du = \frac{(a-1)\pi^{3/2} \Gamma(a-3/2)}{2\Gamma(a)}.
	$$
	Finally, by exactly the same reasoning as in the end of Remark 3 in \cite{BCM}, one can compute 
	$$
	\E \frac{1}{\Upsilon^\uparrow_a(1)^{a-1}} = \frac{2\Gamma(a)}{(a-1)\pi^{3/2} \Gamma(a-3/2)},
	$$
	which is precisely the inverse of the above integral. Indeed, since $\Upsilon_a^\uparrow$ is a positive self-similar Markov process, let us denote by $\xi^\uparrow$ the Lévy process arising in the Lamperti representation of $\Upsilon^\uparrow_a$. Then by \cite{BY02}, we know that $\E 1/\Upsilon^\uparrow_a(1)^{a-1} = 1/((a-1)\E \xi^\uparrow_1)$. But the Laplace exponent of $(\Gamma(a)/\pi)\xi^\uparrow$ is given by $\kappa_{a-1}(\omega_+ +q) $ where $\omega_+=a+1/2$ and $\kappa_{a-1}$ is defined by Equation (19) of \cite{BBCK}. One could have made the computations using the characteristic exponent of $\xi^\uparrow$ identified in Proposition 2 of \cite{KPR10}.
\end{proof}

From Proposition \ref{limite locale bivariée} one can get the following local estimates for $\tau_{-\ell}$ and $P_{\ell}(\tau_{-\ell}-1)$. For $\tau_{-\ell}$ it suffices to sum on the values of $P_\ell(\tau_{-\ell}-1)$ in an interval $[\ell\delta,\ell/\delta]$ for an arbitrary small $\delta>0$. Thus we get that for all $s>0$, when $\ell \to \infty$,
$$
\P^{(1)}_\ell (\tau_{-\ell}  = \lfloor \ell^{a-1} s\rfloor) \ge (1+o(1))\frac{\pi p_{\bf q}^{1-1/(a-1)}}{ \ell^{a-1} s^{1/(a-1)}} \int_0^{\infty} \frac{1}{ \sqrt{u} (1+u)^a} f_a^\uparrow\left({u\over (p_{\bf q}s)^{1/(a-1)}}\right)  du.
$$
In the same vein, one can check the local estimate for $P_\ell(\tau_{-\ell}-1)$: for all $u>0$, when $\ell \to \infty$,
$$\P^{(1)}_\ell (P_\ell(\tau_{-\ell}-1)  = \lfloor \ell u  \rfloor) \ge  \frac{2\Gamma(a)}{ \sqrt{\pi}\Gamma(a-3/2) \ell }  \frac{u^{a-2}}{ \sqrt{u} (1+u)^a}   (1+o(1)).$$

\subsection{Scaling limit of the random walk conditioned to die at $-\ell$}

The convergence in distribution of $(\tau_{-\ell}/\ell^{a-1}, P_\ell(\tau_{-\ell}-1)/\ell)$ implied by Proposition \ref{limite locale bivariée} leads us to the following scaling limit: the rescaled random walk conditioned to die at $-\ell$, which is written $({P_\ell (\ell^{a-1} t ) / \ell})_{t\ge 0}$ under $\P^{(1)}_\ell$ converges in distribution towards $(\Upsilon^\downarrow_{-1,a}(p_{\bf q} t))_{t\ge 0}$ where $\Upsilon^\downarrow_{-1,a}$ is the $(a-1)$-stable Lévy process conditioned to stay positive until it dies at $-1$. This process is defined as follows. Let $(\tilde{\tau},Z)$ be a random variable of law 
$$
g(s,z)ds dz \coloneqq \frac{\pi}{ s^{1/(a-1)}}  \frac{1}{ \sqrt{z} (1+z)^a} f_a^\uparrow\left(\frac{z}{ s^{1/(a-1)}}\right) ds dz.
$$
We define $\Upsilon^\downarrow_{-1,a}$ so that, conditionally on $(\tilde{\tau},Z)$, the path $(\Upsilon^\downarrow_{-1,a}(t))_{0\le t <  \tilde{\tau}}$ is a stable meander associated with $\Upsilon_a$ of length $\tilde{\tau}$ conditioned on $\Upsilon^\downarrow_{-1,a}(\tilde{\tau}-)=Z$, and for all $t\ge \tilde{\tau}$, we set $\Upsilon^\downarrow_{-1,a}(t)=-1$.

Formally, the law of the process $\Upsilon^\downarrow_{-1,a}$ can be characterized with the law of the meander of length one $\widetilde{\Upsilon}_a$. To express it, we introduce the density of $\tilde{\tau}$ that we write $f$. 
For all $s>0$, we define the process $\widetilde{\Upsilon}_a^{s}$ using $\widetilde{\Upsilon}_a$ by setting for all $t \ge 0$,
$$
\widetilde{\Upsilon}^{s}_a(t) \coloneqq 
\left\{
\begin{array}{ll}
	s^{1/(a-1)} \widetilde{\Upsilon}_a\left(t/s\right) &\text{ if } t< s, \\
	-1 & \text{ otherwise.}
\end{array}
\right.
$$
Notice that $ (s^{1/(a-1)}\widetilde{\Upsilon}_a(t/s))_{t\ge 0}$ has the same law as the meander of length $ s $ evaluated at time $t $ because $\Upsilon_a$ is $(a-1)$-stable.
We then define the law of $\Upsilon^\downarrow_{-1,a}$ by the relation
$$
\E \, F\left(\left(\Upsilon^\downarrow_{-1,a}(t)\right)_{t\ge 0}\right) 
=\int_0^{\infty} f(s) \E \left(
\frac{g(s,s^{1/(a-1)}\widetilde{\Upsilon}_a(1))}{f(s)\tilde{f}_a\left(\widetilde{\Upsilon}_a(1)\right)}
F\left(\widetilde{\Upsilon}_a^{s}\right)
\right) ds
$$
for all bounded measurable function $F: \D(\R_+,\R)\to \R$ , where $\D(\R_+,\R) $ is the space of càdlàg functions from $\R_+$ to $\R$ equipped with the $J_1$ Skorokhod topology and $\tilde{f}_a$ is the density of $\widetilde{\Upsilon}_a(1)$. 
This is indeed a characterization of the law of $\widetilde{\Upsilon}^\downarrow_{-1,a}$ as the meander of length $s>0$ has the same law as $ (s^{1/(a-1)}\widetilde{\Upsilon}_a(t/s))_{t\ge 0}$.
\begin{theorem} \label{cvpérimètre}
	With the convention that $P_\ell(k)=-\ell$ for $k\ge \tau_{-\ell}$, we have the convergence in distribution with respect to the Skorokhod topology:
	$$
	\text{Under } \P^{(1)}_\ell, \qquad
	\left(\frac{P_\ell (\ell^{a-1} t ) }{ \ell}\right)_{t \ge 0}
	\mathop{\longrightarrow}\limits_{\ell \to \infty}^{(\mathrm{d})}
	(\Upsilon^\downarrow_{-1,a}(p_{\bf q}t))_{t\ge 0}.
	$$
\end{theorem}
The above theorem provides a first reason to see $\Upsilon^\downarrow_{-1,a}$ as the Lévy process $\Upsilon_a$ conditioned to stay positive until it jumps and dies at $-1$.
\begin{proof}
	The main ideas of the proof are the scaling limit result on random walks conditioned to stay positive up to time $n$ of \cite{D85}, which was recalled in Theorem \ref{cvméandre}, together with absolute continuity relations. Let $F : \D(\R_+,\R) \to \R $ be a bounded continuous real function. For conciseness, we set for all $\ell \ge 1$, for all $s,u\ge 0$,
	$$
	g_\ell(s,u) = \ell^a\P_\ell^{(1)} \left( \tau_{-\ell}=\lfloor \ell^{a-1} s \rfloor \text{ and } P_\ell(\lfloor \ell^{a-1} s \rfloor -1) = \lfloor \ell u \rfloor \right), \ f_\ell(s)  = \ell^{a-1} \P_\ell^{(1)} \left( \tau_{-\ell}=\lfloor \ell^{a-1} s \rfloor \right)
	$$
	and
	$$
	\tilde{f}_{\ell,s}(u) = \ell \P_1\left(\left. S_{\lfloor \ell^{a-1} s \rfloor -1} = \lfloor \ell u \rfloor \right| S_1,\ldots,S_{\lfloor \ell^{a-1} s \rfloor -1}\ge 1\right),
	$$
	where we recall that $(S_i)_{i\ge 0}$ is a $\nu$-random walk.  We also write for all $s\ge 0, \ell \ge 1, t\ge 0$,
	$$S^{s,\ell}(t) = \left\{
	\begin{array}{ll}
		{S_{\lfloor\ell^{a-1} t \rfloor}}/{\ell} & \mbox{if } t<s, \\
		-1 & \mbox{otherwise.}
	\end{array}
	\right.
	$$
	Then, by the expression of the law of $P_\ell$ under $\P^{(1)}_\ell$, we get
	\begin{align*}
		\E^{(1)}_\ell &F\left(\left(\frac{P_\ell(\ell^{a-1} t)}{\ell}\right)_{t\ge 0}\right)\\
		&=
		\sum_{k,n\ge 1} \P_\ell^{(1)} \left( \tau_{-\ell}=k, P_\ell(k-1)=n\right) \E_1 \left(\left. F\left(S^{k/\ell^{a-1}, \ell}\right) \right| S_1, \ldots, S_{k-1} \ge 1, \ S_{k-1}=n\right) \\
		&=
		\int_0^{\infty} ds \sum_{n \ge 1} \frac{1}{\ell}g_\ell(s, \frac{n}{\ell}) \E_1 \left(\left. F\left(S^{s, \ell}\right) \right| S_1, \ldots, S_{\lfloor \ell^{a-1} s \rfloor-1} \ge 1, \ S_{\lfloor \ell^{a-1} s \rfloor-1}=n\right) \\
		&=
		\int_0^{\infty}ds f_\ell(s) \E_1 \left( \left.\frac{g_\ell(s,S^{s, \ell}(s-1/\ell^{a-1}) )}{f_\ell(s)\tilde{f}_{\ell,s}(S^{s, \ell}(s-1/\ell^{a-1}))}F\left(S^{s, \ell}\right) \right| S_1, \ldots, S_{\lfloor \ell^{a-1} s \rfloor-1} \ge 1\right).
	\end{align*}
	The factor inside the expectation can be computed as follows: for all $k,n \ge 1$,
	\begin{align*}
		\frac{g_\ell\left({k}/{\ell^{a-1}}, {n}/{\ell}\right)}{f_\ell(k/\ell^{a-1})\tilde{f}_{\ell,k/\ell^{a-1}}(n/\ell)}
		&=
		\frac{\frac{2(\ell+1)}{h^\downarrow(\ell)} \P_1(S_1, \ldots, S_{k-1}\ge 1 \text{ and } S_{k-1}=n) \nu(-\ell-n)}{\P^{(1)}_\ell(\tau_{-\ell}=k)\P_1(S_{k-1}= n | S_1, \ldots, S_{k-1} \ge 1)} \\
		&=\frac{2(\ell+1)\nu(-\ell-n) \P_1(S_1,\ldots,S_{k-1}\ge1)}{h^\downarrow(\ell) \P_\ell^{(1)}(\tau_{-\ell}=k)} \\
		&=\frac{\nu(-\ell-n) \P_1(S_1,\ldots,S_{k-1}\ge1)}{\P_1(S_1,\ldots,S_{k-1}\ge 1 \text{ and } S_k= -\ell)} \\
		&= \left({\sum_{m\ge 1} \frac{\nu(-\ell-m)}{\nu(-\ell-n)} \P_1(S_{k-1}=m | S_1, \ldots , S_{k-1}\ge1)}\right)^{-1} \\
		&=\left({\E_1 \left( \left.\frac{\nu(-\ell-S_{k-1})}{\nu(-\ell-n)} \right| S_1, \ldots , S_{k-1}\ge1 \right) }\right)^{-1}.
	\end{align*}
	Using Theorem \ref{cvméandre}, is easy to see that if $s,u>0$, then
	$$
	\E_1 \left( \left.\frac{\nu(-\ell-S_{\lfloor \ell^{a-1} s \rfloor-1})}{\nu(-\ell-\lfloor \ell u \rfloor)} \right| S_1, \ldots , S_{\lfloor \ell^{a-1} s \rfloor-1}\ge1 \right) 
	\mathop{\longrightarrow}\limits_{\ell \to \infty} \int_0^{\infty} \left( \frac{1+u}{1+v} \right)^a \tilde{f}_a\left( \frac{v}{(p_{\bf q} s)^{1/(a-1)}}\right) dv ,
	$$
	where $\tilde{f}_a$ is the density of $\widetilde{\Upsilon}_a(1)$. Indeed, by (\ref{condition queue}), 
	$$
	\sup_{n,m\ge 1} \left| \frac{\nu(-\ell-m)}{\nu(-\ell-n)} - \left( \frac{\ell+n}{\ell+m} \right)^a\right| \mathop{\longrightarrow}\limits_{\ell \to \infty} 0.
	$$
	Furthermore, since there exists a constant $C_1>0$ such that for all $x\ge 0$, $f_a^\uparrow (x) = C_1 x^{1/2} \tilde{f}_a(x)$, we have
	\begin{align*}
		\int_0^{\infty} \left( \frac{1+u}{1+v} \right)^a \tilde{f}_a\left( \frac{v}{(p_{\bf q} s)^{1/(a-1)}}\right) dv  
		&=\frac{1}{C_1} ({p_{\bf q} s})^{1/(2(a-1))} (1+u)^a \int_0^{\infty} \frac{1}{\sqrt{v} (1+v)^a} f_a^\uparrow\left( \frac{v}{(p_{\bf q}s)^{1/(a-1)}}\right) dv \\
		&= \frac{({p_{\bf q} s})^{1/(2(a-1))} (1+u)^a}{C_1} \frac{(p_{\bf q}s)^{1/(a-1)}}{\pi} f(p_{\bf q} s) \\
		&= \frac{f(p_{\bf q}s)\tilde{f}_a\left(\frac{u}{(p_{\bf q} s)^{1/(a-1)}}\right)}{g(p_{\bf q}s,u)}.
	\end{align*}
	Thus, again using Theorem \ref{cvméandre}, we obtain that for all $s>0$, conditionally on $S_1,\ldots, S_{\lfloor \ell^{a-1} s \rfloor} \ge 1$, the random variable
	$
	{g_\ell(s,S^{s, \ell}(s-1/\ell^{a-1}) )}/({f_\ell(s)\tilde{f}_{\ell,s}(S^{s, \ell}(s-1/\ell^{a-1}))})
	$
	converges in law towards
	$$
	\frac{1}{\int_0^{\infty} \left( \frac{1+(p_{\bf q} s)^{1/(a-1)}\widetilde{\Upsilon}_a(1)}{1+v} \right)^a \tilde{f}_a\left( \frac{v}{(p_{\bf q} s)^{1/(a-1)}}\right) dv}= \frac{g(p_{\bf q}s,(p_{\bf q} s)^{1/(a-1)}\widetilde{\Upsilon}_a(1))}{f(p_{\bf q}s)\tilde{f}_a\left(\widetilde{\Upsilon}_a(1)\right)}.
	$$
	Besides, if $s>0$, one can upper bound uniformly in $n,\ell\ge 1$
	\begin{align*}
		\frac{g_\ell\left(s, {n}/{\ell}\right)}{f_\ell(s)\tilde{f}_{\ell,s}(n/\ell)}&=
		\frac{1}{\sum_{m\ge 1} \frac{\nu(-\ell-m)}{\nu(-\ell-n)} \P_1(S_{\lfloor \ell^{a-1} s \rfloor-1}=m | S_1, \ldots , S_{\lfloor \ell^{a-1} s \rfloor-1}\ge1)} 
		\\
		&\le \frac{C}{\P_1( \left. S_{\lfloor \ell^{a-1} s \rfloor -1} \in [0,\ell] \right| S_1,\ldots,S_{\lfloor \ell^{a-1} s \rfloor -1} \ge 1)} \\
		&\le \frac{C'}{\P( (p_{\bf q} s)^{1/(a-1)} \widetilde{\Upsilon}_a(1) \in [0,1] )}  \text{ by Theorem \ref{cvméandre}},
	\end{align*}
	where $C,C'$ are constants that only depend on $\nu$.
	Therefore, by dominated convergence and using once again Theorem \ref{cvméandre}, for all $s>0$,
	\begin{align*}
		G_\ell(s)&\coloneqq\E_1 \left( \left.\frac{g_\ell(s,S^{s, \ell}(s-1/\ell^{a-1}) )}{f_\ell(s)\tilde{f}_{\ell,s}(S^{s, \ell}(s-1/\ell^{a-1}))}F\left(S^{s, \ell}\right) \right| S_1, \ldots, S_{\lfloor \ell^{a-1}s \rfloor-1} \ge 1\right) \\
		&\mathop{\longrightarrow}\limits_{\ell \to \infty}G(s)\coloneqq\E \left(
		\frac{g(p_{\bf q}s,(p_{\bf q} s)^{1/(a-1)}\widetilde{\Upsilon}_a(1))}{f(p_{\bf q}s)\tilde{f}_a\left(\widetilde{\Upsilon}_a(1)\right)}
		F\left((\widetilde{\Upsilon}_a^{p_{\bf q}s}(p_{\bf q} t))_{t\ge 0}\right)
		\right).
	\end{align*}
	Furthermore, from Proposition \ref{limite locale bivariée}, and using Scheffé's lemma, $f_\ell$ converges in $\mathrm{L}^1(\R)$ towards $p_{\bf q}f(p_{\bf q}\cdot)$. Therefore, again by dominated convergence, we obtain that for all $A>0$, 
	\begin{align*}
		\E^{(1)}_\ell \left(F\left(\left(\frac{P_\ell(\ell^{a-1} t)}{\ell}\right)_{t\ge 0}\right){\bf 1}_{\frac{\tau_{-\ell}}{\ell^{a-1}}<A} \right)
		&=\int_0^A f_\ell(s) G_\ell(s) ds \\
		&\mathop{\longrightarrow}\limits_{\ell \to \infty} \int_0^A p_{\bf q}f(p_{\bf q}s) G(s)ds\\
		&=\E \left(F\left(\left(\Upsilon^\downarrow_{-1,a}(p_{\bf q}t)\right)_{t\ge 0}\right){\bf 1}_{\tilde{\tau}<p_{\bf q}A} \right),
	\end{align*}
	where the last equality stems from the definition of $\Upsilon^\downarrow_{-1,a}$ and from the change of variables $s'=p_{\bf q} s$. But since ${\tau_{-\ell}}/{\ell^{a-1}}$ converges in law, and since $F$ is bounded, we get that
	$$
	\sup_{\ell \ge 1}\E^{(1)}_\ell \left(F\left(\left(\frac{P_\ell(\ell^{a-1} t)}{\ell}\right)_{t\ge 0}\right){\bf 1}_{\frac{\tau_{-\ell}}{\ell^{a-1}}\ge A} \right) \mathop{\longrightarrow}\limits_{A \to \infty} 0,
	$$
	hence the desired result.
\end{proof}

The process $\Upsilon^\downarrow_{-1,a}$ can be linked to the stable Lévy process conditioned to stay positive $\Upsilon_a^\uparrow$ using Lemma \ref{couplage marche} and Theorem \ref{cvpérimètre}:

\begin{corollary}\label{couplage continu}
	Let $\Upsilon^\downarrow_{-1,a}$ as in Theorem \ref{cvpérimètre}. Let $\tilde{\tau}$ be the time at which $\Upsilon^\downarrow_{-1,a}$ dies at $-1$. For all bounded measurable function $F : \D([0,t],\R) \to \R$, 
	$$
	\E\, {\bf 1}_{\tilde{\tau}>t} F\left( \left(\Upsilon^\downarrow_{-1,a}(s)\right)_{0\le s \le t}\right)
	=
	\E\, \frac{1}{1+\Upsilon_a^\uparrow(t)} F \left(\left(\Upsilon_a^\uparrow(s)\right)_{0\le s\le t}\right).
	$$
\end{corollary}
One can also define for all $\lambda>0$, the Lévy process conditioned to stay positive until it jumps and dies at $-\lambda$, denoted by $\Upsilon^\downarrow_{-\lambda,a}$. It satisfies the scaling limit:
$$
\text{Under } \P^{(1)}_{\lfloor \lambda \ell \rfloor}, \ \left( \frac{P_{\lfloor \lambda \ell \rfloor}(\ell^{a-1} t)}{\ell}\right)_{t \ge 0} \mathop{\longrightarrow}\limits_{\ell \to \infty}^{(\mathrm{d})}
\left(\Upsilon^\downarrow_{-\lambda,a}(p_{\bf q}t)\right)_{t\ge 0}.
$$
Let $\tilde{\tau}_{-\lambda,a}$ be the time at which $\Upsilon^\downarrow_{-\lambda,a}$ jumps at $-\lambda$. From the coupling $\P^{(1),n}_{\lfloor \lambda \ell \rfloor, \infty}$, we get that conditionally on $\tilde{\tau}_{-\lambda,a}> t$, the process $(\Upsilon^\downarrow_{-\lambda,a}(s))_{0\le s\le t}$ has the same law as $( \Upsilon_a^\uparrow(s))_{0\le s \le t}$. Furthermore, from the scaling limit it is easy to see that $\tilde{\tau}_{-\lambda,a}$ has the same law as $\lambda^{a-1} \tilde{\tau}$. It is thus possible to show that
$$
\Upsilon^\downarrow_{-\lambda,a} \mathop{\longrightarrow}\limits_{\lambda \to \infty}^{(\mathrm{d})} \Upsilon_a^\uparrow
$$
in $\D(\R_+,\R)$ equipped with the Skorokhod $J_1$-topology. This convergence echoes the local convergence of $P_\ell$ under $\P^{(1)}_\ell$ towards $P_\infty$ under $\P^{(1)}_\infty$ when $\ell \to \infty$.

\section{The peeling exploration ``starting from a uniform random edge''}\label{section exploration}

In this section, we recall the filled-in (edge) peeling exploration of a map with a distinguished face called the target and how this exploration enables to recover the distance from the root face to that target face. In Subsection \ref{échange racine cible}, we show that for our purpose, we can start the exploration ``from a uniform  random edge'' which will become the root while the former root face of degree $2\ell$ will become the target face.

\subsection{Filled-in exploration of a map with a target}
The edge-peeling process was introduced by Budd in \cite{B16} and the results presented here can be found in more details in Chapters III and IV of \cite{StFlour} or in \cite{BuddPeeling}. 

A map with a target face $\mathfrak{m}_\square = (\mathfrak{m},\square)$ is a (finite rooted bipartite planar) map $\mathfrak{m}$ given with a distinguished face $\square$ different from the root face $f_r$. For all $\ell,p\ge 1$, we denote by $\mathcal{M}_p^{(\ell)}$ the set of finite maps of perimeter $2\ell$ with a target face of perimeter $2p$. We set for all $\mathfrak{m}_\square \in \mathcal{M}_p^{(\ell)}$:
$$
w_{\bf q}(\mathfrak{m}_\square) = \prod_{f\in \mathrm{Faces}(\mathfrak{m})\setminus \left\{ f_r, \square\right\}} q_{\deg(f)/2}.
$$
We then define also the partition function $W_p^{(\ell)} = \sum_{\mathfrak{m}_\square \in \mathcal{M}_p^{(\ell)}} w_{\bf q}(\mathfrak{m}_\square)$. 
Analogously, a map with a target vertex $\mathfrak{m}_\bullet=(\mathfrak{m}, \bullet)$ is a (finite rooted bipartite planar) map $\mathfrak{m}$ with a distinguished vertex $\bullet$. We denote by $\mathcal{M}_0^{(\ell)}$ the set of finite maps of perimeter $2\ell$ with a target vertex. We set $w_{\bf q}(\mathfrak{m}_\bullet)= w_{\bf q}(\mathfrak{m})$ for all $\mathfrak{m}_\bullet \in\mathcal{M}_0^{(\ell)}$, and $W_0^{(\ell)} = \sum_{\mathfrak{m}_\bullet \in \mathcal{M}_0^{(\ell)}} w_{\bf q}(\mathfrak{m}_\bullet)$. The fact that $\bf q$ is admissible implies that $W_p^{(\ell)}<\infty$ for all $\ell\ge 1, p\ge 0$. Moreover, we know that (see e.g. Equation (3.9) in \cite{StFlour}) for all $p,\ell \ge 1$, 
\begin{equation}\label{fonction de partition}
	W_p^{(\ell)} = \frac{1}{2} h^\downarrow_p(\ell) c_{\bf q}^{\ell+p} \qquad \text{and} \qquad 
	W_0^{(\ell)} = h^\downarrow(\ell) c_{\bf q}^{\ell}.
\end{equation}
We write $\P^{(\ell)}_p$ the associated Boltzmann probability measure on $\mathcal{M}_p^{(\ell)}$ for all $p\ge 0$.

Let $\ell \ge 1$. By Theorem 2 of \cite{BuddPeeling}, one can define a random infinite planar map $\mathfrak{M}_\infty^{(\ell)}$ as the local limit in distribution as $p \to \infty$ of random maps of law $\P^{(\ell)}_p$. The map $\mathfrak{M}_\infty^{(\ell)}$ is almost surely locally finite and one-ended, and is called a $\bf q$-Boltzmann infinite planar random map of the plane. Its law is denoted by $\P_\infty^{(\ell)}$.

A (filled-in) peeling algorithm $\mathcal{A}$ is a function which associates to a map $\overline{\mathfrak{e}}$ with a distinguished simple face (different from the root face) an edge on the boundary $\partial\overline{\mathfrak{e}}$ of this distinguished face. Let $\mathfrak{m}_*$ be a map with a target face $\square$ (``$*=\square$''), a target vertex $\bullet$ (``$*=\bullet$'') or an infinite one-ended bipartite map (``$*=\infty$''). A filled-in peeling exploration of $\mathfrak{m}_*$ with algorithm $\mathcal{A}$ is an increasing sequence $(\overline{\mathfrak{e}}_n)_{n\ge 0}$ of sub-maps of $\mathfrak{m}$ which contain the root face and which are maps with a distinguished simple face called a hole, where by sub-map we mean that gluing a well chosen map with perimeter corresponding with that of the hole of $\overline{\mathfrak{e}}_n$ will give back the map $\mathfrak{m}_*$. This sequence of sub-maps is constructed inductively using the algorithm $\mathcal{A}$: $\overline{\mathfrak{e}}$ is the map with only two faces which are the root face and a hole of the same perimeter and then for each $n\ge 0$, the sub-map $\overline{\mathfrak{e}}_{n+1}$ is obtained from $\overline{\mathfrak{e}}_{n}$ by peeling the edge $\mathcal{A}(\overline{\mathfrak{e}}_{n})$ in $\mathfrak{m}_*$ (see Figure \ref{dessin explo remplie}). When peeling an edge, there are three cases:
\begin{enumerate}[(i)]
	\item The face in $\mathfrak{m}_*$ on the other side of $\mathcal{A}(\overline{\mathfrak{e}}_n)$ is not a face of $\overline{\mathfrak{e}}_n$ and is not the target face $\square$ when $*=\square$; then $\overline{\mathfrak{e}}_{n+1}$ is obtained by gluing that face to $\overline{\mathfrak{e}}_n$ onto the edge $\mathcal{A}(\overline{\mathfrak{e}}_n)$; we denote this case by $C_k$ where $k$ is the half-degree of the discovered face.
	\item The face in $\mathfrak{m}_*$ on the other side of $\mathcal{A}(\overline{\mathfrak{e}}_n)$ is the target face $\square$ (when $*=\square$); in that case, denoted by $C_p^{\mathrm{stop}}$ where $p$ is the half-degree of the target face, we set $\overline{\mathfrak{e}}_{n+1} = \mathfrak{m}_\square$ and the exploration stops.
	\item Or the other side of $\mathcal{A}(\overline{\mathfrak{e}}_n)$ corresponds to a face already discovered in $\overline{\mathfrak{e}}_n$ which must be on the boundary of the hole. In this case $\overline{\mathfrak{e}}_{n+1}$ is obtained by identifying the two edges in the hole of $\overline{\mathfrak{e}}_n$. This creates at most two holes. If $*\in\{\square, \bullet\}$, we fill-in the hole which does not contain the target. If $*=\infty$, we fill-in the hole containing a finite part of $\mathfrak{m}_\infty$.  This case is denoted by $G_{k,*}$ or $G_{*,k}$ where $2k$ is the degree of the hole which is filled-in, depending whether this hole is created on the left or on the right of the peeled edge. If $*=\bullet$ and $k=|\partial\overline{\mathfrak{e}}_{n}|/2-1$, then the exploration stops.
\end{enumerate}

\begin{figure}[h]
	\centering
	\hbox{
		\includegraphics[scale=1.2]{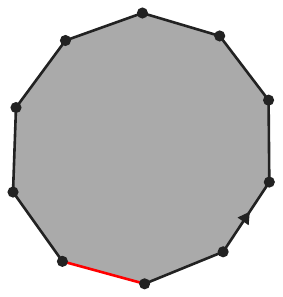}
		\includegraphics[scale=1.05]{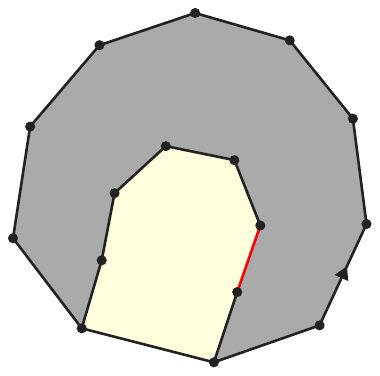}
		\includegraphics[scale=1.15]{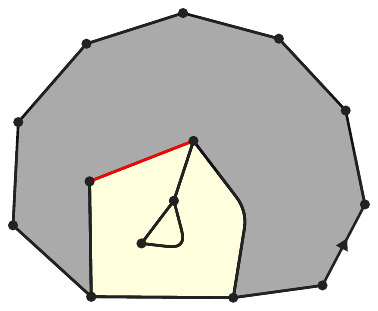}
		\includegraphics[scale=1.25]{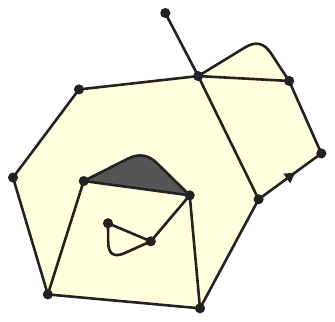}
	}
	\caption{ A filled-in exploration of a map with a target face $\mathfrak{m}_\square \in \mathcal{M}^5_1$. The peeled edge is in red, the root face is in white, the target face is in black and the hole containing the target is in grey. At the first step, an event $C_4$ happens, then an event $G_{1,*}$ occurs and finally the event $C_1^\mathrm{stop}$ ends the exploration.}
	\label{dessin explo remplie}
\end{figure}

The peeling algorithm $\mathcal{A}$ may be arbitrary, even random, but we require it to be Markovian with respect to the exploration in the sense that, conditionally on the explored region $\overline{\mathfrak{e}}_n$, the conditional law of $\mathcal{A}(\overline{\mathfrak{e}}_n)$ only depends on $\overline{\mathfrak{e}}_n$. For instance, the peeled edge may be chosen uniformly at random on the boundary of the hole as in Subsection \ref{sous-section explo uniforme} or in such a way that the dual distance from the peeled edge to the root is non-decreasing along the exploration as in Subsection \ref{sous-section explo par couche}. During the peeling exploration, one records the evolution of half of the number of edges on the boundary of the hole, which is denoted by $(P_p(n))_{n\ge 0}$ if $*\in\{\bullet,\square\}$ and $(P_\infty(n))_{n\ge 0}$ if $*=\infty$ and called the perimeter process associated with the exploration. One can express the law of the peeling process (and the law of the perimeter process) under $\P^{(\ell)}_p$ or under $\P^{(\ell)}_\infty$ using Doob $h$-transforms of the random walk $S=(S_n)_{n\ge 0}$ of step distribution $\nu$.

Now we recall the law of the filled-in peeling exploration $(\overline{\mathfrak{e}}_n)_{n\ge 0}$ under $\P^{(\ell)}_p$ or under $\P^{(\ell)}_\infty$, as described by Proposition 4.7 from \cite{StFlour} in the case $*=\square,\bullet$ and \cite{B16} in the case $*= \infty$ (see also Proposition 7.4 in \cite{StFlour}).

\begin{proposition}\label{Markov explo}(\cite{StFlour},\cite{B16})
	Let $\ell,p\ge 1$, let $\mathcal{A}$ be a peeling algorithm and $(\overline{\mathfrak{e}}_n)_{n\ge 0}$ be the peeling process under $\P_0^{(\ell)},\P^{(\ell)}_p$ or $\P^{(\ell)}_\infty$. Then $(\overline{\mathfrak{e}}_n)_{n\ge 0}$ is a Markov chain with the following transition probabilities: conditionally on $\overline{\mathfrak{e}}_n$, and $P(n)=m\ge 1$ (assuming that the exploration has not stopped yet), the events $C_k, C_p^\mathrm{stop}$(for $*=\square$)$, G_{k,*}$ and $G_{*,k}$ occur respectively with probabilities
	$$
	\nu(k-1)\frac{h(m+k-1)}{h(m)}, 
	\
	\nu(-m-p)\frac{h(-p)}{h(m)},
	\
	\frac{1}{2}\nu(-k-1)\frac{h(m-k-1)}{h(m)}
	\text{ and } 
	\frac{1}{2}\nu(-k-1)\frac{h(m-k-1)}{h(m)},
	$$
	where $h=h^\downarrow_p$ for $*=\square$, $h=h^\downarrow$ for $*=\bullet$ and $h=h^\uparrow$ for $*=\infty$.
	Moreover, conditionally on each event, the holes filled-in during the exploration are independent finite $\bf q$-Boltzmann maps (without target) of corresponding perimeters.
\end{proposition}

As a consequence, for all $p\ge 0$, the perimeter process under $\P_p^{(\ell)}$ (resp. under $\P_\infty^{(\ell)}$) indeed corresponds to the Doob $h^\downarrow_p$-transform killed when it reaches $-p$ (resp. Doob $h^\uparrow$-transform) of the walk $S$ started from $\ell$, if we set by convention that the perimeter drops to $-p$ when the exploration stops.

\subsection{Uniform peeling and first-passage percolation distance}\label{sous-section explo uniforme}
It is possible to use the peeling process to study the first-passage percolation distance on a map with a target $\mathfrak{m}_*$, or an infinite map of the plane. A very convenient algorithm is the uniform peeling algorithm $\mathcal{A}_\mathrm{uniform}$ which selects an edge on the boundary of the hole uniformly at random at every step. We refer to \cite{StFlour} 13.1 or \cite{BC} 2.4 for details.

This exploration can be related to the fpp distance on $\mathfrak{m}_*$ in the following way: for all $t\ge 0$, let $\overline{\mathrm{Ball}}^\mathrm{fpp}_t (\mathfrak{m}_*)$ be the sub-map of $\mathfrak{m}_*$ obtained by keeping the connected subset of dual edges whose endpoints are at fpp distance at most $t$ from the root face, then by gluing the corresponding faces according to those dual edges and finally by filling in the holes that do not contain the target (the finite holes in the case $*=\infty$). Then the process $(\overline{\mathrm{Ball}}^\mathrm{fpp}_t (\mathfrak{m}_*))_{t\ge 0}$ admits jumps $0=T_0<T_1<\cdots$. One can check (see Proposition 2.3 in \cite{BC}) that for all $n\ge 0$, the sub-map
$\overline{\mathrm{Ball}}^\mathrm{fpp}_{T_{n+1}} (\mathfrak{m})$ can be obtained from $\overline{\mathrm{Ball}}^\mathrm{fpp}_{T_n} (\mathfrak{m}_*)$ by the peeling of a uniform random edge on the boundary of the hole. Furthermore, conditionally on $(\overline{\mathrm{Ball}}^\mathrm{fpp}_{T_n} (\mathfrak{m}_*))_{n\ge 0}$, the random variables $T_{n+1}-T_n$ are independent and follow exponential laws of parameter $|\partial \overline{\mathrm{Ball}}^\mathrm{fpp}_{T_n} (\mathfrak{m}_*)|$ where $\partial \overline{\mathrm{Ball}}^\mathrm{fpp}_{T_n} (\mathfrak{m}_*)$ is the set of edges on the boundary of $\overline{\mathrm{Ball}}^\mathrm{fpp}_{T_n} (\mathfrak{m}_*)$.

From the above results, by coupling the uniform exploration with the exponential variables that define the fpp distance, the fpp distance from the root face to the submap $\mathfrak{u}$ filling the hole of boundary $\partial \overline{\mathfrak{e}}_n$ can be written as
\begin{equation}\label{distance fpp à l'arête épluchée}
	d_\mathrm{fpp}^\dagger (f_r,\mathfrak{u}) = 
	T_{n+1} = \sum_{i=0}^{n} \frac{\mathcal{E}_i}{|\partial \overline{\mathrm{Ball}}^\mathrm{fpp}_{T_i} (\mathfrak{m}_*)|}
	= \sum_{i=0}^{n} \frac{\mathcal{E}_i}{2P(i)},
\end{equation}
where $P(i) =|\partial \overline{\mathrm{Ball}}^\mathrm{fpp}_{T_i} (\mathfrak{m})|/2$ is the perimeter process associated with the exploration and the $\mathcal{E}_i$ for $i\ge 0$ are i.i.d. exponential random variables of parameter $1$ which are independent from the exploration $(\overline{\mathfrak{e}}_n)_{n\ge 0}$, and a fortiori from $(P(n))_{n\ge 0}$. In particular, if $\mathfrak{m}_\square \in \mathcal{M}^{(\ell)}_p$ for $\ell, p\ge 1$ and $n$ is the time at which the target face $\square$ is discovered, i.e. the event $C^\mathrm{target}_p$ happens at time $n$, or equivalently $\overline{\mathfrak{e}}_n=\mathfrak{m}_\square$, then $d_\mathrm{fpp}^\dagger(f_r,\square)=T_n$.

\subsection{Peeling by layers and dual graph distance}\label{sous-section explo par couche}

In the same way, a particular peeling algorithm $\mathcal{A}_\mathrm{layers}$ called the peeling by layers algorithm is suited to the study of the dual graph distance on a map with a target $\mathfrak{m}_*$ (or an infinite map of the plane). Recall that $d^\dagger_\mathrm{gr}$ denotes the dual graph distance on a map $\mathfrak{m}$. If $f$ is a face of $\mathfrak{m}$ the distance to the root face $d^\dagger_\mathrm{gr}(f_r,f)$ is also called the height of $f$. We can also define the height of an edge $e$ as the minimum of the heights of the faces on both sides of $e$.

The peeling algorithm $\mathcal{A}_\mathrm{layers}$, described in \cite{BC} 2.3, see also \cite{StFlour} p.185, is defined only on (finite) sub-maps $\overline{\mathfrak{e}}$ of $\mathfrak{m}_*$ with a hole (containing the target when $*=\square$) which satisfy the following hypothesis:
\begin{enumerate}[({\bf H})] 
	\item \label{hypothèse couche} There exists an integer $h\ge 0$ such that all the edges on the boundary of the hole $\partial \overline{\mathfrak{e}}$ are at height $h$ or $h+1$, and such that the set of edges of $\partial \overline{\mathfrak{e}}$ at heigth $h$ forms a (non empty) segment. 
\end{enumerate}
If $\overline{\mathfrak{e}}$ satisfies ({\bf H}), we then set $\mathcal{A}_\mathrm{layers}(\overline{\mathfrak{e}})$ to be the unique edge of $\partial \overline{\mathfrak{e}}$ at height $h$ such that the edge immediately on its left is at height $h+1$. If all the edges of $\partial \overline{\mathfrak{e}}$ are at height $h$ the edge $\mathcal{A}_\mathrm{layers}(\overline{\mathfrak{e}})$ is then chosen deterministically.

One can check that on the events $C_k$, $G_{*,k}$ and $G_{k,*}$ the sub-map obtained after peeling $\mathcal{A}_\mathrm{layers}(\overline{\mathfrak{e}})$ satisfies ({\bf H}) again. Moreover, $\overline{\mathfrak{e}}_0$ satisfies ({\bf H}) as well. Thus the peeling exploration with algorithm $\mathcal{A}_\mathrm{layers}$ is well defined.

For every $n \ge 0$, let $H(n)$ be the minimum of the heights of the edges of $\partial \overline{\mathfrak{e}}_n$. Then $H(0)=0$ and for all $n\ge 0$, we have $\Delta H(n)\coloneqq H(n+1)-H(n) \in \{0,1\}$. One has $\Delta H(n)=1$ when all the edges of $\partial \overline{\mathfrak{e}}_{n+1}$ are at height $H(n)+1$. The process $H(n)$ is called the height process associated with the peeling by layers exploration. The integer $H(n)+1$ can also be seen as the height of the hole in $\overline{\mathfrak{e}}_n$. In particular, if $\mathfrak{m}_\square \in \mathcal{M}^{(\ell)}_p$ for some $\ell,p\ge1$ and $n$ is the first time such that $\overline{\mathfrak{e}}_n=\mathfrak{m}_\square$, then $d_\mathrm{gr}^\dagger(f_r,\square)=H(n-1)+1$.

\subsection{Exchanging the root and the target}\label{échange racine cible}

\begin{figure}[h]
	\centering
	\hbox{
		\includegraphics[scale=1.4]{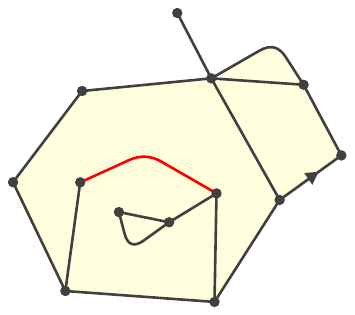}
		\includegraphics[scale=1.55]{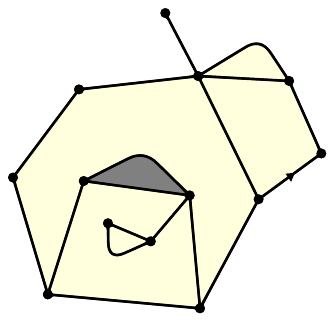}
		\includegraphics[scale=1.4]{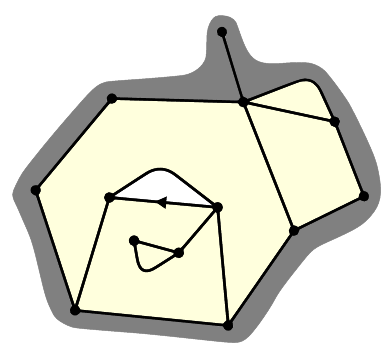}
	}
	\caption{From left to right: a map with a distinguished edge in red, the same map with the edge ``unzipped'' giving rise to a distinguished $2$-face (in grey) and finally the same map which has been re-rooted.}
	\label{dessin échange racine cible}
\end{figure}

In order to study the distance from the root face to a uniform random face of $\mathfrak{M}^{(\ell)}$, we will first look at the distance between the root face and a uniform random edge $E$ in $\mathfrak{M}^{(\ell)}$. We will then come back to a uniform random face using a biasing argument in Subsection \ref{uvufue}. This distance is defined as the minimum of the distances from the root face to a face next to this edge. 
By unzipping the uniform random edge and thus creating a distinguished face of degree $2$, one can see that the map obtained from $\mathfrak{M}^{(\ell)}$ belongs to $\mathcal{M}_1^{(\ell)}$. We denote this map by $\widetilde{\mathfrak{M}}_1^{(\ell)}$.

\begin{remark}\label{petite modif des distances}
	This unzipping operation does not modify the distances so much. 
	More precisely, the graph distance $d^\dagger_{\mathrm{gr}} (f_r,\square)$ in $\widetilde{\mathfrak{M}}_1^{(\ell)}$ equals $d^\dagger_{\mathrm{gr}} (f_r,E)+1$ in $\mathfrak{M}^{(\ell)}$ while $d^\dagger_{\mathrm{fpp}} (f_r,\square)$ in $\widetilde{\mathfrak{M}}_1^{(\ell)}$ is lower bounded by $d^\dagger_{\mathrm{fpp}} (f_r,E)$ and is upper bounded by $d^\dagger_{\mathrm{fpp}} (f_r,E)+\mathcal{E}$ in $\mathfrak{M}^{(\ell)}$ where $\mathcal{E}$ is an independent exponential variable of parameter $1$.
\end{remark}

Besides, one can notice that exchanging the roles of the root face and of the target face by choosing a root edge on the former target face does not modify the distance between the root face and the target face. 
All the operations we did are illustrated by Figure \ref{dessin échange racine cible}.  This swap between the root face and the target face will be useful insofar as in the next subsection we will introduce a coupling of $\P_\ell^{(1)}$ with $\P^{(1)}_\infty$ that will enable us to use the results of \cite{BCM}. Let $\widetilde{\mathfrak{M}}_\ell^{(1)}$ be the map obtained from $\widetilde{\mathfrak{M}}_1^{(\ell)}$ by exchanging the root face and the target face and choosing the root edge uniformly at random around the new root face. 

So as to transfer our results under $\P^{(1)}_\ell$ to the map $\mathfrak{M}^{(\ell)}$ equipped with a random uniform edge $E$, we state the following lemma. 

\begin{lemma}\label{lemme transfert biais}
	If $(\varphi_\ell)_{\ell \ge 1}$ is a sequence of functions from $\mathcal{M}^{(1)}_\ell$ to $\R_+$ uniformly bounded by a constant $C>0$  
	such that
	$$
	\E^{(1)}_\ell \varphi_\ell(\mathfrak{m}_\square) \mathop{\longrightarrow}\limits_{\ell \to \infty} 0,
	$$
	then
	$$
	\E \varphi_\ell\left(\widetilde{\mathfrak{M}}_\ell^{(1)}\right) \mathop{\longrightarrow}\limits_{\ell \to \infty} 0.
	$$
\end{lemma}
In particular, if $(A_\ell)_{\ell \ge 1}$ is a sequence of subsets of $\mathcal{M}^{(1)}_\ell$ 
such that
$
\P^{(1)}_\ell (A_\ell) \to 0
$ as $\ell \to \infty$, 
then
$
\P(\widetilde{\mathfrak{M}}_\ell^{(1)} \in A_\ell) \to 0
$.

\begin{proof}
	Observe first that the law of $\widetilde{\mathfrak{M}}_1^{(\ell)}$ can be characterized using $\P^{(\ell)}_1$. Indeed, if $\mathfrak{m}_\square \in \mathcal{M}_1^{(\ell)}$ is obtained by unzipping an edge of a map $\mathfrak{m}' \in \mathcal{M}^{(\ell)}$, then
	\begin{align*}
		\P( \widetilde{\mathfrak{M}}_1^{(\ell)}  = \mathfrak{m}_\square )
		&=
		\frac{1}{\# \mathrm{Edges} (\mathfrak{m}')}\P( \mathfrak{M}^{(\ell)}  = \mathfrak{m}')
		=
		\frac{1}{\# \mathrm{Edges} (\mathfrak{m}')}\frac{w_{\bf q}(\mathfrak{m}')}{W^{(\ell)}}
		\\
		&=
		\frac{1}{\# \mathrm{Edges} (\mathfrak{m}_\square)-1}\frac{W^{(\ell)}_1}{W^{(\ell)}}
		\frac{w_{\bf q}(\mathfrak{m}_\square)}{W^{(\ell)}_1}
		=
		\frac{W^{(\ell)}_1}{W^{(\ell)}} 
		\frac{1}{\# \mathrm{Edges} (\mathfrak{m}_\square)-1}
		\P^{(\ell)}_1(\mathfrak{m}_\square)
		.
	\end{align*}
	Another well known and useful consequence of this unzipping operation is that
	\begin{equation}\label{EqVolume}
		\E \, \# \mathrm{Edges}(\mathfrak{M}^{(\ell)}) 
		=\frac{W_1^{(\ell)}}{W^{(\ell)}}
		\mathop{\sim}\limits_{\ell \to \infty} \frac{\frac{1}{2}h^\downarrow_1(\ell) c_{\bf q}^{\ell+1}}{\frac{p_{\bf q}}{2}c_{\bf q}^{\ell+1} \ell^{-2}}
		\mathop{\sim}\limits_{\ell \to \infty} b_{\bf q} \ell^{3/2} \text{ where } b_{\bf q} = \frac{1}{2p_{\bf q} \sqrt{\pi}},
	\end{equation}
	applying successively (\ref{criticité de type 2}), (\ref{fonction de partition}), (\ref{h flèche}) and the fact that $h^\downarrow(\ell) \sim 1/\sqrt{\pi \ell}$.
	
	Next, the law of $\widetilde{\mathfrak{M}}_\ell^{(1)}$ can be expressed as follows. Let $\mathfrak{m}_\square \in \mathcal{M}^{(1)}_\ell$ which is obtained by re-rooting a map $\mathfrak{m}'_\square \in \mathcal{M}^{(\ell)}_1$. Then one can check that
	\begin{equation}\label{loi réenraciné}
		\P\left(\widetilde{\mathfrak{M}}_\ell^{(1)} = \mathfrak{m}_\square
		\right)
		=
		\frac{2\ell}{2} \P\left(\widetilde{\mathfrak{M}}_1^{(\ell)} = \mathfrak{m}'_\square \right)
		=
		\frac{2\ell}{2}
		\frac{W^{(\ell)}_1}{W^{(\ell)}} 
		\frac{\P^{(\ell)}_1(\mathfrak{m}'_\square)}{\# \mathrm{Edges} (\mathfrak{m}'_\square)-1}
		=
		\frac{W^{(\ell)}_1}{W^{(\ell)}} 
		\frac{\P^{(1)}_\ell(\mathfrak{m}_\square)}{\# \mathrm{Edges} (\mathfrak{m}_\square)-1}
		.
	\end{equation}
	We deduce that
	\begin{equation}\label{eq preuve lemme racine}
		\E \varphi_\ell\left(\widetilde{\mathfrak{M}}_\ell^{(1)} \right)
		=
		\sum_{\mathfrak{m}_\square \in \mathcal{M}^{(1)}_\ell} \frac{W^{(\ell)}_1}{W^{(\ell)}} 
		\frac{\P^{(1)}_\ell(\mathfrak{m}_\square)}{\# \mathrm{Edges} (\mathfrak{m}_\square)-1}\varphi_\ell(\mathfrak{m}_\square )
		= \frac{W^{(\ell)}_1}{W^{(\ell)}}  \E^{(1)}_\ell\left(\frac{\varphi_\ell(\mathfrak{m}_\square)}{\# \mathrm{Edges} (\mathfrak{m}_\square)-1} \right).
	\end{equation}
	Let $\vp >0$. We next distinguish whether the number of edges is larger than $\vp \ell^{3/2}+1$ or not. On the one hand, using (\ref{EqVolume})
	, one can see that
	$$
	\frac{W^{(\ell)}_1}{W^{(\ell)}}  \E^{(1)}_\ell\left(\frac{\varphi_\ell(\mathfrak{m}_\square)}{\# \mathrm{Edges} (\mathfrak{m}_\square)-1} {\bf 1}_{\#\mathrm{Edges}(\mathfrak{m}_\square)\ge \vp \ell^{3/2}+1} \right)
	\le \frac{W^{(\ell)}_1}{\vp \ell^{3/2} W^{(\ell)}} \E^{(1)}_\ell \varphi_\ell(\mathfrak{m}_\square)
	\mathop{\longrightarrow}\limits_{\ell \to \infty} 0.
	$$
	On the other hand,
	\begin{align*}
		\frac{W^{(\ell)}_1}{W^{(\ell)}}  \E^{(1)}_\ell\left(\frac{\varphi_\ell(\mathfrak{m}_\square){\bf 1}_{ \#\mathrm{Edges}(\mathfrak{m}_\square)\le \vp \ell^{3/2}}}{\# \mathrm{Edges} (\mathfrak{m}_\square)-1}  \right)
		&\le
		\frac{W^{(\ell)}_1}{W^{(\ell)}}  \E^{(1)}_\ell\left(C \frac{{\bf 1}_{\#\mathrm{Edges}(\mathfrak{m}_\square)\le \vp \ell^{3/2}}}{\# \mathrm{Edges} (\mathfrak{m}_\square)-1}  \right)\\
		&=C \P\left(\#\mathrm{Edges}\left(\widetilde{\mathfrak{M}}_\ell^{(1)}\right) \le \vp \ell^{3/2}\right)  \text{\enskip by (\ref{eq preuve lemme racine}) }\\
		&=C \P(\#\mathrm{Edges}(\mathfrak{M}^{(\ell)}) \le \vp \ell^{3/2}-1).
	\end{align*}
	Furthermore,
	$$
	\lim_{\vp \to 0} \limsup_{\ell\to \infty}\P^{(\ell)}(\#\mathrm{Edges}(\mathfrak{m})\le \vp \ell^{3/2}-1) =0,
	$$
	since $\# \mathrm{Edges}(\mathfrak{M}^{(\ell)}) \ell^{-3/2}$ converges in distribution towards a positive random variable as $\ell \to \infty$, see e.g. Proposition 10.4 in \cite{StFlour}.
\end{proof}

\section{Scaling limit of the distance to a target}\label{section distance racine face}

This section and the rest of this paper are devoted to the study of the dual distances on the $3/2$-stable map $\mathfrak{M}^{(\ell)}$ of perimeter $2\ell$. We thus henceforth assume that $\bf q$ is non-generic critical of type $a=2$. We will drop the subscript $2$ in $\Upsilon_2$, $\Upsilon^\uparrow_2$, $\Upsilon^\downarrow_2$ and $\Upsilon^\downarrow_{-1,2}$. 

As announced before, we first study the distance from the root face to a uniform random edge of $\mathfrak{M}^{(\ell)}$ and then transpose the result to the distance to a uniform random face and to a uniform random vertex. From Subsection \ref{échange racine cible}, it suffices to study the distance from the root face $f_r$ to the target face $\square$ under $\P^{(1)}_\ell$. Thus, we may use the filled-in exploration of $\P^{(1)}_\ell$ presented in Section \ref{section exploration} with in particular the peeling algorithms $\mathcal{A}_\mathrm{uniform}$ and $\mathcal{A}_\mathrm{layers}$.
\subsection{The first-passage percolation distance to a uniform random edge in $\mathfrak{M}^{(\ell)}$}

From Subsection \ref{sous-section explo uniforme}, in particular (\ref{distance fpp à l'arête épluchée}), we know that the fpp distance from the root face $f_r$ to the target face $\square$ under $\P_\ell^{(1)}$ is
\begin{equation}\label{distance fpp à la racine}
	d_\text{fpp}^\dagger (f_r, \square) = \sum_{k=0}^{\tau_{-\ell}-1} \frac{\mathcal{E}_k }{2P_\ell (k)},
\end{equation}
where the $\mathcal{E}_k$'s are i.i.d. exponential random variables of parameter one independent of the perimeter process.
The properties of the perimeter process proved in Section \ref{section périmètre} enable us to state the following proposition, taking us closer to the first convergence of Theorem \ref{distance de la racine à une face uniforme}.
\begin{proposition}\label{distance fpp arete uniforme}
	Under $\P_\ell^{(1)}$,
	$$\frac{d_\mathrm{fpp}^\dagger (f_r, \square)  }{ \log \ell } \mathop{\longrightarrow}\limits_{\ell \to \infty}^{(\P)} \frac{1 }{ \pi^2 p_{\bf q}}.$$
	As a consequence, if $E^{(\ell)}$ is a uniform edge in $\mathfrak{M}^{(\ell)}$, then
	$$\frac{d_\mathrm{fpp}^\dagger (f_r, E^{(\ell)})  }{ \log \ell } \mathop{\longrightarrow}\limits_{\ell \to \infty}^{(\P)} \frac{1 }{ \pi^2 p_{\bf q}}.$$
\end{proposition}

\begin{proof}
	For the first convergence, it suffices to prove that 
	\begin{equation}\label{eq sans exp}
		\frac{1}{\log \ell}\sum_{k=0}^{\tau_{-\ell}-1} \frac{1 }{2P_\ell (k)} \mathop{\longrightarrow}\limits_{\ell \to \infty}^{(\P)} \frac{1}{ \pi^2 p_{\bf q}}.
	\end{equation}
	Indeed, assuming (\ref{eq sans exp}), like in the beginning of the proof of Proposition 3 of \cite{BCM}, we can reason conditionally on $P_\ell$:
	\begin{align*}
		\E \left(\left.\left(\frac{1}{ \log \ell}\sum_{k=0}^{\tau_{-\ell}-1} \frac{\mathcal{E}_k-1}{2P_\ell (k)} \right)^2\right| P_\ell\right)
		=\frac{1}{(\log \ell)^2} \sum_{k=0}^{\tau_{-\ell}-1} \frac{1 }{4P_\ell (k)^2} 
		\le \frac{1}{ (\log \ell)^2} \sum_{k=0}^{\tau_{-\ell}-1} \frac{1 }{ 4P_\ell (k)} \mathop{\longrightarrow}\limits_{\ell \to \infty}^{(\P)} 0,
	\end{align*}
	where the convergence comes from (\ref{eq sans exp}). 
	In order to prove (\ref{eq sans exp}), one can first see that by Theorem \ref{cvpérimètre}, for all $\vp >0$,
	\begin{equation}\label{intégrale de epsilon}
		\sum_{k=\lfloor \vp \ell \rfloor}^{\tau_{-\ell}-1} \frac{1 }{2P_\ell (k)} \mathop{\longrightarrow}\limits_{\ell \to \infty}^{(\mathrm{d})}
		\int_{\vp}^{\tilde{\tau}} \frac{1}{2\Upsilon^\downarrow_{-1}(p_{\bf q} t) }dt.
	\end{equation}
	Besides, under the coupling $\P^{(1), \lfloor \vp \ell \rfloor}_{\ell,\infty}$ of Lemma \ref{couplage marche}, on the event $\tau_{-\ell} > \lfloor \vp \ell \rfloor$ (which holds with probability greater than $1-o(1)$ uniformly in $\ell$ as $\vp \to 0$), for all $k \le \lfloor \vp \ell \rfloor$ we have $P_\ell(k) = P_\infty(k)$. Thus, on the event $\tau_{-\ell} > \lfloor \vp \ell \rfloor$, we have
	$$
	{1\over \log \ell}\sum_{k=0}^{\lfloor \vp \ell \rfloor} {1 \over 2P_\ell (k)} = {1\over \log \ell}\sum_{k=0}^{\lfloor \vp \ell \rfloor} {1 \over 2P_\infty (k)}\mathop{\longrightarrow}\limits_{\ell \to \infty}^{(\P)} {1\over \pi^2 p_{\bf q}},
	$$
	by Equation (15) of \cite{BCM}. Consequently, combining this convergence with (\ref{intégrale de epsilon}), we can conclude that (\ref{eq sans exp}) holds, which ends the proof of the first convergence.
	
	\noindent The first convergence of the statement is immediately transferred to the map $\mathfrak{M}^{(\ell)}$ equipped with a uniform edge by Remark \ref{petite modif des distances} and by Lemma \ref{lemme transfert biais}, using the function $\varphi_\ell$ defined by
	$$\forall \mathfrak{m}_\square \in \mathcal{M}^{(1)}_\ell, \qquad \varphi_\ell(\mathfrak{m}_\square)=\E\left(\left|\frac{d_\mathrm{fpp}^\dagger (f_r, \square) }{ \log \ell }
	-\frac{1}{\pi^2 p_{\bf q}}
	\right|\wedge 1\right),$$
	where the expectation applies to the exponential weights defining the fpp distance.
\end{proof}

\subsection{The dual graph distance to a uniform random edge in $\mathfrak{M}^{(\ell)}$}

The height process along the filled-in exploration using the peeling by layers algorithm of a map under $\P_\ell^{(1)}$ is denoted by $H_\ell$. We recall that $H_\ell(n)$ is the smallest height of an edge in $\partial \overline{\mathfrak{e}}_n$. From Subsection \ref{sous-section explo par couche}, we know that the dual graph distance from the root face $f_r$ to the target face $\square$ under $\P_\ell^{(1)}$ is
$
d_\mathrm{gr}^\dagger (f_r, \square) = H_\ell(\tau_{-\ell} -1 )+1
$.

Now we can state and prove the main result of this section, which gives the asymptotic behavior of the dual graph distance from the root face to the target face under $\P^{(1)}_\ell$, and brings us closer to the second convergence of Theorem \ref{distance de la racine à une face uniforme} by the arguments of Subsection \ref{échange racine cible}.

\begin{proposition}\label{distanceàunsommetuniforme}
	Under $\P^{(1)}_\ell$, 
	$$
	{d_\mathrm{gr}^\dagger (f_r, \square) \over (\log \ell)^2} \mathop{\longrightarrow}\limits_{\ell \to \infty}^{(\P)} {1\over 2 \pi^2}.
	$$
	As a consequence, if $E^{(\ell)}$ is a uniform random edge of $\mathfrak{M}^{(\ell)}$, then, in $\mathfrak{M}^{(\ell)}$,
	$$
	{d_\mathrm{gr}^\dagger (f_r, E^{(\ell)}) \over (\log \ell)^2} \mathop{\longrightarrow}\limits_{\ell \to \infty}^{(\P)} {1\over 2 \pi^2}.
	$$
\end{proposition}

\begin{proof}
	First of all, from Proposition \ref{Markov explo}, it is easy to see that for all $n\ge 1$, for any increasing sequence of maps with holes $\mathfrak{u}_0\subset \mathfrak{u}_1 \subset \cdots \subset \mathfrak{u}_{n-1} $,
	\begin{align*}
		\P_\ell^{(1)}&(\mathfrak{e}_0=\mathfrak{u}_0,\ldots, \mathfrak{e}_{n-1} = \mathfrak{u}_{n-1} |\tau_{-\ell}=n \text{ and } P_\ell({n-1}) = |\partial\mathfrak{u}_{n-1}|/2) \\
		&= \P_\infty^{(1)}(\mathfrak{e}_0=\mathfrak{u}_0,\ldots, \mathfrak{e}_{n-1} = \mathfrak{u}_{n-1} |P_\infty(n-1) = |\partial\mathfrak{u}_{n-1}|/2).
	\end{align*}
	As a result, if $F: \R \to \R$ is a continuous bounded function,
	\begin{align*}
		\E^{(1)}_\ell F\left( \frac{H_\ell(\tau_{-\ell}-1)}{(\log \ell)^2}\right)
		&= \sum_{n,k\ge 1} \E_\infty^{(1)}\left(\left.F\left( \frac{H_\infty(\tau_{-\ell}-1)}{(\log \ell)^2}\right) \right| P_\infty(n-1)=k\right) \P_\ell^{(1)}\left(\tau_{-\ell}= n, P_\ell(n-1)=k\right)
	\end{align*}
	We take $F: x \mapsto |x-1/(2\pi^2)|\wedge 1$. Let $\vp>0$. Let $\delta>0$ be small enough so that for all $\ell \ge 1$, under $\P^{(1)}_\ell$, we have $\tau_{-\ell}/\ell \in [\delta, 1/\delta]$ and $ P_\ell(\tau_{-\ell}-1)/\ell \in [\delta,1/\delta]$ with probability at least $1-\vp$. Such a $\delta$ exists thanks to Proposition \ref{limite locale bivariée}.
	Then, by Proposition \ref{limite locale bivariée} and Lemma \ref{lemme3BCM}, uniformly on $s,u\in [\delta,1/\delta]$ as $\ell \to \infty$,
	\begin{align*}
		\frac{\P^{(1)}_\ell(\tau_{-\ell} = \lfloor \ell s \rfloor \text{ and } P_\ell(\lfloor \ell s \rfloor -1) = \lfloor \ell u \rfloor)}
		{\P_\infty^{(1)} (P_\infty(\lfloor \ell s \rfloor -1)= \lfloor \ell u \rfloor)} 
		&=\frac{\P^{(1)}_\ell(\tau_{-\ell} = \lfloor \ell s \rfloor \text{ and } P_\ell(\lfloor \ell s \rfloor -1) = \lfloor \ell u \rfloor)}
		{(h^\uparrow(\lfloor \ell u \rfloor)/h^\uparrow(1))\P_1(S_1, \ldots, S_{\lfloor \ell s \rfloor -1}\ge1 \text{ and } S_{\lfloor \ell s \rfloor -1}= \lfloor \ell u \rfloor)}\\
		&\sim \frac{\pi p_{\bf q}}{\ell\sqrt{u} (1+u)^2},
	\end{align*}
	where we used that $h^\uparrow(\ell) \sim 2\sqrt{\ell/\pi}$. In particular, the above ratio is uniformly bounded in $\ell \ge 1$ and $u,s \in [\delta,1/\delta]$ by a constant $C(\delta)$ times $1/\ell$.
	Thus
	\begin{align*}
		\E^{(1)}_\ell F\left( \frac{H_\ell(\tau_{-\ell}-1)}{(\log \ell)^2}\right)
		&\le \vp + C(\delta)\ell^{-1}\int_\delta^{1/\delta}ds\int_\delta^{1/\delta} du \ell^2\E^{(1)}_\infty\left(F\left(\frac{H_\infty(\lfloor \ell s \rfloor-1)}{(\log \ell)^2}\right) {\bf 1}_{P_\infty(\lfloor \ell s \rfloor -1) = \lfloor \ell u \rfloor}\right) \\
		&\le \vp + C(\delta) \int_\delta^{1/\delta}ds\E^{(1)}_\infty\left(F\left(\frac{H_\infty(\lfloor \ell s \rfloor-1)}{(\log \ell)^2}\right) \right).
	\end{align*}
	The last integral tends to $0$ as $\ell \to + \infty$ because of Proposition 4 from \cite{BCM}, stating that under $\P_\infty^{(1)}$,
	$$
	{H_\infty(\ell) \over (\log \ell )^2}
	\mathop{\longrightarrow}\limits_{\ell \to \infty}^{(\P)}
	{1\over 2\pi^2}.
	$$
	This entails the first convergence.
	
	\noindent The second convergence is a straightforward consequence of the first one due to Remark \ref{petite modif des distances} and to Lemma \ref{lemme transfert biais}, using the function $\varphi_\ell$ defined by
	$$\forall \mathfrak{m}_\square \in \mathcal{M}^{(1)}_\ell, \qquad \varphi_\ell(\mathfrak{m}_\square)=\left|\frac{d_\mathrm{gr}^\dagger (f_r, \square) }{ (\log \ell)^2 }
	-\frac{1}{2\pi^2 }
	\right|\wedge 1.$$
	This concludes the proof.
\end{proof}

\subsection{Replacing the uniform random edge by a vertex or a face}\label{uvufue}

Let us explain how the results of the preceding subsections also give us the distance from the root face to a uniform random vertex or to a uniform random face, hence proving Theorem \ref{distance de la racine à une face uniforme} at the end of this subsection. This subsection relies on classical arguments involving the peeling exploration and bijections with trees which may be skipped by the reader.

We first focus on the case of a uniform random vertex. Notice that $ |\mathrm{Vertices}(\mathfrak{M}^{(\ell)})|=\sum_{\vec{e}} {1\over \deg v(\vec{e})}$, where the sum is over all the oriented edges $\vec{e}$ of $\mathfrak{M}^{(\ell)}$.
To get a uniform random vertex on $\mathfrak{M}^{(\ell)}$, conditionally on $\mathfrak{M}^{(\ell)}$, we pick a random oriented edge $\vec{E}$ of $\mathfrak{M}^{(\ell)}$ of law
\begin{equation}\label{E flèche}
	\forall \vec{e} \in \mathrm{Edges}(\mathfrak{M}^{(\ell)}), \ \ \P\left( \left.\vec{E} = \vec{e} \ \right| \mathfrak{M}^{(\ell)}\right) \coloneqq 
	{1\over \deg v(\vec{e})} {1\over |\mathrm{Vertices}(\mathfrak{M}^{(\ell)})|},
\end{equation}
where $v(\vec{e})$ is the vertex from which starts $\vec{e}$. Then for all vertex $v$ of $\mathfrak{M}^{(\ell)}$
\begin{equation}\label{eq sommet uniforme}
	\P\left(\left.v(\vec{E}) = v\right|\mathfrak{M}^{(\ell)}\right) = {1\over |\mathrm{Vertices}(\mathfrak{M}^{(\ell)})|},
\end{equation}
i.e. $V\coloneqq v(\vec{E})$ is a random vertex taken uniformly on $\mathfrak{M}^{(\ell)}$. 

Let us control the degree of 
the root vertex $v_r$ (from which starts the root edge) under $\P^{(1)}_\ell$. The lemma below is an extension of Lemma 15.6 from \cite{StFlour}.
\begin{lemma}\label{degré arête}
	There exists $c \in (0,1)$ such that for all $\ell,j \ge 1$, 
	
	$$
	\P^{(1)}_\ell (\deg v_r \ge j ) \le c^{j-2}.
	$$
\end{lemma}

\begin{proof}
	As in \cite{StFlour} we use the peeling exploration ``peeling around the root vertex'', defined as follows: as long as the root vertex belongs to the boundary of the hole $\partial \overline{\mathfrak{e}}_n$, we peel the edge on $\partial \overline{\mathfrak{e}}_n$ adjacent on the left of $v_r$ until it becomes an interior vertex. More precisely, after an event of type $G$, we fill-in the hole which does not have $v_r$ on its boundary and we continue the exploration. Thus, the hole may not contain the target after some time. We write $G_{k_1,k_2}$ the event of an identification of two edges on the boundary of a hole of perimeter $2p=2k_1+2k_2+2$ which does not contain the target, giving rise to a hole of perimeter $2k_1$ on the left of the peeled edge (which will be filled-in) and a hole of perimeter $2k_2$ on the right. This event has the transition probability $W^{(k_1)}W^{(k_2)}/W^{(p)}$. The exploration ends when $v_r$ becomes an inner vertex, i.e. after an event of type $G_{*,0}$ or after an event of type $G_{k,0}$ for some $k\ge 0$.  
	
	At one step of this peeling exploration, let $2p$ be the perimeter of the hole whose boundary contains $v_r$. We distinguish the two possible cases:
	\begin{itemize}
		\item Either $v_r$ is in the boundary of the hole containing the target face. In that case, the exploration stops if and only if the event $G_{*,0}$ occurs. And this happens with probability $\frac{1}{2} \nu(-1) \frac{h^\downarrow_\ell(p-1)}{h^\downarrow_\ell(p)}$ by Proposition \ref{Markov explo} (and necessarily $p\ge 2$ so that the hole which contains the target has positive boundary length).
		\item Either $v_r$ is in the boundary of a hole which does not contain the target face. In this case, the exploration stops when the event $G_{p-1,0}$ occurs. This event happens with probability $W^{(p-1)}W^{(0)}/W^{(p)}$.
	\end{itemize}
	But one can see by (\ref{criticité de type 2}), (\ref{h flèche}) and since $h^\downarrow(\ell)\sim 1/\sqrt{\pi \ell}$ when $\ell \to \infty$ that
	$$
	\frac{W^{(0)}W^{(0)}}{W^{(1)}} \wedge \inf_{\substack{p \ge 2 \\ \ell\ge 1}} \left(\frac{1}{2} \nu(-1) \frac{h^\downarrow_\ell(p-1)}{h^\downarrow_\ell(p)} \wedge \frac{W^{(p-1)}W^{(0)}}{W^{(p)}} \right) >0.
	$$
	Besides, in the first case, when $p = 1$ we automatically have $p\ge 1$ at the next step by an event of type $C_k$, and with positive probability we will have $p\ge 2$ at the next step. Moreover, at each step of the exploration, $\deg(v_r)$ only increases by zero or one, hence the statement of the lemma.
\end{proof}

We now recall a lemma which gathers some well known consequences of the Bouttier-Di Francesco-Guitter bijection (from \cite{BDFG}) and Janson \& Stef\'ansson's trick (from \cite{JS} Section 3), together with the \L{}ukasiewicz path.
\begin{lemma}\label{lemme BDG}(\cite{BDFG},\cite{JS})
	The measure $\mu$ on $\Z_{\ge -1}$ defined by
	$$
	\mu(-1)= \frac{4}{c_{\bf q}} \qquad \text{and}
	\qquad
	\mu(k)= q_{k+1} \binom{2k+1}{k+1} (c_{\bf q}/4)^k \text{ for all } k \ge 0.
	$$
	is a probability measure.
	Moreover, to a map $\mathfrak{m}_\bullet$ under $\P^{(\ell)}_0$ we can associate a random walk $(Y_n)_{n\ge 0}$ of step distribution $\mu$ starting at zero and stopped when it attains $-\ell$ such that if we denote by $\tau_{-\ell}(Y)$ the time at which $Y$ reaches $-\ell$, then
	$$
	\left\{
	\begin{array}{l}
		\# \left\{ n \in \lb 0, \tau_{-\ell}(Y)-1 \rb;  \ Y_{n+1}-Y_n=-1\right\}+1 = \# \mathrm{Vertices} (\mathfrak{m}_\bullet) \\
		\# \left\{ n \in \lb 0, \tau_{-\ell}(Y) -1\rb;  \ Y_{n+1}-Y_n\neq -1\right\}+1  = \# \mathrm{Faces} (\mathfrak{m}_\bullet)
		.
	\end{array}
	\right.
	$$
	As a consequence $\tau_{-\ell}(Y) = \# \mathrm{Edges} (\mathfrak{m}_\bullet)$.
\end{lemma}

Using the previous lemmas, we are able to control the degree of a random uniform vertex in the map $\mathfrak{M}^{(\ell)}$. The degree of a uniform random vertex of the map $\mathfrak{M}^{(\ell)}$ is stochastically dominated by a random variable whose law does not depend on $\ell$:

\begin{lemma}\label{majoration degré sommet uniforme}
	There exists $C>0$ and $c \in (0,1)$ such that for all $k,\ell\ge 1$, if $V$ is a uniform random vertex of $\mathfrak{M}^{(\ell)}$, then
	$$
	\P ( \deg(V) = k) \le Cc^k. 
	$$
\end{lemma}

\begin{proof}
	Let $\ell \ge 1$ and $\vp>0$. Let $K>0$ whose value will be chosen later.
	For $k\ge 1$,
	\begin{align}
		\P(\deg(V) = k)  = &\P(\deg(V)=k \text{ and } K |\mathrm{Vertices}(\mathfrak{M}^{(\ell)})|\ge  |\mathrm{Edges}(\mathfrak{M}^{(\ell)})| ) \label{premier terme} \\
		&+ \P(\deg(V)=k \text{ and } K |\mathrm{Vertices}(\mathfrak{M}^{(\ell)})|<  |\mathrm{Edges}(\mathfrak{M}^{(\ell)})| ) \label{second terme}
	\end{align}
	For the first term (\ref{premier terme}), if we denote by $\vec{E}'$ a uniform oriented edge of $\mathfrak{M}^{(\ell)}$, one can see that for all $k\ge 1$,
	\begin{align*}
		\P&\left(\deg(V)=k \text{ and } K  |\mathrm{Vertices}(\mathfrak{M}^{(\ell)})|\ge |\mathrm{Edges}(\mathfrak{M}^{(\ell)})|\right) \\
		&=
		\E \left( {\bf 1}_{ K|\mathrm{Vertices}(\mathfrak{M}^{(\ell)})|\ge  |\mathrm{Edges}(\mathfrak{M}^{(\ell)})| } 
		\sum_{\vec{e} } {1\over \deg v(\vec{e})} {1\over |\mathrm{Vertices}(\mathfrak{M}^{(\ell)})|} {\bf 1}_{\deg (v(\vec{e})) = k } \right) \text{ by (\ref{eq sommet uniforme})}\\
		&= \E^{(\ell)} \left( {\bf 1}_{ K|\mathrm{Vertices}(\mathfrak{m})|\ge |\mathrm{Edges}(\mathfrak{m})| } {1 \over k |\mathrm{Vertices}(\mathfrak{m})|} \# \{ \vec{e} ;  \ \deg v(\vec{e}) = k\}\right) \\
		&\le \E^{(\ell)} \left( {\bf 1}_{K |\mathrm{Vertices}(\mathfrak{m})|\ge  |\mathrm{Edges}(\mathfrak{m})| }
		{2|\mathrm{Edges}(\mathfrak{m})|\over k |\mathrm{Vertices}(\mathfrak{m})|} \P(\deg (v(\vec{E}')) = k | \mathfrak{M}^{(\ell)}=\mathfrak{m})\right) \\
		&\le \E^{(\ell)} \left( 2K \P(\deg (v(\vec{E}')) = k |  \mathfrak{M}^{(\ell)}=\mathfrak{m}) \right) \\
		&=2K\P_\ell^{(1)}(\deg (v_r)=k) \\
		& \le 2K c^{k-2} \text{ by Lemma \ref{degré arête},}
	\end{align*}
	with $C >0$ and $c\in (0,1)$ independent of $\ell$. We now focus on the second term (\ref{second terme}).

	\begin{align*}
		\P&\left(\deg(V)=k \text{ and } K|\mathrm{Vertices}(\mathfrak{M}^{(\ell)})|<  |\mathrm{Edges}(\mathfrak{M}^{(\ell)})| \right) \\
		&\le \P\left(|\mathrm{Edges}(\mathfrak{M}^{(\ell)})|\ge k \text{ and } K|\mathrm{Vertices}(\mathfrak{M}^{(\ell)})|<  |\mathrm{Edges}(\mathfrak{M}^{(\ell)})| \right).
	\end{align*}
	Let $Y$ be the $\mu$-random walk stopped when it attains $-\ell$ associated with $\mathfrak{m}_\bullet$ under $\P_0^{(\ell)}$ by Lemma \ref{lemme BDG}. We also know that $\tau_{-\ell}(Y)=|\mathrm{Edges}(\mathfrak{m})|$ and $D\coloneqq\#\{ i \in \lb 0, \tau_{-\ell}(Y)-1 \rb ; \ Y_{i+1}-Y_{i} = -1 \}+1 = |\mathrm{Vertices}(\mathfrak{m})|$. 
	Thus, by biasing by the number of vertices, we obtain that
	
	\begin{align*}
		\P&\left(|\mathrm{Edges}(\mathfrak{M}^{(\ell)})|\ge k \text{ and } K |\mathrm{Vertices}(\mathfrak{M}^{(\ell)})|<  |\mathrm{Edges}(\mathfrak{M}^{(\ell)})| \right) \\
		&=\frac{1}{W^{(\ell)}} \sum_{\mathfrak{m} \in \mathcal{M}^{(\ell)}}
		w_{\bf q}(\mathfrak{m}) {\bf 1}_{|\mathrm{Edges}(\mathfrak{m})|\ge k \text{ and } K|\mathrm{Vertices}(\mathfrak{m})|<  |\mathrm{Edges}(\mathfrak{m})|  } \\
		&= \frac{W^{(\ell)}_0}{W^{(\ell)}} \E^{(\ell)}_0 \left(\frac{1}{|\mathrm{Vertices}(\mathfrak{m_\bullet})|}
		{\bf 1}_{|\mathrm{Edges}(\mathfrak{m}_\bullet)|\ge k \text{ and } K|\mathrm{Vertices}(\mathfrak{m}_\bullet)|<  |\mathrm{Edges}(\mathfrak{m}_\bullet)|  } \right) \\
		&=\frac{W^{(\ell)}_0}{W^{(\ell)}} \E \left(
		\frac{1}{D}
		{\bf 1}_{ KD\le  \tau_{-\ell}(Y) \text{ and } \tau_{-\ell}(Y) \ge k }
		\right) \\
		&=\frac{W^{(\ell)}_0}{W^{(\ell)}} \E \left(
		\frac{1}{D}
		\P\left(\left.
		KD\le  \tau_{-\ell}(Y) \text{ and } \tau_{-\ell}(Y) \ge k 
		\right| D\right)
		\right).
	\end{align*}
	Let us upper bound the conditional probability inside the expectation. One can note that the random walk $Y$ can be built by first sampling the non-zero steps and then inserting between these steps a geometric number (starting at zero, of parameter $1-\mu(0)$) of zero steps (when $\mu(0)=0$, we do not insert any such zero step). Moreover the number of positive steps before $\tau_{-\ell}(Y)$ is smaller than $D$. Thus one can write
	$$
	\tau_{-\ell}(Y)\le 2D+ \sum_{i=1}^{2D} \mathcal{G}_i,
	$$
	where the $\mathcal{G}_i$ are geometric i.i.d. random variables of parameter $1-\mu(0)$, where by convention $\mathcal{G}_i=0$ when $\mu(0)=0$. Let $\beta=(1-\mu(0))/(2(3-\mu(0)))$ so that $1-2\beta=4\beta/(1-\mu(0))$. We will distinguish whether $D\ge \beta k$ or not. On the one hand, if $D \le \beta$, then
	\begin{align*}
		\P\left(\left.
		KD\le  \tau_{-\ell}(Y) \text{ and } \tau_{-\ell}(Y) \ge k 
		\right| D\right) {\bf 1}_{D\le \beta k} 
		&\le
		\P\left(\left.
		\tau_{-\ell}(Y) \ge k 
		\right| D\right) {\bf 1}_{D\le \beta k} \\
		&\le \P\left( \left. 2D+ \sum_{i=1}^{2D} \mathcal{G}_i \ge k \right| D\right) {\bf 1}_{D\le \beta k}\\
		&\le \P\left( \sum_{i=1}^{2\beta k} \mathcal{G}_i \ge (1-2\beta)k \right)\\
		&\le C' (c')^k,
	\end{align*}
	for $C'>0$, $c' \in (0,1)$ which do not depend on $\ell$, where the last inequality comes from the large deviations for sums of i.i.d. geometric variables.
	On the other hand, when $D\ge \beta k$, by choosing $K\ge 2+4/(1-\mu(0))$,
	\begin{align*}
		\P\left(\left.
		KD\le  \tau_{-\ell}(Y) \text{ and } \tau_{-\ell}(Y) \ge k 
		\right| D\right) {\bf 1}_{D\ge \beta k} 
		&=
		\P\left(\left.
		KD\le  \tau_{-\ell}(Y) \right| D\right) {\bf 1}_{D\ge \beta k} \\
		&\le
		\P\left(\left.
		(K-2)D\le  \sum_{i=1}^{2D} \mathcal{G}_i \right| D\right) {\bf 1}_{D\ge \beta k}\\
		&\le C'(c')^D{\bf 1}_{D\ge \beta k} \le C' (c')^{\beta k},
	\end{align*}
	by the same large deviations. This ends the proof since 
	$
	\E ({1}/{D}) = \E^{(\ell)}_0 ({1}/{|\mathrm{Vertices}(\mathfrak{m})|}) = {W^{(\ell)}}/{W^{(\ell)}_0}
	$
	and thus
	$$
	\frac{W^{(\ell)}_0}{W^{(\ell)}} \E \left(
	\frac{1}{D}
	\P\left(\left.
	KD\le  \tau_{-\ell}(Y) \text{ and } \tau_{-\ell}(Y) \ge k 
	\right| D\right)
	\right)
	\le C' (c')^{\beta k},
	$$
	hence the upper bound of the statement.
\end{proof}
If $v$ is a vertex of $\mathfrak{m}$, and if $d^\dagger$ is a distance on $\mathfrak{m}^\dagger$, we define $d^\dagger (f_r,v)$ as the smallest distance from $f_r$ to a face $f$ next to $v$.

\begin{proposition}\label{proppointunif}
	If for all $\ell\ge 1$, $V^{(\ell)}$ is a random uniform vertex of $\mathfrak{M}^{(\ell)}$, then
	$$
	\frac{d_\mathrm{gr}^\dagger (f_r, V^{(\ell)}) }{(\log \ell)^2} \mathop{\longrightarrow}\limits_{\ell \to \infty}^{(\P)} \frac{1}{ 2 \pi^2} 
	\qquad \text{and} \qquad
	\frac{d_\mathrm{fpp}^\dagger (f_r, V^{(\ell)})}{ \log \ell} \mathop{\longrightarrow}\limits_{\ell \to \infty}^{(\P)} \frac{1}{  \pi^2 p_{\bf q}} .
	$$
\end{proposition}

\begin{proof}
	We do the proof for the dual graph distance (the same ideas work for the fpp distance using Proposition \ref{distance fpp arete uniforme}). First of all, we can work with $d_\mathrm{gr}^\dagger (f_r, \vec{E})$ with $\vec{E}$ defined as above by (\ref{E flèche}) since $d_\mathrm{gr}^\dagger (f_r, \vec{E})-\deg(V^{(\ell)}) \le d_\mathrm{gr}^\dagger (f_r, V^{(\ell)}) \le d_\mathrm{gr}^\dagger (f_r, \vec{E})$ and since by Lemma \ref{majoration degré sommet uniforme} $\deg V^{(\ell)}$ is stochastically dominated by a random variable which does not depend on $\ell$. Let $\vp, \delta >0$.
	Let $\vec{E}'$ be a uniform random oriented edge. By Proposition \ref{distanceàunsommetuniforme}, for $ \ell$ large, with probability at least $ 1-\delta$,
	\begin{align*}
		\P&\left( \left.{d_\mathrm{gr}^\dagger (f_r,\vec{E}) \over (\log \ell)^2} \in \left[ {1\over 2\pi^2}-\vp , {1\over 2\pi^2} +\vp \right] \right| \mathfrak{M}^{(\ell)}\right) = \sum_{\vec{e}} {1\over \deg v(\vec{e})} {1\over |\mathrm{Vertices}(\mathfrak{M}^{(\ell)})|} {\bf 1}_{{d_\mathrm{gr}^\dagger (f_r,\vec{E}) \over (\log \ell)^2} \in \left[ {1\over 2\pi^2}-\vp , {1\over 2\pi^2} +\vp \right]}
		\\
		&= {2|\mathrm{Edges}(\mathfrak{M}^{(\ell)})|\over |\mathrm{Vertices}(\mathfrak{M}^{(\ell)})|} \E \left( \left. {1\over \deg v(\vec{E}')} {\bf 1}_{{d_\mathrm{gr}^\dagger (f_r,\vec{E}') \over (\log \ell)^2} \in \left[ {1\over 2\pi^2}-\vp , {1\over 2\pi^2} +\vp \right]} \right| \mathfrak{M}^{(\ell)} \right) \\
		&\ge {2|\mathrm{Edges}(\mathfrak{M}^{(\ell)})|\over |\mathrm{Vertices}(\mathfrak{M}^{(\ell)})|} \left( \E \left( \left.  {1\over \deg v(\vec{E}')} \right| \mathfrak{M}^{(\ell)}\right) -\delta \right)  \\
		&= 1-{2|\mathrm{Edges}(\mathfrak{M}^{(\ell)})|\over |\mathrm{Vertices}(\mathfrak{M}^{(\ell)})|} \delta \\
		&\ge 1 - C\delta,
	\end{align*}
	where $C>0$ is a constant that does not depend on $\ell$ and $\delta$.
	The last inequality comes from Lemma \ref{lemme BDG} and the law of large numbers.
\end{proof}

The case of the distance from the root face to a uniform random face of $\mathfrak{M}^{(\ell)}$ is even simpler and relies on similar ideas. 

\begin{proof}[Proof of Theorem \ref{distance de la racine à une face uniforme}]
	Let $\ell \ge 1$.
	\phantomsection \label{preuve distance de la racine à une face uniforme} Let $F$ be a uniform random face on $\mathfrak{M}^{(\ell)}$ conditionally on $\mathfrak{M}^{(\ell)}$. Conditionally on $F$ and $\mathfrak{M}^{(\ell)}$, let $\vec{E}$ be uniform random oriented edge such that the face on its right is $F$. Then by definition of the distance to an edge,
	$$
	d_\mathrm{gr}^\dagger (f_r, \vec{E}) \le 
	d_\mathrm{gr}^\dagger (f_r, F) \le 
	d_\mathrm{gr}^\dagger (f_r, \vec{E})+1
	\qquad\text{and}\qquad
	d_\mathrm{fpp}^\dagger (f_r, \vec{E}) \le 
	d_\mathrm{fpp}^\dagger (f_r, F) \le 
	d_\mathrm{fpp}^\dagger (f_r, \vec{E})+\mathcal{E},
	$$
	where $\mathcal{E}$ is an exponential random variable of parameter $1$.
	Thus it is enough to show that 
	$$
	{d^\dagger_\mathrm{gr} (f_r,\vec{E}) \over (\log \ell)^2} \mathop{\longrightarrow}\limits_{\ell \to \infty}^{(\P)} {1 \over \pi^2}
	\
	\text{ and }
	\
	{d^\dagger_\mathrm{fpp} (f_r,\vec{E}) \over \log \ell} \mathop{\longrightarrow}\limits_{\ell \to \infty}^{(\P)} {2 \over \pi^2 p_{\bf q}}.
	$$
	We now focus on the dual graph distance. For the fpp distance, the same ideas apply using Proposition \ref{distance fpp arete uniforme}. If $\vec{e}$ is an oriented edge, we write $f(\vec{e})$ the face on the right of $\vec{e}$. Conditionally on $\mathfrak{M}^{(\ell)}$, let $\vec{E}'$ be a uniform oriented edge on $\mathfrak{M}^{(\ell)}$. Let $\vp,\delta>0$. Then by Proposition \ref{distanceàunsommetuniforme}, for $\ell $ large, with probability at least $1-\delta$,
	\begin{align*}
		\P &\left( \left.{d_\mathrm{gr}^\dagger (f_r,\vec{E}) \over (\log \ell)^2} \in \left[ {1\over 2\pi^2}-\vp , {1\over 2\pi^2} +\vp \right] \right| \mathfrak{M}^{(\ell)}\right) = \sum_{\vec{e}} {1\over \deg f(\vec{e})} {1\over |\mathrm{Faces}(\mathfrak{M}^{(\ell)})|} {\bf 1}_{{d_\mathrm{gr}^\dagger (f_r,\vec{E}) \over (\log \ell)^2} \in \left[ {1\over 2\pi^2}-\vp , {1\over 2\pi^2} +\vp \right]}
		\\
		&= {2|\mathrm{Edges}(\mathfrak{M}^{(\ell)})|\over |\mathrm{Faces}(\mathfrak{M}^{(\ell)})|} \E \left( \left. {1\over \deg f(\vec{E}')} {\bf 1}_{{d_\mathrm{gr}^\dagger (f_r,\vec{E}') \over (\log \ell)^2} \in \left[ {1\over 2\pi^2}-\vp , {1\over 2\pi^2} +\vp \right]} \right| \mathfrak{M}^{(\ell)} \right) \\
		&\ge {2|\mathrm{Edges}(\mathfrak{M}^{(\ell)})|\over |\mathrm{Faces}(\mathfrak{M}^{(\ell)})|} \left( \E \left( \left. {1\over \deg f(\vec{E}')} \right| \mathfrak{M}^{(\ell)} \right) -\delta \right) \\
		&={2|\mathrm{Edges}(\mathfrak{M}^{(\ell)})|\over |\mathrm{Faces}(\mathfrak{M}^{(\ell)})|}
		\frac{1}{2|\mathrm{Edges}(\mathfrak{M}^{(\ell)})|} \left(\sum_{\vec{e}} \frac{1}{\deg f(\vec{e})}\right) -{2|\mathrm{Edges}(\mathfrak{M}^{(\ell)})|\over |\mathrm{Faces}(\mathfrak{M}^{(\ell)})|}\delta .
	\end{align*}
	But by Lemma \ref{lemme BDG} and the law of large numbers, with probability at least $1-\delta$ for $\ell$ large,
	$$
	{2|\mathrm{Edges}(\mathfrak{M}^{(\ell)})|\over |\mathrm{Faces}(\mathfrak{M}^{(\ell)})|}
	\frac{1}{2|\mathrm{Edges}(\mathfrak{M}^{(\ell)})|} \left(\sum_{\vec{e}} \frac{1}{\deg f(\vec{e})}\right) -{2|\mathrm{Edges}(\mathfrak{M}^{(\ell)})|\over |\mathrm{Faces}(\mathfrak{M}^{(\ell)})|}\delta
	= 1-{2|\mathrm{Edges}(\mathfrak{M}^{(\ell)})|\over |\mathrm{Faces}(\mathfrak{M}^{(\ell)})|} \delta 
	\ge 1 - C\delta,
	$$
	where $C>0$ is a constant that does not depend on $\ell$ (nor on $\delta$, $\vp$).
\end{proof}

\section{Distance between two uniform random faces}\label{section distance entre sommets uniformes}
This section is devoted to the proof of the following corollary, which establishes the scaling limit in probability of the distance between two uniform random faces.
\begin{corollary}\label{distance entre deux faces uniformes}
	For all $\ell \ge 1$, let $F^{(\ell)}_1$ and $F^{(\ell)}_2$ be two independent uniform random faces in $\mathfrak{M}^{(\ell)}$. Then
	$$\frac{d_\mathrm{fpp}^\dagger (F^{(\ell)}_1, F^{(\ell)}_2) }{\log \ell } \mathop{\longrightarrow}\limits_{\ell \to \infty}^{(\P)} \frac{2}{ \pi^2 p_{\bf q}} \qquad
	\text{and} \qquad
	\frac{d_\mathrm{gr}^\dagger  (F^{(\ell)}_1, F^{(\ell)}_2)}{(\log \ell)^2 } \mathop{\longrightarrow}\limits_{\ell \to \infty}^{(\P)} {1 \over \pi^2 }.
	$$
\end{corollary}
Its statement can be extended to the case of uniform random edges or uniform random vertices, see Theorem \ref{distance entre deux arêtes uniformes} and Corollary \ref{cor distance entre deux sommets uniformes}. Actually, we will first prove the result for uniform random edges. 
The main idea will be to show that geodesics from the root to two uniform random edges ``split near the root'' with high probability. Hence, the map $\mathfrak{M}^{(\ell)}$ looks like a star with many branches as in Figure \ref{fig:ssimus_intro}.
\subsection{Two random edges are in ``different branches'' with high probability}
We prove here that the branches of two uniform random edges split near the root face. More precisely, we would like to prove that the filled-in exploration targeted at the first random uniform edge swallows the second edge at a time relatively far from the end of the exploration with high probability.
However, we do not know the law of the peeling exploration of $\mathfrak{M}^{(\ell)}$ targeted at a random uniform edge, but Lemma \ref{lemme transfert biais} gives us a way to transfer results about maps with target under $\P^{(\ell)}_1$ to $\mathfrak{M}^{(\ell)}$ equipped with a uniform edge. With this in mind, let $\mathfrak{M}^{(\ell)}_1$ be a random map with a target $2$-face of law $\P^{(\ell)}_1$. The target $2$-face $\square$ will play the role of the first uniform edge.
Let $E^{(\ell)}$ be a uniform random edge of $\mathfrak{M}^{(\ell)}_1$ (where $\square$ is considered as one edge), which will play the role of the second uniform edge. Let $\tau_{-1}^{(\ell)}$ be the time at which the perimeter process $P_1^{(\ell)}$ associated with the filled-in exploration $(\overline{\mathfrak{e}}_n^{(\ell)})_{n\ge 0}$ of $\mathfrak{M}^{(\ell)}_1$ dies at $-1$, i.e. the duration of that filled-in exploration.
Recall from Proposition \ref{périmètre des cartes à grand bord} that $\tau_{-1}^{(\ell)}/\ell$ converges in law.
\begin{lemma}\label{séparation}
	We have
	$$
	\lim_{\vp \to 0} \limsup_{\ell \to \infty} \P(\square\text{ and } E^{(\ell)} \text{ lie in the same hole until time } (\tau_{-1}^{(\ell)}-\vp \ell)\vee 0 )=0.
	$$
\end{lemma}

\begin{proof}
	Let us consider the filled-in exploration $(\overline{\mathfrak{e}}_n^{(\ell)})_{n\ge 0}$. The idea is to show that $\square$ and $E^{(\ell)}$ are in different holes at a splitting event happening ``not too late''.
	By Proposition \ref{périmètre des cartes à grand bord}, we know that
	\begin{equation}\label{eqlimech}
		\left( {P^{(\ell)}_1(\ell t ) \over \ell} \right)_{t\ge 0} \mathop{\longrightarrow}\limits_{\ell \to \infty}^{(\mathrm{d})} \left(\Upsilon^\downarrow(p_{\bf q} t)\right)_{t\ge 0},
	\end{equation}
	where $\Upsilon^\downarrow$ is the Cauchy process $\Upsilon$ starting from $1$ and conditioned to die continuously at $0$. Let $\zeta\coloneqq \inf \left\{t\ge 0; \ \Upsilon^\downarrow (t)=0\right\}$ be the lifetime of $\Upsilon^\downarrow$.
	
	First, we note that the number of negative jumps of size larger than half of $\Upsilon^\downarrow$ is a.s. infinite. More precisely
	$$
	N\coloneqq\# \left\{ t \in [0, \zeta[; \ \Upsilon^\downarrow (t)< {\Upsilon^\downarrow (t-) \over 2} \right\} = \infty \text{ a.s.}
	$$
	Indeed, by Proposition 5.2 of \cite{BBCK} (in the case $\theta=1, \rho=1/2,\hat{\beta}=1,\gamma=1/2,\hat{\gamma}=1/2$) up to a time change, $\Upsilon^\downarrow$ is the exponential of a Lévy process. Let $\Lambda$ be its Lévy measure. %One can compute $\Pi$ using that 
	The Laplace exponent of that Lévy process is $\Phi^-(q)= \kappa_1(q-3/2)= (q-1/2)\tan(\pi q)$ by Section 4.3 of \cite{BBCK} with $\omega_-=3/2$ for $q\in (-1/2,3/2)$. One can check that $\Lambda$ is the image measure of ${1 \over \pi\sqrt{x} (1-x)^2} dx$ by the mapping $x\mapsto \log x$ 
	. Then $N$ is a Poisson random variable of expectation
	$$
	\E N  = \int_0^{\infty} dt \int_{-\infty}^{-\log 2} \Lambda(dy) = \infty \times \int_{0}^{1/2}{1 \over \pi\sqrt{x} (1-x)^2} dx= \infty,
	$$
	hence a.s. $N = \infty $. 
	
	Moreover, by (\ref{eqlimech}), the number $N_\vp^{(\ell)}$ of negative jumps of size larger than half of $P^{(\ell)}_1(\cdot)$ until time $\tau^{(\ell)}_{-1}-\vp \ell$ converges in distribution towards the number $N_\vp$ of negative jumps of $\Upsilon^\downarrow$ of size larger than half of $\Upsilon^\downarrow(\cdot-)$ until time $\zeta-\vp p_{\bf q}$ as $\ell \to \infty$, which further goes a.s. to $N=\infty$ as $\vp \to 0$.

	Now we fix such a step: let $n$ be such that $P_1^{(\ell)}(n+1) <{P_1^{(\ell)}(n)}/{ 2}$. We write $k_1  = P_1^{(\ell)}(n+1)$ and $k_2  = P_1^{(\ell)}(n) - P_1^{(\ell)}(n+1)-1$ the half-perimeters of the two holes created at this step. Let us show that the probability that the uniform random edge $E^{(\ell)}$ lies in the filled-in hole $\mathfrak{h}$, which is the hole that does not contain $\square$, conditionally on not being an internal edge of the explored region $\mathfrak{e}$ obtained from $\overline{\mathfrak{e}}_n^{(\ell)}$ after creating the two holes but before filling in the hole $\mathfrak{h}$
	and conditionally on $\mathfrak{e}$, is greater than ${1/3}$ as soon as $k_2$ is sufficiently large. This condition is verified if $P_1^{(\ell)} (n)$ is large enough. 
	
	Indeed, by Proposition \ref{Markov explo}, conditionally on $\mathfrak{e}$, the maps filling the two holes are independent $\bf q$-Boltzmann planar maps of corresponding perimeter, and the one which contains $\square$ has a target face of degree $2$. Thus if we write for all $\ell \ge 1, p\ge 0,n\ge 1$,
	$$
	W_p^{(\ell)}[n] = \sum_{\substack{\mathfrak{m}_\square\in \mathcal{M}^{(\ell)}_p \\ |\mathrm{Edges}(\mathfrak{m})|=n}} w_{\bf q} (\mathfrak{m}_\square)
	\qquad
	\text{and}
	\qquad
	W^{(\ell)}[n] = \sum_{\substack{\mathfrak{m}\in \mathcal{M}^{(\ell)} \\ |\mathrm{Edges}(\mathfrak{m})|=n}} w_{\bf q} (\mathfrak{m}),
	$$
	then
	\begin{align*}
		\P(E^{(\ell)} \in \mathfrak{h} | \mathfrak{e} \text{ and } E^{(\ell)} \not\in \mathfrak{e} ) &= 
		\sum_{\mathfrak{u}_1 \in \mathcal{M}_1^{(k_1)}, \mathfrak{u}_2 \in \mathcal{M}_{\phantom{1}}^{(k_2)}} { |\mathrm{Edges}(\mathfrak{u}_2)| \over |\mathrm{Edges}(\mathfrak{u}_1)|-1+|\mathrm{Edges}(\mathfrak{u}_2)|} {w_{\bf q} (\mathfrak{u}_1) \over W_1^{(k_1)}} {w_{\bf q} (\mathfrak{u}_2) \over W_{\phantom{1}}^{(k_2)}} \\
		&= \sum_{n_1, n_2 \ge 1} {n_2 \over n_1 +n_2-1} {W_1^{(k_1)}[n_1] \over W_1^{(k_1)} } {W_{\phantom{1}}^{(k_2)}[n_2] \over W_{\phantom{1}}^{(k_2)}} \\
		&= \sum_{n_1, n_2 \ge 1} {1 \over n_1+n_2 -1} {W_1^{(k_1)}[n_1] \over W_1^{(k_1)} } {W_1^{(k_2)}[n_2+1] \over W_1^{(k_2)}} {W_1^{(k_2)} \over W_{\phantom{1}}^{(k_2)}},
	\end{align*}
	where the last equality holds since by taking into account that zipping the target face gives a distinguished edge, we have $W_1^{(\ell)}[n+1]=nW^{(\ell)}[n]$ for all $n\ge 1$. Yet, it is known that 
	\begin{equation}\label{fonction de partition taille n}
		\forall n \ge 1, \qquad W_1^{(\ell)}[n] = \frac{\ell}{\ell+1} \binom{2\ell}{\ell} (c_{\bf q}/4)^{\ell+1} \P(\tau_{-\ell-1}(Y)=n),
	\end{equation}
	where $Y$ is a $\mu$-random walk starting from zero as in Lemma \ref{lemme BDG}, see for instance Theorem 3.12 in \cite{StFlour}. Thus
	\begin{align*}
		\sum_{n_1, n_2 \ge 1} {1 \over n_1+n_2 -1} {W_1^{(k_1)}[n_1] \over W_1^{(k_1)} } &{W_1^{(k_2)}[n_2+1] \over W_1^{(k_2)}} {W_1^{(k_2)} \over W^{(k_2)}} \\
		&=\sum_{n_1,n_2 \ge 1} {1\over n_1 +n_2 -1} \P(\tau_{-k_1-1}(Y)  = n_1) \P(\tau_{-k_2-1}(Y)  = n_2+1) {W_1^{(k_2)} \over W^{(k_2)}} \\
		&= \E {1\over \tau_{-k_1-k_2-2}(Y)-2} {{1\over 2} h_1^\downarrow(k_2) c_{\bf q}^{k_2+1} \over W^{(k_2)}} \text{ by (\ref{fonction de partition}}).
	\end{align*}
	Furthermore, since we have $W_1^{(\ell)}[n+1]=nW^{(\ell)}[n]$ for all $n\ge 1$ so that using (\ref{fonction de partition taille n}), one can see that for all $k\ge 1$,
	$$ \E \frac{1}{\tau_{-k-1}(Y)-1}= \frac{W^{(k)}}{\frac{1}{2}h^\downarrow_1(k)c_{\bf q}^{k+1}}.$$
	The same computation can be found e.g.\@ in Equation (3.11) in \cite{StFlour}.
	As a consequence, since $\E {1/(\tau_{-k-1}(Y)-2)} \sim \E {1/(\tau_{-k-1}(Y)}-1)$ when $k \to \infty$,
	\begin{align*}
		\E {1\over \tau_{-k_1-k_2-2}(Y)-2} {{1\over 2} h_1^\downarrow(k_2) c_{\bf q}^{k_2+1} \over W^{(k_2)}}
		&\mathop{\sim}\limits_{k_2 \to \infty} {W^{(k_1+k_2+1)} \over {1\over 2} h_1^\downarrow (k_1+k_2+1) c_{\bf q}^{k_1+k_2+2}} {{1\over 2} h_1^\downarrow(k_2) c_{\bf q}^{k_2+1} \over W^{(k_2)}} \\
		&\mathop{\sim}\limits_{k_2 \to \infty} {{p_{\bf q} \over 2} c_{\bf q}^{k_1+k_2+2} {1\over (k_1+k_2)^2} \over {1\over 2} h_1^\downarrow (k_1+k_2+1) c_{\bf q}^{k_1+k_2+2}} {{1\over 2} h_1^\downarrow(k_2) c_{\bf q}^{k_2+1} \over {p_{\bf q} \over 2} c_{\bf q}^{k_2+1} {1\over k_2^2}} \text{ by (\ref{criticité de type 2})} \\
		&\mathop{\sim}\limits_{k_2 \to \infty} {k_2^{3/2} \over (k_1+k_2)^{3/2}} \ge {1\over 2^{3/2}} >{1\over 3}, 
	\end{align*}
	where we used that $k_1<k_2+1$ by assumption. From the above computations, we obtain that
	$$
	\P(E^{(\ell)} \in \mathfrak{h} \mid \mathfrak{e} \text{ and } E^{(\ell)} \not\in \mathfrak{e} )
	> \frac{1}{3}
	$$
	if $P_1^{(\ell)}(n)$ is large enough. Therefore, the probability that $\square$ and $E^{(\ell)}$ lie in the same hole in the exploration until the $k$-th negative jump of $P_1^{(\ell)}$ larger that half of the perimeter is less than $(2/3)^k$ for large $\ell$. We also use the fact that at the time of this $k$-th jump the perimeter of the distinguished hole is of order $\ell$ thanks to (\ref{eqlimech}). Thus,
	$$
	\limsup_{\ell \to \infty} \P( \square \text{ and } E^{(\ell)} \text{ lie in the same hole until time } (\zeta-\vp) \ell ) \le \E \left(\left(\frac{2}{3}\right)^{N_\vp}\right).
	$$
	One concludes by dominated convergence.
\end{proof}

\subsection{Distance between two uniform random edges, vertices or faces}

\begin{figure}[h]
	\centering
	\includegraphics[scale=0.7]{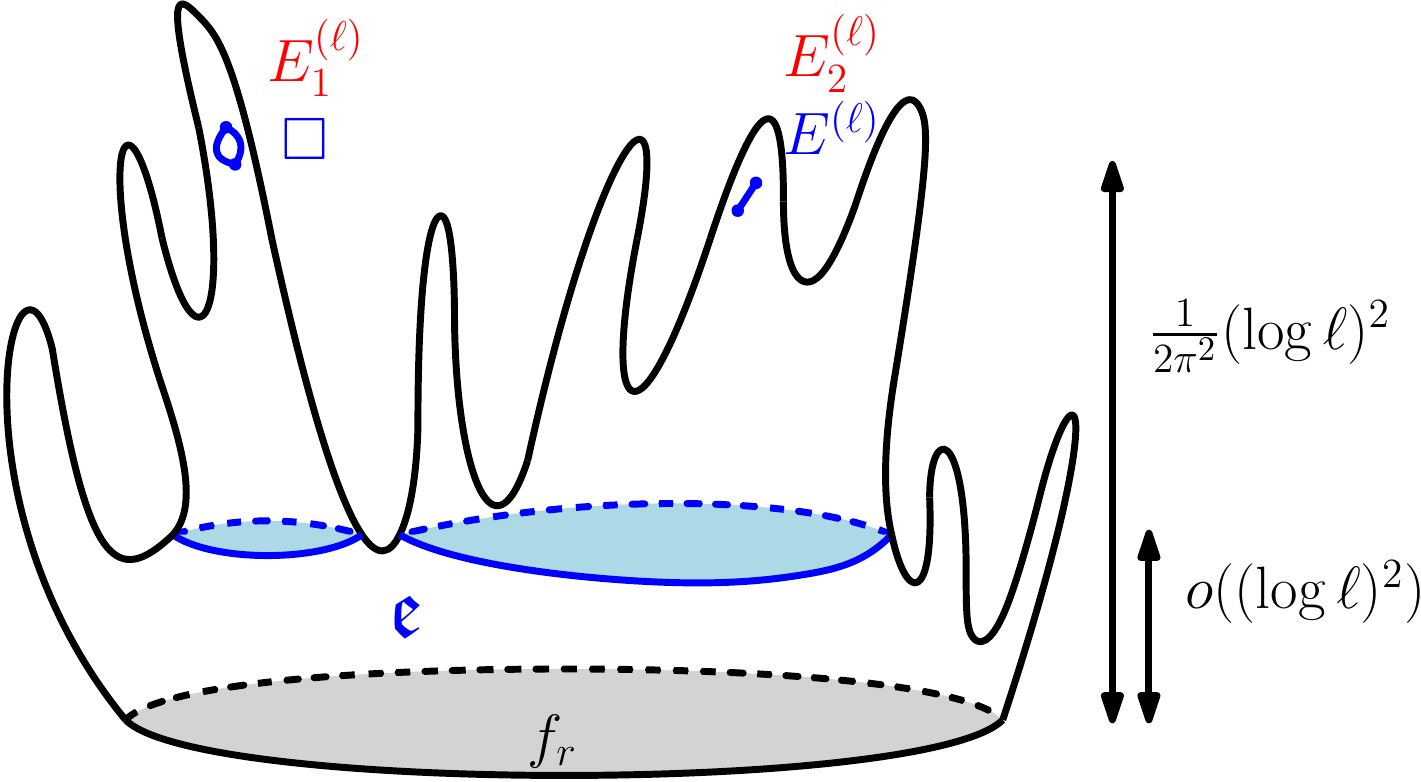}
	\caption{Illustration of the proof of Theorem \ref{distance entre deux arêtes uniformes}: the map $\mathfrak{M}^{(\ell)}$ (or $\mathfrak{M}^{(\ell)}_1$) is represented as a cactus where the height corresponds to the distance to the root face $f_r$. The blue holes are the two holes of the explored region $\mathfrak{e}$. The target $2$-face $\square$ and the uniform random edge $E^{(\ell)}$ on $\mathfrak{M}^{(\ell)}_1$, in blue, correspond respectively to the random uniform edges $E_1^{(\ell)}$ and $E_2^{(\ell)}$ in $\mathfrak{M}^{(\ell)}$, in red. }
	\label{image ça branche}
\end{figure}

\begin{theorem}\label{distance entre deux arêtes uniformes}
	If $\mathfrak{M}^{(\ell)}$ follows the law $\P^{(\ell)}$ and $E_1^{(\ell)},E_2^{(\ell)}$ are two uniform random edges of $\mathfrak{M}^{(\ell)}$, then
	$$
	{d^\dagger_\mathrm{gr} (E_1^{(\ell)},E_2^{(\ell)}) \over (\log \ell)^2} \mathop{\longrightarrow}\limits_{\ell \to \infty}^{(\P)} {1 \over \pi^2}
	\
	\text{ and }
	\
	{d^\dagger_\mathrm{fpp} (E_1^{(\ell)},E_2^{(\ell)}) \over \log \ell} \mathop{\longrightarrow}\limits_{\ell \to \infty}^{(\P)} {2 \over \pi^2 p_{\bf q}}
	$$
\end{theorem}

The general idea of the proof is given in Figure \ref{image ça branche}. The key step is an intermediate result for the map $\mathfrak{M}^{(\ell)}_1$ of law $\P^{(\ell)}_1$.

\begin{lemma}\label{lemme distance entre deux arêtes uniformes}
	If $\mathfrak{M}^{(\ell)}_1$ is a random map with target of law $\P^{(\ell)}_1$ and $E^{(\ell)}$ is a random uniform edge on $\mathfrak{M}^{(\ell)}_1$ (where the target $2$-face has been zipped and is considered as an edge), then
	$$
	\frac{d^\dagger_\mathrm{gr} (\square,E^{(\ell)}) - d^\dagger_\mathrm{gr}(f_r,E^{(\ell)}) }{(\log \ell)^2} \mathop{\longrightarrow}\limits_{\ell \to \infty}^{(\P)} {1 \over 2\pi^2}
	\
	\text{ and }
	\
	\frac{d^\dagger_\mathrm{fpp} (\square,E^{(\ell)})-d^\dagger_\mathrm{fpp}(f_r,E^{(\ell)})}{ \log \ell} \mathop{\longrightarrow}\limits_{\ell \to \infty}^{(\P)} {1 \over \pi^2 p_{\bf q}}
	$$
\end{lemma}

\begin{proof}
	We consider the filled-in peeling exploration of $\mathfrak{M}^{(\ell)}_1$ using the algorithm $\mathcal{A}_\mathrm{layers}$ for the dual graph distance and the algorithm $\mathcal{A}_\mathrm{uniform}$ for the fpp distance. Let $n$ be the time at which during an event of type $G$, $\square$ and $E^{(\ell)}$ are in two distinct holes $\mathfrak{h}_1$ and $\mathfrak{h}_2$ (when it exists). By Lemma \ref{séparation}, we know that for $\vp$ small, $n\le \tau^{(\ell)}_{-1}-\vp \ell$ with large probability. We also denote by $\mathfrak{e}$ the sub-map obtained after identifying the edges but before filling-in the hole containing $E^{(\ell)}$.

	We work with the dual graph distance, and thus with the peeling by layers algorithm. We point out that a path from $\square$ to $E^{(\ell)}$ must go through $\mathfrak{e}$ (see Figure \ref{image ça branche}), so that
	$$
	d^\dagger_\mathrm{gr} (\square, E^{(\ell)}) \ge d^\dagger_\mathrm{gr} (\partial \mathfrak{e},\square) +d^\dagger_\mathrm{gr} (\partial \mathfrak{e}, E^{(\ell)}).
	$$
	But at time $n$, the perimeter $P_1(n)$ is larger than some constant times $\ell$ with large probability thanks to Proposition \ref{périmètre des cartes à grand bord} since $n\le \tau_{-1}-\vp \ell$. Consequently, by Proposition \ref{distanceàunsommetuniforme}, and since the map filling the hole $\mathfrak{h}_1$ has the law $\P^{(P_1(n))}_1$ conditionally on the exploration up to time $n$ by Proposition \ref{Markov explo},
	$$
	d^\dagger_\mathrm{gr} (\partial \mathfrak{e},\square) = {(\log \ell)^2\over 2\pi^2} (1+o(1))
	\qquad\text{and} \qquad
	d^\dagger_\mathrm{gr} (f_r,\square) = {(\log \ell)^2\over 2\pi^2} (1+o(1))
	$$
	in probability on the event that $n$ exists and $n\le \tau_{-1}-\vp \ell$. In particular, in probability and on the same event,
	$
	d^\dagger_\mathrm{gr} (f_r,\partial \mathfrak{e}) = o((\log \ell)^2).
	$
	As a consequence, $d^\dagger_\mathrm{gr} (f_r,E^{(\ell)}) = d^\dagger_\mathrm{gr} (\partial \mathfrak{e}, E^{(\ell)})+ o ((\log \ell)^2)$.
	Thus, in probability,
	$$
	d^\dagger_\mathrm{gr} (\square,E^{(\ell)})  \ge {(\log \ell)^2\over 2\pi^2} (1+o(1)) + d^\dagger_\mathrm{gr} (f_r,E^{(\ell)}).
	$$
	The converse inequality is a direct consequence of the triangle inequality and Proposition \ref{distanceàunsommetuniforme}. 
	For the fpp distance, the same reasoning holds using Proposition \ref{distance fpp arete uniforme}.
\end{proof}

\begin{remark}
	In the above proof, we saw that $d^\dagger_\mathrm{gr}(f_r,\partial \mathfrak{e})=o((\log \ell)^2)$ and $d^\dagger_\mathrm{fpp}(f_r,\partial \mathfrak{e})=o(\log \ell)$. Nonetheless, it is easy to see that $d_\mathrm{fpp}^\dagger(f_r,\partial \mathfrak{e})= O(1)$ in probability using Proposition \ref{périmètre des cartes à grand bord} and Lemma \ref{séparation}. Moreover, we point out that Lemma \ref{analogue6}, together with Lemma \ref{séparation}, enables to show that $d^\dagger_\mathrm{gr}(f_r,\partial \mathfrak{e})=O(\log \ell)$ in probability.
\end{remark}

Theorem \ref{distance entre deux arêtes uniformes} follows by a judicious application of Lemma \ref{lemme transfert biais}:

\begin{proof}[Proof of Theorem \ref{distance entre deux arêtes uniformes}]
	For the dual graph distance, we apply Lemma \ref{lemme transfert biais} and Lemma \ref{lemme distance entre deux arêtes uniformes} to the function $\varphi_\ell$ defined by
	$$
	\forall \mathfrak{m}_\square \in \mathcal{M}^{(\ell)}_1, \qquad
	\varphi_\ell(\mathfrak{m}_\square) = \E \left(
	\left|\frac{d^\dagger_\mathrm{gr} (\square,E(\mathfrak{m}_\square) - d^\dagger_\mathrm{gr}(f_r,E(\mathfrak{m}_\square)) }{(\log \ell)^2}
	-
	\frac{1}{2\pi^2}
	\right|
	\wedge 1 \right),
	$$
	where $E(\mathfrak{m}_\square)$ is a random uniform edge on $\mathfrak{m}_\square$ (where the target $2$-face is considered as a single edge). As a result, we obtain that if $E_1^{(\ell)}$ and $E^{(\ell)}_2$ are two independent random uniform edges in $\mathfrak{M}^{(\ell)}$, then
	$$
	\E \left( \left|\frac{d^\dagger_\mathrm{gr} (E_1^{(\ell)},E_2^{(\ell)}) - d^\dagger_\mathrm{gr}(f_r,E_2^{(\ell)}) }{(\log \ell)^2}
	-
	\frac{1}{2\pi^2}
	\right|
	\wedge 1 \right)
	\mathop{\longrightarrow}\limits_{\ell \to \infty}
	0.
	$$
	The first convergence of the theorem then follows by applying once again Proposition \ref{distanceàunsommetuniforme}. For the fpp distance, the function $\varphi_\ell$ is defined similarly but the expectation also involves the exponential weights defining the fpp distance.
\end{proof}

Using the same ideas as in the proof of Proposition \ref{proppointunif} we also obtain the following result

\begin{corollary}\label{cor distance entre deux sommets uniformes}
	If $\mathfrak{M}^{(\ell)}$ follows the law $\P^{(\ell)}$ and $V_1,V_2$ are two uniform vertices of $\mathfrak{M}^{(\ell)}$, then
	$$
	{d^\dagger_\mathrm{gr} (V_1,V_2) \over (\log \ell)^2} \mathop{\longrightarrow}\limits_{\ell \to \infty}^{(\P)} {1 \over \pi^2}
	\
	\text{ and }
	\
	{d^\dagger_\mathrm{fpp} (V_1,V_2) \over \log \ell} \mathop{\longrightarrow}\limits_{\ell \to \infty}^{(\P)} {2 \over \pi^2 p_{\bf q}}.
	$$
\end{corollary}
\begin{proof}
	We do the proof for the dual graph distance, the same ideas work for the fpp distance. We can work with $d_\mathrm{gr}^\dagger (\vec{E}_1, \vec{E}_2)$ where $\vec{E}_1$,$\vec{E}_2$ are i.i.d. oriented edges whose law is the same as the law of $\vec{E}$ in (\ref{E flèche}) since $d_\mathrm{gr}^\dagger (\vec{E}_1, \vec{E}_2)-\deg(V_1^{(\ell)})-\deg(V_2^{(\ell)}) \le d_\mathrm{gr}^\dagger (V_1^{(\ell)}, V_2^{(\ell)}) \le d_\mathrm{gr}^\dagger (\vec{E}_1, \vec{E}_2)$ and since by Lemma \ref{majoration degré sommet uniforme} $\deg V^{(\ell)}_1$ and $\deg V^{(\ell)}_2$ are stochastically dominated by a random variable which does not depend on $\ell$. 
	We then conclude by following the same lines as in the proof of Proposition \ref{proppointunif} using 
	Theorem \ref{distance entre deux arêtes uniformes}.
\end{proof}
We finally prove Corollary \ref{distance entre deux faces uniformes} using the same ideas as in the proof of Theorem \ref{distance de la racine à une face uniforme} in Subsection  \ref{preuve distance de la racine à une face uniforme} and using Theorem \ref{distance entre deux arêtes uniformes}.

\section{Diameter of $\mathfrak{M}^{(\ell,\dagger)} $}\label{section diamètre}

The main purpose of this section is to give bounds on the diameter of $\mathfrak{M}^{(\ell,\dagger)} $ for the graph distance and for the fpp distance, proving Theorem \ref{majoration du diamètre graphe et fpp}. The upper bounds are achieved using a first moment method together with large deviations on the height of a uniform random edge, obtained using supermartingales. Two of these supermartingales (of Lemmas \ref{sur-martingale2} and \ref{martingale tronquée}) are related to the martingale $(M_n^{(\ell)})_{n\ge 0}$ defined using the perimeter process $P_\ell$ under $\P^{(1)}_\ell$ as follows:
\begin{equation}\label{martingale}
	\forall n\ge 0, \qquad
	M_n^{(\ell)} = \frac{1}{h^\downarrow_\ell(P_\ell(n))} \prod_{j=0}^{n-1}\left(\frac{1}{\nu([-P_\ell(j)+1,\infty)) + \nu(-P_\ell(j)-\ell)}\vee {\bf 1}_{P_\ell(j)=-\ell} \right).
\end{equation}
It is indeed a martingale since for all $n \ge 0$, if $P_\ell(n) \ge 1$, then
\begin{align*}
	\E^{(1)}_\ell
	\left(\left. M_{n+1}^{(\ell)} \right\vert \mathcal{F}_n\right)
	=&
	\frac{h^\downarrow_\ell(-\ell)}{h^\downarrow_\ell(P_\ell(n))} \nu(-P_\ell(n)-\ell) \frac{1}{h^\downarrow_\ell(-\ell)}
	\prod_{j=0}^{n}\frac{1}{\nu([-P_\ell(j)+1,\infty)) + \nu(-P_\ell(j)-\ell)}
	\\
	&+
	\sum_{k \ge 1-P_\ell(n)}
	\frac{h^\downarrow_\ell(P_\ell(n)+k)}{h^\downarrow_\ell(P_\ell(n))} \nu(k) \frac{1}{h^\downarrow_\ell(P_\ell(n)+k)}
	\prod_{j=0}^{n} \frac{1}{\nu([-P_\ell(j)+1,\infty)) + \nu(-P_\ell(j)-\ell)} \\
	=&M^{(\ell)}_n,
\end{align*}
while the indicator ${\bf 1}_{P_\ell(j)=-\ell}$ enables to keep a martingale after $P_\ell$ is absorbed at $-\ell$. This martingale is particularly well suited to the study of the distances on $3/2$-stable maps owing to the asymptotic behavior (\ref{comportement asymptotique nu}) of the step distribution $\nu$: intuitively, since $\nu((-\infty,-P_\ell(n)))\approx p_{\bf q}/P_\ell(n)$, the above product can be approximated by the exponential of $\sum_{j=0}^{n-1} 1/P_\ell(j)$ (times an appropriate constant) and can thus be related to the fpp distance (\ref{distance fpp à la racine}) via a concentration inequality.

One notable fact is that our bounds show that the diameter under the fpp distance heavily depends on the tails of $\nu$, but also on the non-asymptotic behavior of $\nu$. Hence, even though the fpp diameter and the distance between uniform points are of the same order, the ratio does not converge to a universal constant. On the contrary, in the case of the dual graph distance, our upper bound for the diameter is universal, echoing the universal scaling limit of the height of a random uniform face in Theorem \ref{distance de la racine à une face uniforme}.

\subsection{Upper bound of the diameter for the fpp distance}

In this subsection, we upper bound the diameter of $\mathfrak{M}^{(\ell,\dagger)}$ for the first-passage percolation distance. Let $\gamma({\bf q}) \coloneqq \min_{k,\ell \ge 1} k(\nu((-\infty,-k])-\nu(-\ell-k)))$, which is well defined thanks to (\ref{comportement asymptotique nu}) and belongs to $(0,1)$. The proposition below implies the first upper bound of Theorem \ref{majoration du diamètre graphe et fpp}.

\begin{proposition}\label{majoration du diamètre}
	Let $\mathrm{Diam}_\mathrm{fpp}^\dagger(\mathfrak{M}^{(\ell)})$ be the diameter of $\mathfrak{M}^{(\ell,\dagger)}$ for the fpp distance. Then there exists a constant $C>0$ which does not depend on $\bf q$ such that with probability $1-o(1)$ when $\ell \to \infty$,
	$$
	\frac{\mathrm{Diam}_\mathrm{fpp}^\dagger(\mathfrak{M}^{(\ell)})}{\log \ell}
	\le C\left(1+ \frac{1}{\gamma({\bf q})}\right).
	$$
\end{proposition}

So as to prove the above proposition, we introduce a first family of supermartingales associated with the perimeter process and related to the martingale (\ref{martingale}):

\begin{lemma}\label{sur-martingale2}
	Let $\lambda>0$. For all $\ell \ge 1, n\ge 0$, let
	$$
	M^{\lambda,\ell}_n = \frac{1}{h^\downarrow_\ell(P_\ell(n))} \exp \left({\lambda \sum_{j=0}^{n-1} \frac{1}{P_\ell(j)}{\bf 1}_{P_\ell(j)>0}} \right) .
	$$
	If $\lambda\le\gamma({\bf q})$
	, then for all $\ell \ge 1$, under $\P_\ell^{(1)}$,  $(M_n^{\lambda,\ell})_{n\ge 0}$ is a supermartingale with respect to the filtration $(\mathcal{F}_n)_{n \ge 0}$ associated with the peeling exploration.
\end{lemma}

\begin{proof}
	If $\lambda\le\gamma({\bf q})$, then for all $\ell \ge 1,n \ge 0$,
	\begin{align*}
		\E_\ell^{(1)}\left( \left. M_{n+1}^{\lambda,\ell}\right| \mathcal{F}_n\right)
		&=
		\sum_{\substack{k>-P_\ell(n) \\\text{ or } k=-P_\ell(n)-\ell}} \frac{1}{h^\downarrow_\ell(P_\ell(n)+k)}\exp \left({\lambda \sum_{j=0}^{n} \frac{1}{P_\ell(j)}{\bf 1}_{P_\ell(j)>0}} \right)
		\frac{h^\downarrow_\ell(P_\ell(n)+k)}{h^\downarrow_\ell(P_\ell(n))} \nu(k) \\
		&= M_n^{\lambda,\ell}  e^{\lambda  \frac{1}{P_\ell(n)}{\bf 1}_{P_\ell(n)> 0}}
		\left( \left(\sum_{k>-P_\ell(n)} \nu(k) \right) + \nu(-P_\ell(n)-\ell)\right)\\
		&\le M_n^{\lambda,\ell},
	\end{align*}
	thanks to the definition of $\gamma({\bf q})<1$ and using that $\exp({(\lambda /{P_\ell(n)}){\bf 1}_{P_\ell(n)> 0}}) \le 1/(1-({\lambda }/{P_\ell(n)}){\bf 1}_{P_\ell(n)> 0}) $.
\end{proof}
The above supermartingales lead to a large deviation inequality for the distance from the root to a uniform random edge.
\begin{corollary}\label{grandes déviations distance à une arête uniforme2}
	Let $c>0$. Let $E$ be a uniform  random edge of $\mathfrak{M}^{(\ell)}$. If $K> 0$ is large enough, then
	$$
	\P\left( d_\mathrm{fpp}^\dagger(f_r,E) \ge K \log \ell \right) = O(\ell^{-c}) \text{ when } \ell \to \infty.
	$$
	One can choose $K=C'(1+c)(1+ 1/\gamma({\bf q}))$, where $C'$ is large enough and does not depend on $\bf q$, $c$.
\end{corollary}
\begin{proof}
	Let $c>0$. First of all, since the distance to an edge is smaller than the distance to the $2$-face obtained by unzipping the edge, and by (\ref{loi réenraciné}),
	\begin{align*}
		\P\left(d_\mathrm{fpp}^\dagger(f_r,E) \ge K \log \ell\right)
		&\le
		\frac{W_1^{(\ell)}}{W^{(\ell)}} \E_1^{(\ell)} \left(\frac{1}{\#\mathrm{Edges}(\mathfrak{m}_\square)-1} {\bf 1}_{d_\mathrm{fpp}^\dagger(f_r,\square)\ge K \log \ell}\right)\\
		&\le 
		\frac{W_1^{(\ell)}}{W^{(\ell)}} \P_1^{(\ell)} \left(d_\mathrm{fpp}^\dagger(f_r,\square)\ge K \log \ell\right).
	\end{align*}
	But, owing to (\ref{EqVolume})
	, we know that $W_1^{(\ell)}/W^{(\ell)}$ is a $O(\ell^{3/2})$, so that it is enough to prove the analogous statement for $\P_1^{(\ell)} (d_\mathrm{fpp}^\dagger(f_r,\square)\ge K \log \ell)$. 
	Moreover, recalling (\ref{distance fpp à la racine}), one can see that it suffices to show that there exists $K'>0$ such that
	\begin{equation}\label{eqsomme2}
		\P^{(1)}_\ell\left( \sum_{k=0}^{\tau_{-\ell}-1} {1 \over P_\ell(k)}\ge K' \log \ell \right) =O(\ell^{-c}) \text{ when } \ell \to \infty.
	\end{equation}
	
	Indeed, if we assume (\ref{eqsomme2}), then by conditioning on the perimeter process and applying Bernstein's inequality (the version we use comes from Corollary 2.10 of \cite{Ta94})
	, we get that there exists a constant $C>0$ such that for all $\ell \ge 1$ and for all $t>0$, 
	
	\begin{align*}
		\P^{(1)}_\ell \left( \left. \left| \sum_{k=0}^{\tau_{-\ell}-1} \frac{\mathcal{E}_k-1}{P_\ell(k)} \right| > t \log \ell \right| P_\ell \right) 
		&\le 2\exp \left(-C\min \left(\frac{t^2 (\log \ell)^2}{\sum_{k=0}^{\tau_{-\ell}-1} \frac{1}{P_\ell(k)^2}}, \frac{t \log \ell}{\max_{0\le k \le \tau_{-\ell}-1} \frac{1}{P_\ell(k)}}\right)\right) \\
		& \le 2\exp \left(-C\min \left(\frac{t^2 (\log \ell)^2}{\sum_{k=0}^{\tau_{-\ell}-1} \frac{1}{P_\ell(k)}}, t \log \ell\right)\right) \\
		& \le 2 \ell^{-(\frac{Ct^2}{K'} \wedge Ct) } \text{ with probability at least } 1-O(\ell^{-c}).
	\end{align*}
	One can then conclude by choosing $t>0$ such that $\frac{Ct^2}{K'} \wedge Ct \ge c$ and by taking $K=\frac{K'+t}{2}$.
	
	Now, let us prove (\ref{eqsomme2}). Let $\lambda \in (0,\gamma({\bf q})]$ so that for all $\ell \ge 1$, $M^{\lambda,\ell}$ is a supermartingale (given by Lemma \ref{sur-martingale2}). Then, using Markov's inequality and Fatou's lemma,
	\begin{align*}
		\P^{(1)}_\ell\left( \sum_{k=0}^{\tau_{-\ell}-1} {1 \over P_\ell(k)}\ge K' \log \ell \right) 
		&\le\ell^{-\lambda K'} \E^{(1)}_\ell e^{\lambda \sum_{k=0}^{\tau_{-\ell}-1} \frac{1}{P_\ell(k)}} \\
		&= \ell^{-\lambda K'} \E^{(1)}_\ell M^{\lambda, \ell}_{\tau_{-\ell}}\\
		&\le \ell^{-\lambda K'} \E^{(1)}_\ell M^{\lambda, \ell}_{0} =\ell^{-\lambda K'} \frac{1}{h^\downarrow_\ell(1)}  = O(\ell^{3/2-\lambda K'}) = O(\ell^{-c})
	\end{align*}
	for $K' \ge (c+3/2)/\lambda$.
\end{proof}

\begin{proof}[Proof of Proposition \ref{majoration du diamètre}]
	To begin with, since the diameter is upper bounded by twice the maximal distance to the root face, it suffices to show that there exists $K>0$ such that
	$$
	\P\left( \exists e \in \mathrm{Edges}(\mathfrak{M}^{(\ell)}), \ d_\mathrm{fpp}^\dagger(f_r,e)
	\ge K \log \ell\right)
	\mathop{\longrightarrow}\limits_{\ell \to \infty} 0,
	$$
	using also that for all $f \in \mathrm{Faces}(\mathfrak{M}^{(\ell)})$, if $e$ is an edge surrounding $f$, then $d_\mathrm{fpp}^\dagger(f_r,f) \le d_\mathrm{fpp}^\dagger(f_r,e) +\mathcal{E}$, where $\mathcal{E}$ is an exponential random variable of parameter $1$.
	
	\noindent 
	Then we distinguish whether the number of edges is too large or not:
	let $\vp>0$.
	\begin{align*}
		\P&\left( \exists e\in \mathrm{Edges}(\mathfrak{M}^{(\ell)}), \ d_\mathrm{fpp}^\dagger(f_r,e)
		\ge K \log \ell\right) \\
		\le
		&\enskip \E \left( \#\left\{e\in \mathrm{Edges}(\mathfrak{M}^{(\ell)}); \ d_\mathrm{fpp}^\dagger(f_r,e) \ge K \log \ell \right\} {\bf 1}_{\# \mathrm{Edges}(\mathfrak{M}^{(\ell)}) \le \frac{2 b_{\bf q}}{\vp} \ell^{3/2}}\right) \\&+ \P\left(\# \mathrm{Edges}(\mathfrak{M}^{(\ell)}) \ge \frac{2 b_{\bf q}}{\vp} \ell^{3/2}\right)
	\end{align*}
	and the last term is smaller than $\vp$ for $\ell$ large enough by (\ref{EqVolume}). Henceforth we focus on the first term. Conditionally on $\mathfrak{M}^{(\ell)}$, let $E$ be a uniform random edge of $\mathfrak{M}^{(\ell)}$.
	\begin{align*}
		\E &\left( \#\left\{e\in \mathrm{Edges}(\mathfrak{M}^{(\ell)}); \ d_\mathrm{fpp}^\dagger(f_r,e) \ge K \log \ell \right\} {\bf 1}_{\# \mathrm{Edges}(\mathfrak{M}^{(\ell)}) \le \frac{2 b_{\bf q}}{\vp} \ell^{3/2}}\right) \\
		&\le \E \left(\frac{2 b_{\bf q}}{\vp} \ell^{3/2} \P( d_\mathrm{fpp}^\dagger(f_r,E) \ge K \log \ell |\mathfrak{M}^{(\ell)})  \right) \\
		&=\frac{2 b_{\bf q}}{\vp} \ell^{3/2} \P( d_\mathrm{fpp}^\dagger(f_r,E) \ge K \log \ell ).
	\end{align*}
	The expression on the last line tends to zero when $\ell \to \infty$ if $K$ is chosen large enough according to Corollary \ref{grandes déviations distance à une arête uniforme2}.
\end{proof}
\subsection{Lower bound for the diameter for the fpp distance}

A lower bound for the diameter is already provided by Corollary \ref{distance entre deux faces uniformes}. Nevertheless, this lower bound only depends on the tail behavior of $\nu$. Actually, in the case of the fpp distance, the diameter depends heavily on $\nu$. Let us show for instance the influence of $q_1$, or equivalently $\nu(0)$, on the asymptotic behavior of the diameter, which implies in particular the first lower bound of Theorem \ref{majoration du diamètre graphe et fpp}, hence ending the proof of that theorem.

\begin{proposition}\label{minoration du diamètre}
	For all $\delta >0$, with probability $1-o(1)$ as $\ell \to \infty$,
	$$
	\frac{\mathrm{diam}_\mathrm{fpp}^\dagger (\mathfrak{M}^{(\ell)})}{\log \ell}\ge \frac{3}{4(1-q_1^2)} -\delta
	.
	$$
\end{proposition}

\begin{figure}[h]
	\centering
	\hbox{
		\includegraphics[scale=1.6]{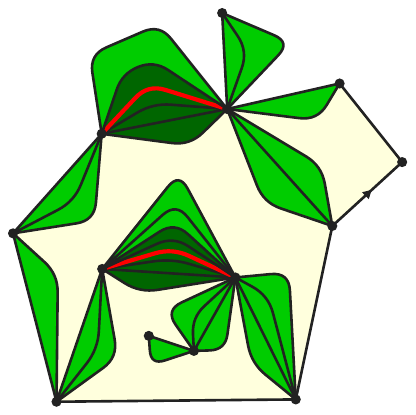}
		\includegraphics[scale=2.1]{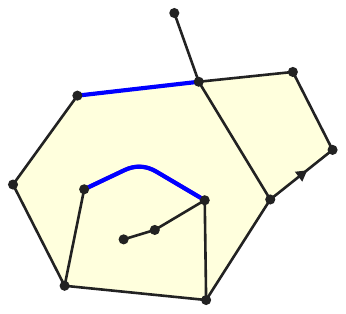}
	}
	\caption{On the left, the map $\mathfrak{M}^{(\ell)}$ with two random edges $E_1$ and $E_2$ in red. The $2$-faces are in green and form ``watermelons'' and $E_1$ and $E_2$ are in the middle of their watermelons. The largest fpp-balls around $E_1$ and $E_2$ made of $2$-faces only are in dark green. On the right the map $\widetilde{\mathfrak{M}}^{(\ell)}$ obtained from $\mathfrak{M}^{(\ell)}$ by deleting the $2$-faces. The blue edges are the two random uniform edges $\tilde{E}_1$ and $\tilde{E}_2$ that correspond to the two watermelons containing $E_1$ and $E_2$.}
	\label{dessin pastèque}
\end{figure}

\begin{proof}
	Let us prove that result with a second moment method.
	If $e$ is an edge of $\mathfrak{M}^{(\ell)}$, let $R(e)$ denote the radius of the largest fpp-ball centred at $e$ only constituted of faces of degree $2$. Notice that this fpp-ball is included in the maximal sub-map made of $2$-faces and edges which have the same endpoints as $e$, which is called the ``watermelon'' containing $e$ (see Figure \ref{dessin pastèque}).
	
	By Proposition 11.10 of \cite{StFlour}, we know that the map $\widetilde{\mathfrak{M}}^{(\ell)}$ obtained from $\mathfrak{M}^{(\ell)}$ by replacing each watermelon by one edge  (see Figure \ref{dessin pastèque}) is a $\tilde{\bf q}$-Boltzmann map of perimeter $2\ell$, where the weight sequence $\tilde{\bf q}$ is defined by 
	$$
	\left\{
	\begin{array}{lc}
		\tilde{q}_1=0 &  \\
		\tilde{q}_k=q_k/(1-q_1)^k & \text{for } k\ge 2.
	\end{array}
	\right.
	$$
	Moreover, conditionally on $\widetilde{\mathfrak{M}}^{(\ell)}$, the map ${\mathfrak{M}}^{(\ell)}$ is obtained by adding independently a geometric number of $2$-faces on each edge of $\widetilde{\mathfrak{M}}^{(\ell)}$, with parameter $q_1$. We denote by $M(\tilde{e})$ the number of edges of the watermelon corresponding to an edge $\tilde{e}$ of $\widetilde{\mathfrak{M}}^{(\ell)}$ so that conditionally on $\widetilde{\mathfrak{M}}^{(\ell)}$, the random variables $M(\tilde{e})-1$ are i.i.d. geometric of parameter $q_1$. By the paragraph just above Proposition 11.10 in \cite{StFlour} we also know that $c_{\bf q}= c_{\tilde{\bf q}}/(1-q_1)$ and that $\tilde{\bf q}$ is again a critical weight sequence of type $2$.
	
	If $\tilde{E}$ is a random uniform edge in $\widetilde{\mathfrak{M}}^{(\ell)}$ and $E$ is the edge in the middle of the watermelon corresponding to $\tilde{E}$ (and if the watermelon has an even number of edges, we choose $E$ uniformly among the two edges in the middle), then, conditionally on $\widetilde{\mathfrak{M}}^{(\ell)}$ and on $\tilde{E}$, the random variable $M(\widetilde{E})$ is geometric of parameter $q_1$ starting at one. Hence, $R(E)$ is lower-bounded by a sum of $\lfloor M(E)/2 \rfloor-1$ exponential random variables of parameter $2$ (i.e. of expectation $1/2$). As a consequence, for all $t\ge 0$,
	$$\P\left(R(E)\ge t \left| \widetilde{\mathfrak{M}}^{(\ell)}, \tilde{E} \right.\right) \ge
	\P\left(\left.\sum_{i=1}^{\lfloor M(\tilde{E})/2\rfloor-1} \frac{\mathcal{E}_i}{2} \ge t \right| \widetilde{\mathfrak{M}}^{(\ell)}, \tilde{E} \right),
	$$
	where the $\mathcal{E}_i$'s are i.i.d. exponential random variables of parameter $1$ which are independent from $\widetilde{\mathfrak{M}}^{(\ell)}$ and $M(\tilde{E})$. But one can lower-bound
	\begin{align*}
		\P\left(\left.\sum_{i=1}^{\lfloor M(\tilde{E})/2\rfloor-1} \frac{\mathcal{E}_i}{2} \ge t \right|\widetilde{\mathfrak{M}}^{(\ell)}, \tilde{E}\right)
		&=
		\sum_{k=4}^{\infty} q_1^{k-1} (1-q_1)\int_{2t}^{\infty} \frac{x^{\lfloor k/2\rfloor-2} e^{-x}}{(\lfloor k/2\rfloor-2)!} dx\\
		&\ge 
		q_1^5 (1-q_1) \int_{2t}^{\infty} \sum_{k=4}^{\infty} (q_1^{2})^{\lfloor k/2\rfloor-2} \frac{x^{\lfloor k/2\rfloor-2} e^{-x}}{(\lfloor k/2\rfloor-2)!} dx \\
		&=2q_1^5 (1-q_1) \int_{2t}^{\infty} e^{(q_1^2-1)x} dx \\
		&=\frac{2q_1^5}{1+q_1}e^{-2(1-q_1^2)t} .
	\end{align*}
	Henceforth, let $\tilde{E}_1,\ldots, \tilde{E}_{\lfloor \ell^{3/2}\rfloor}$ be independent random uniform edges of $\widetilde{\mathfrak{M}}^{(\ell)}$ (conditionally on $\widetilde{\mathfrak{M}}^{(\ell)}$). For all $1 \le i \le \lfloor \ell^{3/2} \rfloor$, we write $E_i$ the edge of $\mathfrak{M}^{(\ell)}$ in the middle of the watermelon corresponding to $\tilde{E}_i$ (chosen uniformly at random between the two edges in the middle if the watermelon has an even number of edges). We again write $R(E_i)$ the radius of the largest fpp-ball centred at $E_i$ containing only $2$-faces. 
	For conciseness we drop the $\lfloor \cdot \rfloor$. If $t>0$, then by the second moment method, 
	\begin{align*}
		\P&\left( \left. \exists  i\le \ell^{3/2} , R(E_i)\ge t\log \ell \right| \widetilde{\mathfrak{M}}^{(\ell)}\right)\\
		&\ge \frac{\E (\#\left\{i \in \lb 1, \ell^{3/2} \rb, R(E_i) \ge t \log \ell\right\}|\widetilde{\mathfrak{M}}^{(\ell)})^2 }{\E((\#\left\{i \in \lb 1, \ell^{3/2} \rb, R(E_i) \ge t \log \ell\right\})^2|\widetilde{\mathfrak{M}}^{(\ell)})} \\
		&=
		\frac{\ell^{3} \P(R(E) \ge t \log \ell|\widetilde{\mathfrak{M}}^{(\ell)})^2}{\ell^{3/2} \P(R(E) \ge t \log \ell|\widetilde{\mathfrak{M}}^{(\ell)}) + \ell^{3/2}(\ell^{3/2}-1)\P(R(E_1),R(E_2) \ge t \log \ell|\widetilde{\mathfrak{M}}^{(\ell)})}. 
	\end{align*}
	Now, one can distinguish whether $\tilde{E}_1=\tilde{E}_2$ or not. Then, in probability as $\ell \to \infty$,
	\begin{align*}
		\P&(R(E_1),R(E_2) \ge t \log \ell|\widetilde{\mathfrak{M}}^{(\ell)}) \\
		&\le \P( \tilde{E}_1 =\tilde{E}_2  \text{ and }R(E_1) \ge t \log \ell |\widetilde{\mathfrak{M}}^{(\ell)})
		+\P( \tilde{E}_1 \neq \tilde{E}_2 \text{ and }R(E_1),R(E_2) \ge t \log \ell|\widetilde{\mathfrak{M}}^{(\ell)})\\
		&= \frac{1}{\# \mathrm{Edges}(\widetilde{\mathfrak{M}}^{(\ell)})} \P(R(E) \ge t \log \ell|\widetilde{\mathfrak{M}}^{(\ell)})
		+ \frac{\# \mathrm{Edges}(\widetilde{\mathfrak{M}}^{(\ell)})-1}{\# \mathrm{Edges}(\widetilde{\mathfrak{M}}^{(\ell)})}
		\P(R(E_1),R(E_2) \ge t \log \ell|\tilde{E}_1 \neq \tilde{E}_2, \widetilde{\mathfrak{M}}^{(\ell)}) \\
		&=O(\ell^{-3/2})\P(R(E) \ge t \log \ell|\widetilde{\mathfrak{M}}^{(\ell)})+\P(R(E) \ge t \log \ell|\widetilde{\mathfrak{M}}^{(\ell)})^2,
	\end{align*}
	where the last line is due to the convergence in distribution of $\# \mathrm{Edges}(\widetilde{\mathfrak{M}}^{(\ell)}) \ell^{-3/2}$ as $\ell \to \infty$, see e.g. Proposition 10.4 in \cite{StFlour}. 
	By taking $t<3/(4(1-q_1^2))$ we ensure that 
	$
	\ell^{3/2}\P(R(E) \ge t \log \ell|\widetilde{\mathfrak{M}}^{(\ell)}) \to
	\infty
	$ 
	in probability as $\ell \to \infty$, 
	hence
	$$
	\P\left( \left. \exists  i\le \ell^{3/2} , R(E_i)\ge t\log \ell \right| \widetilde{\mathfrak{M}}^{(\ell)}\right)
	\mathop{\longrightarrow}\limits_{\ell \to  \infty}^{(\P)}
	1.
	$$
	This entails what we wanted to show.
\end{proof}
\subsection{Upper bound of the diameter for the graph distance}

We next upper bound the diameter of $\mathfrak{M}^{(\ell,\dagger)}$ for the graph distance. The proposition below implies the second upper bound of Theorem \ref{majoration du diamètre graphe et fpp}. Unlike the case of the fpp distance, the upper bound for the graph distance is universal (i.e. it does not depend on $\bf q$). 

\begin{proposition}\label{majoration du diamètre graph}
	Let $\mathrm{Diam}_\mathrm{gr}^\dagger(\mathfrak{M}^{(\ell)})$ be the diameter of $\mathfrak{M}^{(\ell,\dagger)}$ for the graph distance. Then for all $\delta >0$, with probability $1-o(1)$ when $\ell \to \infty$,
	$$
	\frac{\mathrm{Diam}_\mathrm{gr}^\dagger(\mathfrak{M}^{(\ell)})}{(\log \ell)^2}
	\le 18+\delta.
	$$
\end{proposition}
To prove this result, we will need to use the height process $H_\ell$ during the peeling by layers exploration, together with the interpolated height process.  
Let us first recall the definition of the interpolated height process from \cite{BCM}. At time $n$ of the exploration, let $D_n$ be the number of edges in the boundary $\partial \overline{\mathfrak{e}}_n$ that are at height $H_\ell(n)$, so that the other $2P_\ell(n)-D_n$ edges are at height $H_\ell(n)+1$. If $f:[0,1]\to [0,1]$ is a non-increasing function, the interpolated height process $H^f_\ell (n)$ under $\P^{(1)}_\ell$ and its increment at time $n<\tau_{-\ell}$ are defined by 
$$
H_\ell^f(n)\coloneqq H_\ell(n)+f(D_n/2P_n) \qquad
\text{and}
\qquad
\Delta H_\ell^f(n) \coloneqq H_\ell^f(n+1)-H_\ell^f(n),
$$
where by convention, we set $H^f_\ell(n)  = H^f_\ell(\tau_{-\ell}-1)$ for $n\ge \tau_{-\ell}$. The next lemma is an analogue of Lemma 6 from \cite{BCM}, bounding from above the mean increase of the interpolated height.
\begin{lemma}\label{analogue6}
	Let $0<\vp<1$. If $f:[0,1]\to [0,1]$ is twice continuously differentiable with $f(0)=1, f(1)=0, f'(0)=f'(1)=f''(0)=f''(1)=0$ and $0\le -f'(x) \le 1+\vp$ for all $x\in [0,1]$, then there exist $C(\vp,{\bf q}),C'(\vp,{\bf q})>0$ such that for all $\ell_0\ge 1,\ell \ge 1,n\ge 0$,
	$$
	\E^{(\ell_0)}_\ell \left( \left. \Delta H_\ell^f(n) \right| \mathcal{F}_n \right) \le  \frac{(1+3\vp)^2}{1-\vp}\frac{p_{\bf q} }{2} \frac{ \log (P_\ell(n)) +C(\vp,{\bf q})}{ P_\ell(n)}
	\enskip \text{and} \enskip
	\E^{(\ell_0)}_\ell \left( \left. \left(\Delta H_\ell^f(n)\right)^2 \right| \mathcal{F}_n \right) \le  \frac{C'(\vp,{\bf q})}{P_\ell(n)}
	,
	$$
	where $(\mathcal{F}_n)_{n\ge 0}$ is the filtration associated with the exploration process.
\end{lemma}

\begin{proof}
	The proof is essentially the same as the proof of Lemma 6 of \cite{BCM}, but a little more complex. Indeed, the same terms appear except that they are multiplied by a factor $(p+\ell)/(p+k+\ell)$. 
	Let us detail the proof. 
	
	We extend $f$ to $\R$ by setting $f(x)=1$ for $x\le 0$ and $f(x)=0$ for all $x \ge 1$. As in \cite{BCM}, by the Markov property of the exploration (Proposition \ref{Markov explo}), we claim that $(H_\ell(n),P_\ell(n),D_n)_{n\ge 0}$ is a Markov chain and that if $p\ge 1$ and $1\le d \le 2p$, then
	\begin{align}
		\E^{(\ell_0)}_\ell &\left( \left. \Delta H_\ell^f(n) \right| P_\ell(n)= p, D_n=d \right) \\
		= &\sum_{k=0}^{\infty} \nu(k) \frac{p+\ell}{p+k+\ell} \frac{h^\uparrow(p+k)}{h^\uparrow(p)} \left[ f\left(\frac{d-1}{2p+2k}\right)-f\left(\frac{d}{2p}\right)\right] \label{première somme}\\
		&+ \sum_{k=1}^{p-1} \frac{\nu(-k)}{2} \frac{p+\ell}{p-k+\ell} \frac{h^\uparrow(p-k)}{h^\uparrow(p)} \left[ f\left(\frac{d-2k}{2p-2k}\right)-f\left(\frac{d}{2p}\right)\right] \label{deuxième somme} \\
		&+ \sum_{k=1}^{p-1} \frac{\nu(-k)}{2} \frac{p+\ell}{p-k+\ell} \frac{h^\uparrow(p-k)}{h^\uparrow(p)} \left[ f\left(\frac{d-1}{2p-2k}\right)-f\left(\frac{d}{2p}\right)\right]. \label{troisième somme}
	\end{align}
	Indeed, if $d=1$, then either $\Delta H_\ell(n)=1$ and $D_{n+1}=2P_\ell(n+1)$, and therefore $\Delta H_\ell^f(n)=1-f(1/(2p))$, with probability $\sum_{k=-p+1}^{\infty} \nu(k) (h^\downarrow_\ell(p+k)/h^\downarrow_\ell(p))$; or $\tau_{-\ell}=n+1$ and in that case $\Delta H_\ell^f(n)=0$, with probability $\nu(-p-\ell)/{h^\downarrow_\ell(p)}$. This agrees with the sums since $f(x)=1$ for $x\le 0$ (and by using the definition of $h^\uparrow$ and $h^\downarrow_\ell$ in (\ref{h flèche})).
	Now, in the case $d\ge 2$, if an event $C_{k+1}$ happens, then $P_\ell(n+1)=p+k$ and $D_{n+1}=d-1$, hence the first line (\ref{première somme}) (again using (\ref{h flèche})). 
	If an event $G_{\square,k-1}$ happens for $1\le k \le p-1$, i.e. if the peeled edge is identified with an edge on its right, then $P_\ell(n+1)=P_\ell(n)-k$ and either $D_{n+1}=d-2k$ and $\Delta H_\ell(n)=0$, or $D_{n+1}=2P_\ell(n+1)$ and $\Delta H_\ell(n)=1$, depending on the sign of $d-2k$. These two cases are taken into account in the second line (\ref{deuxième somme}) since $f(x)=1$ for all $x\le 0$.
	If an event $G_{k-1,\square}$ happens for some $1\le k \le p-1$, i.e. if the peeled edge is identified with an edge on its left, then again $P_\ell(n+1)=P_\ell(n)-k$ but $\Delta H_\ell(n)=0$ and either $D_{n+1}=d-1$ or $D_{n+1}=2P_\ell(n+1)$, depending on which is the smallest. Both situations are incorporated by the last line (\ref{troisième somme}).
	Finally, if the event $C_\ell^\mathrm{stop}$ happens, then $\Delta H_\ell^f(n)=0$ by definition of the interpolated height process.
	
	To deal with these three sums, we will use the notation $y=d/(2p)\in (0,1]$. Let $c>0$ given by Taylor's theorem such that $|f(x)-f(y)-f'(y)(x-y)|\le c(x-y)^2$ for all $x \in \R$. We will also rely on the inequality
	$$
	\forall k>-p, \qquad \sqrt{1+k/p}(1-1/p)\le \frac{h^\uparrow(p+k)}{h^\uparrow(p)}
	\le \sqrt{1+k/p}(1+1/p).
	$$
	Let $N>0$ such that $1-\vp< k^2\nu(-k)/p_{\bf q}< 1+\vp$ and $1-\vp <k\nu([k,\infty))/p_{\bf q} <1+\vp$ for all $k\ge N$, given by (\ref{comportement asymptotique nu}).
	From the mean value inequality applied to $f$, the first $N$ terms in each sum contribute $O(p^{-1})$ uniformly in $d,\ell$, so it suffices to bound the sums restricted to $|k|\ge N$.
	Since $(p+\ell)/(p+k+\ell)\le 1$, the aforementioned inequalities, the first sum from $N$ to $\infty$, which only contains positive terms, is smaller than
	$$
	\sum_{k=N}^{\infty} \nu(k) \sqrt{1+k/p}(1+1/p) \left[|f'(y)|\frac{1+yk}{p+k} + c \left(\frac{1+k}{p+k}\right)^2 \right]
	$$
	where we used that
	$$
	\frac{d}{2p}-\frac{d-1}{2p+2k} \le \frac{1+yk}{p+k}\le \frac{1+k}{p+k}.
	$$
	Thus, by summation by parts, using that $1-\vp <k\nu([k,\infty))/p_{\bf q} <1+\vp$ for all $k\ge N$, and then recognizing a Riemann sum, the first sum (\ref{première somme}) is upper bounded by 
	\begin{equation}\label{maj1}
		(1+\vp)p_{\bf q} y \frac{1}{p} \int_{(N-1)/p}^\infty \frac{dx}{x\sqrt{1+x}} + O(p^{-1})\le (1+\vp)^2 p_{\bf q} y  \frac{\log p}{p} + O(p^{-1}),
	\end{equation} 
	where the $O(p^{-1})$ is uniform in $d$ and $\ell$. 
	For the second sum (\ref{deuxième somme}), which again has only positive terms, we distinguish whether $k\le \vp p$ or $k > \vp  p$. If $k \le \vp p$, $(p+\ell)/(p-k+\ell) \le 1/(1-\vp)$, so that applying the inequalities and the fact that $1-\vp< k^2\nu(-k)/p_{\bf q}< 1+\vp$ for all $k\ge N$ gives us that
	\begin{align*}
		\sum_{k=N}^{\vp p} \frac{\nu(-k)}{2} \frac{p+\ell}{p-k+\ell} &\frac{h^\uparrow(p-k)}{h^\uparrow(p)} \left[ f\left(\frac{d-2k}{2p-2k}\right)-f\left(\frac{d}{2p}\right)\right] \\
		\le
		&\sum_{k=N}^{\vp p} \frac{1+\vp}{1-\vp} \frac{p_{\bf q}}{2k^2} (1+1/p)\sqrt{1-k/p} 
		\left[|f'(y)| (1-y)\frac{k}{p} \frac{1}{1-\frac{k}{p}}+ c\left(\frac{k}{p} \frac{1}{1-\frac{k}{p}}\right)^2\right],
	\end{align*}
	where we used that
	$$
	\frac{d}{2p}-\frac{d-2k}{2p-2k} =(1-y)\frac{k}{p}\frac{1}{1-\frac{k}{p}} \le \frac{k}{p}\frac{1}{1-\frac{k}{p}}.
	$$
	Hence, by recognizing a Riemann sum, we can upper bound the sum for the $k\le \vp p$ by 
	$$\frac{1+\vp}{1-\vp}p_{\bf q} \frac{1-y}{2}|f'(y)| \frac{1}{p}\int_{(N-1)/p}^{\vp+1/p} \frac{1}{x \sqrt{1-x}}dx +O(p^{-1}) \le
	\frac{1+\vp}{1-\vp}p_{\bf q} \frac{1-y}{2}|f'(y)|\frac{\log p}{p} + O(p^{-1}).$$
	If $\vp p \le k \le p-1$, we have
	$$
	1\le\frac{p+\ell}{p-k+\ell} \le \frac{p}{p-k}
	$$
	and
	$$
	\frac{\nu(-k)}{2} \frac{h^\uparrow(p-k)}{h^\uparrow(p)} \left[ f\left(\frac{d-2k}{2p-2k}\right)-f\left(\frac{d}{2p}\right)\right] \le \nu(-k) \frac{h^\uparrow(p-k)}{h^\uparrow(p)}\le \sqrt{(p-k)/p} O(p^{-2}).
	$$
	So, using a Riemann sum, the sum for $k\ge \vp p$ is upper bounded by $O(p^{-1})$. Therefore, the second sum (\ref{deuxième somme}) is bounded by 
	\begin{equation}\label{maj2}
		\frac{1+\vp}{1-\vp}p_{\bf q} \frac{1-y}{2} |f'(y)|\frac{\log p}{p} + O(p^{-1}).
	\end{equation}
	For the last sum (\ref{troisième somme}), we first focus on the positive terms, i.e. those such that $kd \le p$. If $d \ge 2$, then this entails that $(p+\ell)/(p-k+\ell) \le 2$. Moreover, for all $1\le k \le p-1$,
	$$
	\frac{d}{2p}- \frac{d-1}{2p-2k} \le \frac{1}{2p},
	$$
	so that the sum of the positive terms is a $O(p^{-1})$. If $d=1$, then by Taylor's theorem $f(0)-f(1/(2p)) = O(p^{-2})$, so that the sum of the positive terms is a $O(p^{-2})$. In these two cases, the sum of the positive terms of (\ref{troisième somme}) is upper bounded by a $O(p^{-1})$. 
	In the case $0<y<\vp$, the upper bounds for the positive terms already enable to conclude since by summing (\ref{maj1}), (\ref{maj2}) and the $O(p^{-1})$, we obtain the uniform bound
	$$
	\frac{(1+\vp)^3}{1-\vp} \frac{p_{\bf q}}{2} \frac{\log p}{p} + O(p^{-1}) \le \frac{(1+3\vp)^2}{1-\vp}\frac{p_{\bf q}}{2} \frac{\log p}{p} + O(p^{-1}).
	$$
	In the case $y\ge \vp$, we assume that we have chosen $N>1/\vp$, so that $N\ge 2p/d$ and thus all the terms of (\ref{troisième somme}) for $N\le k\le \vp p$ are negative. Using that $(p+\ell)/(p-k+\ell)\ge1$, the sum of the negative terms of (\ref{troisième somme}) with $N\le k\le \vp p$ can be upper bounded by
	$$
	-\frac{1-\vp}{2} \sum_{k=N}^{\vp p} \frac{p_{\bf q}}{k^2} (1-1/p)\sqrt{1-k/p} \left[|f'(y)| \left( y \frac{k}{p} -\frac{1}{2p}\right)-c(2k/p)^2\right],
	$$
	where we used that
	$$
	\frac{d-1}{2p-2k} -\frac{d}{2p} \ge \frac{dk-p}{2p^2} =y\frac{k}{p} -\frac{1}{2p} \qquad \text{and} \qquad \left| \frac{d}{2p}- \frac{d-1}{2p-2k} \right| \le \frac{k}{(1-\vp)p}.
	$$
	Therefore, using again a Riemann sum, the sum (\ref{troisième somme}) is upper bounded by
	\begin{equation}\label{maj3}
		-(1-\vp)p_{\bf q} \frac{y}{2} |f'(y)| \frac{\log p}{p} + O(p^{-1}).
	\end{equation}
	Finally, summing the upper bounds (\ref{maj1}), (\ref{maj2}), (\ref{maj3}) and the $O(p^{-1})$, one recovers the desired result in the case $\vp \le y \le 1$ as well for $\E^{(\ell_0)}_\ell \left( \left. \Delta H_\ell^f(n) \right| \mathcal{F}_n \right)$. For the expectation of the square, one can perform the same computation with the three sums (\ref{première somme}), (\ref{deuxième somme}), (\ref{troisième somme}), except that the expressions into brackets are squared. One can check using the same ideas that each sum is a $O(p^{-1})$ (uniformly in $d,\ell$), hence the second upper bound.
\end{proof}
We will rely on a direct consequence of Lemma \ref{analogue6}.
\begin{corollary}\label{analogue6bis}
	For all $\vp \in (0,1)$, there exists $\lambda(\vp)>0$ (which does not depend on $\bf q$) and $C(\vp,{\bf q})>0$ such that for all $0<\lambda<\lambda(\vp)$, for all $\ell\ge 1$ and $n\ge 0$, 
	$$
	\E^{(1)}_\ell \left( \left.e^{\lambda \Delta H_\ell^f(n)} \right| \mathcal{F}_n\right) \le 1+
	\lambda \left(\frac{(1+3\vp)^2}{1-\vp} \frac{p_{\bf q}}{2}  { \log (P_\ell(n))  \over P_\ell(n)}+\frac{C(\vp,{\bf q})}{P_\ell(n)}\right),
	$$
	where $(\mathcal{F}_n)_{n\ge 0}$ is the filtration associated with the exploration (using the peeling by layers algorithm).
\end{corollary}
\begin{proof}
	It suffices to take $\lambda(\vp)$ small enough so that for all $0<\lambda<c(\vp)$, we have the inequality $e^{\lambda \Delta H_\ell^f(n)} \le 1+ \lambda\Delta H_\ell^f(n)+ \lambda^2 (\Delta H_\ell^f(n))^2$ (such a $\lambda(\vp)$ does exist given that $-1 \le \Delta H_\ell^f(n) \le 1$). 
\end{proof}
We can therefore define another family of supermartingales.

\begin{corollary}\label{sur-martingale3}
	Let $\vp \in (0,1)$. Let $\lambda<\lambda(\vp)$ where $\lambda(\vp)$ is given by Corollary \ref{analogue6bis}. For all $\ell \ge 1$ for all $n\ge 0$, let
	$$
	M^{\vp,\lambda,\ell}_n=  \exp \left({\lambda H_\ell^f(n) - \lambda \sum_{k=0}^{n-1} \left(\frac{(1+3\vp)^3}{1-\vp} \frac{p_{\bf q}}{2}  { \log (P_\ell(k))  \over P_\ell(k)}+\frac{C(\vp,{\bf q})}{P_\ell(k)}\right)} \right).
	$$
	Then, under $\P^{(1)}_\ell$, $\left(M^{\vp,\lambda,\ell}_n\right)_{n \ge 0}$ is a supermartingale with respect to the filtration $(\mathcal{F}_n)_{n\ge 0}$ associated with the exploration.
\end{corollary}

We will also use a variant of Lemma \ref{sur-martingale2}, whose proof is similar and left to the reader. 
\begin{lemma}\label{martingale tronquée}
	If $\vp >0$, let $k_0$ such that $e^{(p_{\bf q}-\vp)/k} \le 1+(\nu((-\infty,-k])-\nu(-\ell-k))$ for all $k\ge k_0$ and $\ell\ge 1$. For all $\ell\ge 1, n\ge 0$, let
	$$
	M^{\vp,k_0,\ell}_n = \frac{1}{h^\downarrow_\ell(P_\ell(n))} \exp \left({(p_{\bf q}-\vp) \sum_{j=0}^{n-1} \frac{1}{P_\ell(j)}{\bf 1}_{P_\ell(j)\ge k_0}} \right) .
	$$
	then for all $\ell \ge 1$, under $\P_\ell^{(1)}$,  $(M^{\vp,k_0,\ell}_n)_{n\ge 0}$ is a supermartingale with respect to the filtration $(\mathcal{F}_n)_{n \ge 0}$ associated with the peeling exploration.
\end{lemma}

Subsequently, from the above results, we obtain a large deviation inequality for the height of a random uniform edge in $\mathfrak{M}^{(\ell)}$.

\begin{corollary}\label{grandes déviations distance à une arête uniforme dgr}
	Let $c>0$. Let $\delta >0$ and
	$$
	K=
	\frac{1}{2}(c+3)\left(c+\frac
	{5}{2}\right)
	+\delta.
	$$
	Let $\mathfrak{M}^{(\ell)}$ of law $\P^{(\ell)}$, let $E$ be a uniform random edge of $\mathfrak{M}^{(\ell)}$. Then,
	$$
	\P\left( d_\mathrm{gr}^\dagger(f_r,E) \ge K (\log \ell)^2 \right) = O(\ell^{-c}) \text{ when } \ell \to \infty.
	$$
\end{corollary}

\begin{proof}
	First of all, since the distance to an edge is smaller than the distance to the $2$-face obtained by unzipping the edge, and by (\ref{loi réenraciné}),
	\begin{align*}
		\P\left(d_\mathrm{gr}^\dagger(f_r,E) \ge K (\log \ell)^2\right)
		&\le
		\frac{W_1^{(\ell)}}{W^{(\ell)}} \E_1^{(\ell)} \left(\frac{1}{\#\mathrm{Edges}(\mathfrak{m}_\square)-1} {\bf 1}_{d_\mathrm{gr}^\dagger(f_r,\square)\ge K (\log \ell)^2}\right)\\
		&\le 
		\frac{W_1^{(\ell)}}{W^{(\ell)}} \P_1^{(\ell)} \left(d_\mathrm{gr}^\dagger(f_r,\square)\ge K (\log \ell)^2\right).
	\end{align*}
	But, since $W_1^{(\ell)}/W^{(\ell)}$ is a $O(\ell^{3/2})$ by (\ref{EqVolume}), it is enough to prove that if $K'=\frac{1}{2}\left(c+\frac{3}{2}\right)(c+1)+\delta$, then
	$$
	\P_1^{(\ell)} (d_\mathrm{gr}^\dagger(f_r,\square)\ge K' (\log \ell)^2) = O(\ell^{-c}).
	$$
	Since $d^\dagger_\mathrm{gr}(f_r,\square)\le H_\ell(\tau_{-\ell}-1)+1$ and $H_\ell(n) \le H_\ell^f(n)$ for all $n \le \tau_{-\ell}-1$, by taking $\delta$ a bit smaller it suffices to show that
	$$
	\P^{(1)}_\ell\left( H_\ell^f(\tau_{-\ell}-1) \ge K' (\log \ell)^2 \right) = O(\ell^{-c}) \text{ when } \ell \to \infty.
	$$
	So as to prove that, we can first see from Corollary \ref{sur-martingale3} that if $\vp \in (0,1)$, if $0<\lambda<\lambda(\vp)$
	and $\mu = \frac{(1+3\vp)^3}{1-\vp} \frac{p_{\bf q}}{2}  $ then for all $\ell \ge 1$, 
	\begin{align*}
		\E^{(1)}_\ell
		\exp \left({\lambda H_\ell^f(\tau_{-\ell}-1) - \lambda \mu \sum_{k=0}^{\tau_{-\ell}-2}   { \log (P_\ell(k))  \over P_\ell(k)}-\lambda C(\vp,{\bf q})\sum_{k=0}^{\tau_{-\ell}-2}\frac{1}{P_\ell(k)}} \right)
		&=
		\E^{(1)}_\ell M^{\vp,\lambda,\ell}_{\tau_{-\ell}-1} \\
		&\le \E^{(1)}_\ell M_0^{\vp, \lambda, \ell} \\
		&= 1.
	\end{align*}
	As a consequence, if $A>0$ is such that $\lambda A >c$, we have
	\begin{align*}
		\P^{(1)}_\ell &\left(H_\ell^f(\tau_{-\ell}) \ge \mu \sum_{k=0}^{\tau_{-\ell}-1}   { \log (P_\ell(k))  \over P_\ell(k)}+ C(\vp,{\bf q}) \sum_{k=0}^{\tau_{-\ell}-1} \frac{1}{P_\ell(k)}  + A \log \ell  \right) \\
		&\le 
		\P^{(1)}_\ell \left(\exp \left({\lambda H_\ell^f(\tau_{-\ell}) - \lambda \mu \sum_{k=0}^{\tau_{-\ell}-1} \left(  { \log (P_\ell(k))  \over P_\ell(k)}+\frac{1}{P_\ell(k)} \right)}\right)\ge   \ell^{A \lambda} \right)  \le \frac{1}{\ell^{A \lambda}} = O(\ell^{-c}).
	\end{align*}
	From (\ref{eqsomme2}), $\sum_{k=0}^{\tau_{-\ell}-1} \frac{1}{P_\ell(k)}$ is well controlled by a constant times $\log \ell$. Thus it is sufficient to show that
	\begin{equation}\label{eqsomme6}
		\P^{(1)}_\ell\left(\sum_{k=0}^{\tau_{-\ell}-1}  { \log (P_\ell(k))  \over P_\ell(k)} {\bf 1}_{P_\ell(k)\ge k_0} \ge \frac{K'}{\mu}(\log \ell)^2\right)  = O(\ell^{-c}),
	\end{equation}
	where $k_0>0$ will be chosen later in the proof.
	However, one can note that since $x \mapsto x/ \log x $ is increasing on $[e,\infty)$, we have
	$$
	\sum_{k=3}^{\tau_{-\ell}-1}  { \log (P_\ell(k))  \over P_\ell(k)} {\bf 1}_{ P_\ell(k)\ge k \log k}
	\le \sum_{k=3}^{\tau_{-\ell}-1} \frac{\log k+ \log \log k}{k \log k} \le 2 \log \tau_{-\ell}
	$$
	Moreover, using Lemma \ref{couplage marche} and then that
	$\E^{(1)}_\infty ({\ell}/{P_\infty(\ell)})=O(1)$ by Lemma 4 of \cite{BCM} as $\ell \to \infty$, 
	\begin{equation}\label{temps de mort}
		\P_\ell^{(1)}( \tau_{-\ell} > \ell^{c+1}) = \E^{(1)}_\infty \frac{1+\ell}{P_\infty(\lfloor\ell^{1+c}\rfloor)+\ell} \le
		2\E_\infty^{(1)} \frac{\ell}{P_\infty(\lfloor\ell^{c+1}\rfloor)} = O(\ell^{-c}).
	\end{equation}
	Consequently,
	$$
	\P^{(1)}_\ell\left(\sum_{k=0}^{\tau_{-\ell}-1}  { \log (P_\ell(k))  \over P_\ell(k)} {\bf 1}_{P_\ell(k)\ge k \log k} \ge 2(c+1)\log \ell \right)  = O(\ell^{-c}).
	$$
	Hence, so as to prove (\ref{eqsomme6}), it suffices to show that
	$$
	\P^{(1)}_\ell\left(\sum_{k=0}^{\tau_{-\ell}-1}  { \log (P_\ell(k))  \over P_\ell(k)} {\bf 1}_{k \log k \ge P_\ell(k)\ge k_0} \ge \frac{K'}{\mu}(\log \ell)^2\right)  = O(\ell^{-c}).
	$$
	But one can upper bound
	\begin{align*}
		\sum_{k=0}^{\tau_{-\ell}-1}  { \log (P_\ell(k))  \over P_\ell(k)} {\bf 1}_{k \log k \ge P_\ell(k)\ge k_0}
		&\le 
		\sum_{k=0}^{\tau_{-\ell}-1}  { \log k+ \log \log k  \over P_\ell(k)} {\bf 1}_{k \log k \ge P_\ell(k)\ge k_0}\\
		&\le \left(\log \tau_{-\ell} + \log \log \tau_{-\ell}\right) 
		\sum_{k=0}^{\tau_{-\ell}-1}  { 1 \over P_\ell(k)} {\bf 1}_{k \log k \ge P_\ell(k)\ge k_0}.
	\end{align*}
	So, using again (\ref{temps de mort}) and then taking $\vp$ small enough so that $\mu = \frac{(1+3\vp)^3}{1-\vp} \frac{p_{\bf q}}{2} $ is close to $\frac{p_{\bf q}}{2}$, we just need to prove that for all $p_{\bf q} >\delta >0$,
	\begin{equation}\label{eqsomme7}
		\P^{(1)}_\ell\left(\sum_{k=0}^{\tau_{-\ell}-1}  { 1  \over P_\ell(k)} {\bf 1}_{P_\ell(k)\ge k_0} \ge \frac{c+\frac{3}{2}}{p_{\bf q}-\delta}(\log \ell)^2\right)  = O(\ell^{-c}).
	\end{equation}
	In order to show that we will use the supermartingale $M^{\delta,k_0, \ell}$ of Lemma \ref{martingale tronquée} for $k_0$ chosen large enough.
	Applying Markov's inequality and Fatou's lemma, we get that for $K'' = (c+3/2)/(p_{\bf q} - \delta)$,
	\begin{align*}
		\P^{(1)}_\ell\left( \sum_{k=0}^{\tau_{-\ell}-1} {1 \over P_\ell(k)} {\bf 1}_{P_\ell(k)\ge k_0}\ge K'' \log \ell \right) 
		&\le\ell^{-(p_{\bf q} - \delta) K''} \E^{(1)}_\ell \exp\left({(p_{\bf q} - \delta) \sum_{k=0}^{\tau_{-\ell}-1} \frac{1}{P_\ell(k)} {\bf 1}_{P_\ell(k)\ge k_0}}\right) \\
		&= \ell^{-(p_{\bf q} - \delta) K''} \E^{(1)}_\ell M^{\delta,k_0, \ell}_{\tau_{-\ell}}\\
		&\le \ell^{-(p_{\bf q} - \delta) K''} \E^{(1)}_\ell M^{\delta,k_0, \ell}_{0} \\
		&=\ell^{-(p_{\bf q} - \delta) K''} \frac{1}{h^\downarrow_\ell(1)}  \\
		&= O(\ell^{3/2-(p_{\bf q} - \delta) K''}) = O(\ell^{-c}),
	\end{align*}
	hence (\ref{eqsomme7}).
\end{proof}

The proof of Proposition \ref{majoration du diamètre graph} is then exactly the same as the proof of Proposition \ref{majoration du diamètre} using Corollary \ref{grandes déviations distance à une arête uniforme dgr}:

\begin{proof}[Proof of Proposition \ref{majoration du diamètre graph}]
	Since the diameter is upper bounded by twice the maximal height, it suffices to show that for all $\delta>0$
	$$
	\P\left( \exists e \in \mathrm{Edges}(\mathfrak{M}^{(\ell)}), \ d_\mathrm{gr}^\dagger(f_r,e)
	\ge (9+\delta) (\log \ell)^2\right)
	\mathop{\longrightarrow}\limits_{\ell \to \infty} 0,
	$$
	inasmuch as for all $f \in \mathrm{Faces}(\mathfrak{M}^{(\ell)})$, if $e$ is an edge surrounding $f$, then $d_\mathrm{gr}^\dagger(f_r,f) \le d_\mathrm{gr}^\dagger(f_r,e) +1$. 
	Next, we distinguish whether the number of edges is too large or not: let $\vp>0$.
	\begin{align*}
		\P&\left( \exists e\in \mathrm{Edges}(\mathfrak{M}^{(\ell)}), \ d_\mathrm{gr}^\dagger(f_r,e)
		\ge (9+\delta) (\log \ell)^2\right) \\
		\le
		&\enskip \E \left( \#\left\{e\in \mathrm{Edges}(\mathfrak{M}^{(\ell)}); \ d_\mathrm{gr}^\dagger(f_r,e) \ge (9+\delta) (\log \ell)^2 \right\} {\bf 1}_{\# \mathrm{Edges}(\mathfrak{M}^{(\ell)}) \le \frac{2 b_{\bf q}}{\vp} \ell^{3/2}}\right) \\
		&+ \P\left(\# \mathrm{Edges}(\mathfrak{M}^{(\ell)}) \ge \frac{2 b_{\bf q}}{\vp} \ell^{3/2}\right),
	\end{align*}
	and the last term is smaller than $\vp$ for $\ell$ large enough by (\ref{EqVolume}). Henceforth we focus on the first term. Conditionally on $\mathfrak{M}^{(\ell)}$, let $E$ be a uniform edge of $\mathfrak{M}^{(\ell)}$.
	\begin{align*}
		\E &\left( \#\left\{e\in \mathrm{Edges}(\mathfrak{M}^{(\ell)}); \ d_\mathrm{gr}^\dagger(f_r,e) \ge (9+\delta) (\log \ell)^2 \right\} {\bf 1}_{\# \mathrm{Edges}(\mathfrak{M}^{(\ell)}) \le \frac{2 b_{\bf q}}{\vp} \ell^{3/2}}\right) \\
		&\le\E \left(\frac{2 b_{\bf q}}{\vp} \ell^{3/2} \P( d_\mathrm{gr}^\dagger(f_r,E) \ge (9+\delta) (\log \ell)^2 |\mathfrak{M}^{(\ell)})  \right) \\
		&=\frac{2 b_{\bf q}}{\vp} \ell^{3/2}  \P( d_\mathrm{gr}^\dagger(f_r,E) \ge (9+\delta) (\log \ell)^2 ).
	\end{align*}
	The expression on the last line tends to zero by Corollary \ref{grandes déviations distance à une arête uniforme dgr} when $\ell \to \infty$ since for $c=3/2$ we have $\left(c+3\right)\left(c+5/2\right)/2 =9$.
\end{proof}
\paragraph{Acknowledgements}
I thank Cyril Marzouk and Nicolas Curien for their constant support and for their insightful comments on earlier versions. I also thank the referee for a careful reading.
\bibliographystyle{abbrv}
\bibliography{Biblio}

\end{document}